\documentclass[reqno]{amsart}

\usepackage{amsthm,amsfonts,amssymb,euscript}

\newcommand{\R}{\ensuremath{\mathbb{R}}}
\newcommand{\Z}{\ensuremath{\mathbb{Z}}}

\newcommand{\ep}{\ensuremath{\varepsilon}}

\newcommand{\al}{\ensuremath{\alpha}}

\newcommand{\de}{\ensuremath{\delta}}

\newcommand{\F}{\ensuremath{\mathbf{F}}}
\newcommand{\N}{\ensuremath{\mathbf{N}}}

\newtheorem{theorem}{Theorem}[section]
\newtheorem{lemma}[theorem]{Lemma}
\newtheorem{proposition}[theorem]{Proposition}
\newtheorem{corollary}[theorem]{Corollary}

\DeclareMathOperator{\supp}{supp}

\numberwithin{equation}{section}

\begin{document}

\title[A para-differential renormalization technique]{A
  para-differential renormalization technique for nonlinear dispersive
  equations}
\author[S.~Herr]{Sebastian~Herr} \address{Universit\"at
  Bonn}\email{herr@math.uni-bonn.de}

\author[A.D.~Ionescu]{Alexandru~D.~Ionescu} \address{University of
  Wisconsin--Madison}\email{ionescu@math.wisc.edu}

\author[C.E.~Kenig]{Carlos~E.~Kenig} \address{University of
  Chicago}\email{cek@math.uchicago.edu}

\author[H.~Koch]{Herbert~Koch} \address{Universit\"at
  Bonn}\email{koch@math.uni-bonn.de}

\thanks{S.H. was supported in part by the DFG, Sonderforschungsbereich 611. A. I. was supported in part by a Packard Fellowship. C. K.
  was supported in part by NSF grant DMS0456583. H.K. was supported in part by the DFG, Sonderforschungsbereich 611.}

\begin{abstract}
  For $\alpha\in(1,2)$ we prove that the initial-value problem
  \begin{equation*}
    \begin{cases}
      \partial_tu+D^\al\partial_xu+\partial_x(u^2/2)=0
      \text{ on }\mathbb{R}_x\times\mathbb{R}_t;\\
      u(0)=\phi,
    \end{cases}
  \end{equation*}
  is globally well-posed in the space of real-valued
  $L^2$-functions. We use a frequency dependent renormalization method
  to control the strong low-high frequency interactions.
\end{abstract}
\subjclass[2000]{Primary: 35Q53; Secondary: 35Q35, 35S50}

\maketitle
\tableofcontents

\section{Introduction}\label{section1}
In this paper we consider the initial-value problem for the
Benjamin-Ono equation with generalized dispersion
\begin{equation}\label{eq-1}
  \begin{cases}
    \partial_tu+D^\al\partial_xu+\partial_x(u^2/2)=0\text{ on }\mathbb{R}_x\times\mathbb{R}_t;\\
    u(0)=\phi,
  \end{cases}
\end{equation}
where $\alpha\in(1,2)$, and $D^\al$ denotes the operator defined by
the Fourier multiplier $\xi\to|\xi|^\alpha$. Let
$H^\sigma_r=H^\sigma_r(\mathbb{R})$, $\sigma\in[0,\infty)$, denote the
space of {\it{real-valued}} functions $\phi$ with the usual Sobolev
norm
\begin{equation*}
  \|\phi\|_{H^\sigma_r}=\|\phi\|_{H^\sigma}=\|(1+|\xi|^2)^{\sigma/2}\widehat{\phi}(\xi)\|_{L^2_\xi}.
\end{equation*} 
Let $H^\infty_r=\cap_{\sigma\in \Z_+} H^\sigma_r$ with the induced
metric. Suitable solutions of \eqref{eq-1} satisfy the $L^2$
conservation law: if $T_1<T_2\in\R$ and $u\in C((T_1,T_2):H^\infty_r)$
is a solution of the equation
$\partial_tu+D^\al\partial_xu+\partial_x(u^2/2)=0$ on
$\mathbb{R}\times(T_1,T_2)$ then
\begin{equation}\label{conserve}
  \|u(t_1)\|_{H^0_r}=\|u(t_2)\|_{H^0_r}\,\text{ for any }t_1,t_2\in(T_1,T_2).
\end{equation}
Our main theorem concerns the global well-posedness in $H^0_r$ of the
initial-value problem \eqref{eq-1}.
\begin{theorem}\label{Main1}
  (a) Assume $\phi\in H^\infty_r$. Then there is a unique global
  solution
  \begin{equation*}
    u=S^\infty(\phi)\in C(\R:H^\infty_r)
  \end{equation*}
  of the initial-value problem \eqref{eq-1}.

  (b) Assume $T\in\R_+$. Then the mapping
  \[S^\infty_T=\mathbf{1}_{(-T,T)}(t)\cdot S^\infty:\; H^\infty_r\to
  C((-T,T):H^\infty_r)\] extends uniquely to a continuous mapping
  \[S^0_T:H^0_r\to C((-T,T):H^0_r),\] and
  \begin{equation*}
    \|S^0_T(\phi)(t)\|_{H^0_r}=\|\phi\|_{H^0_r}\text{ for any }t\in(-T,T).
  \end{equation*}
\end{theorem}

One-dimensional models such as \eqref{eq-1} have been studied
extensively. The case $\alpha=2$ corresponds to the KdV equation,
while the case $\alpha=1$ corresponds to the Benjamin--Ono
equation. Global well-posedness in $H^0_r$ is known in both of these
cases, see \cite{Bo} and \cite{IoKe} respectively. Other local and
global well-posedness results for \eqref{eq-1} in Sobolev spaces
$H^s_r$ have been obtained by several authors, see \cite{KePoVe1,
  KePoVe2, KePoVe3, CoKeStTaTa, ChCoTa, Gu2} for the KdV case
$\alpha=2$, and \cite{Io, Po, KoTz, KeKo, Tao1} for the Benjamin--Ono
case $\alpha=1$.

The dispersion generalized model \eqref{eq-1} has also been analyzed
in the literature, see for example \cite{KePoVe1, CoKeSt, MoRi, He,
  Gu}. For example, local well-posedness in Sobolev spaces $H^s(\R)$
for $s>-3/4(\alpha-1)$, and global well-posedness in $H^s_r(\R)$ in
the range $s\geq 0$, has been shown by the first author \cite{He}
under an additional \emph{low frequency constraint} on the initial
data. Without this low frequency constraint, the Sobolev index for
local well-posedness has been pushed down to $s>2-\alpha$ by Z. Guo in
\cite{Gu}, using the method of \cite{IoKeTa}.

The nonlinearity of the dispersion generalized Benjamin-Ono equation
\eqref{eq-1} is too strong to allow a direct perturbative argument
(without a low frequency constraint) since the flow map is not locally
uniformly continuous in $H^s_r(\R)$, $ s> 0$. Problems with this
feature have attracted considerable interest in recent years. It is
not difficult to see the reason for this failure at the hand of the
model problem
\[
v_t+vv_x=0,
\]
see \cite[p.2]{KoTz}. Given a solution $v$ we obtain the family of
solutions
\begin{equation}\label{eq:shift}
  v_c(t,x):=v(t,x-ct)+c\; , \quad c\in \R.
\end{equation}
If $v(0,x)$ is of high frequency, the constant $c$ (the low frequency
part) induces a spatial shift of the high frequency part and the lack
of uniform dependence on the initial data becomes evident. The
construction of smooth, square-integrable analogs of \eqref{eq:shift}
for \eqref{eq-1} has been caried out in \cite{KoTz} in detail in the
case $\alpha=1$, see also \cite{MoSaTz}. We notice that the failure of
uniform continuity is irrespective of the regularity assumption which
is imposed on the initial data.

Tao has interpreted this phenomenon for the Benjamin-Ono equation as a
gauge change, which opened the path to the satisfactory well-posedness
result in \cite{IoKe} for the Benjamin-Ono equation, i.e. the case
$\alpha=1$. There the gauge change can be undone by a multiplication
of projection to positive frequencies of the solution by a function
$e^{i\phi}$. The linear part reduces to a Schr\"odinger equation for
positive frequencies, and commuting $e^{i\phi} $ by the Schr\"odinger
equation leads to a drift term which balances the worst low-high
interaction
\[ u_{\mathrm{low}} \partial_x u_{\mathrm{high}} \] of the quadratic
part.

The same ideas show that one encounters a pseudo-differential gauge
transform for the dispersion generalized Benjamin-Ono equation
\eqref{eq-1}. We do not pursue this pseudo-differential point of view,
but it is advisable to keep it in mind. Instead we decompose the
solution into small frequency bands of size $\sqrt{\lambda}$ at
frequency $\lambda$. At this frequency scale the gauge change is
essentially a multiplication by a purely imaginary phase function. We
carry out bilinear estimates for these frequency bands and study the
effect of the gauge transform. This is technical and painful. The main
contribution of this paper is the demonstration that this circle of
ideas can be carried through for a non trivial example. The phenomenon
described above will most likely be encountered at other problems as
well. Gaining a general understanding of it seems to be desirable and
this is our aim.

The rest of the paper is organized as follows: in section
\ref{estimates} we reduce Theorem \ref{Main1} to proving several
apriori bounds on smooth solutions, and differences of smooth
solutions, of the equation \eqref{eq-1}, on bounded time
intervals. This reduction relies on energy-type estimates.

In section \ref{construct} we construct our main renormalization,
which is the key step to further reducing the problem to perturbative
analysis. After subtracting the low frequency component of the
solution, which is essentially left unchanged by the evolution due to
the null structure of the nonlinearity, we further decompose the
solution into frequency blocks and multiply each frequency block by a
suitable bounded factor. This renormalization leads to an infinite
system of coupled equations satisfied by the frequency blocks. A
similar construction was used by two of the authors in \cite{IoKe} for
the Benjamin--Ono equation. However, in our situation, we need to
renormalize each frequency block by a different factor, which leads to
substantial technical difficulties in the perturbative analysis.

In section \ref{linear} we define our main normed spaces, and show
that the main theorem can be reduced to proving the nonlinear
estimates in Proposition \ref{Lemmat2}.

The remaining sections are concerned with the proof of Proposition
\ref{Lemmat2}. We prove first frequency-localized bilinear estimates,
see sections \ref{locL2} and \ref{Dyadic2}. Then we prove
frequency-localized linear estimates on operators defined by
multiplication by smooth bounded factors, see section \ref{mult}, and
frequency-localized commutator estimates, see section
\ref{commutatorest}. Finally, we put all these estimates together in
section \ref{mainlemma} and \ref{lastsection} to complete the proof of
Proposition \ref{Lemmat2}.

\section{Reduction to a priori estimates}\label{estimates}

We recall first a standard local well-posedness theorem:

\begin{proposition}\label{Lemmaw1}
  Assume $\phi\in H^\infty_r$. Then there is $T=T(
  \|\phi\|_{H^2_r})\in (0,1]$ and a unique solution
  $u=S^\infty_T(\phi)\in C( (-T,T):H^\infty_r)$ of the initial-value
  problem
  \begin{equation}\label{rg0}
    \begin{cases}
      \partial_tu+D^\al\partial_xu+\partial_x(u^2/2)=0\text{ on }\mathbb{R}_x\times(-T,T);\\
      u(0)=\phi.
    \end{cases}
  \end{equation}
  In addition, for any $\sigma\geq 2$,
  \begin{equation*}
    \sup_{t\in(-T,T)}\|u(t)\|_{H^\sigma_r}\leq C(\sigma,\|\phi\|_{H^\sigma_r},\sup_{t\in(-T,T)}\|u(t)\|_{H^2_r}).
  \end{equation*}
\end{proposition}

Most of the paper is concerned with proving suitable {\it{a priori}}
estimates on the solutions $S^\infty_T$ constructed in Proposition
\ref{Lemmaw1}. For Theorem \ref{Main1} (a) we need the following
estimate:

\begin{proposition}\label{Lemmaw2}
  There is a constant $\varepsilon_0=\ep_0(\alpha)>0$ with the
  property that if $T\in(0,1]$, $\phi\in H^\infty_r$,
  \begin{equation}\label{rg5}
    \|\phi\|_{H^0_r}\leq \varepsilon_0,
  \end{equation}
  and $u=S^\infty_T(\phi)\in C( (-T,T):H^\infty_r)$ is a solution of
  the initial-value problem \eqref{rg0}, then
  \begin{equation}\label{rg6}
    \sup_{t\in(-T,T)}\|u(t)\|_{H^2_r}\leq C\|\phi\|_{H^2_r}.
  \end{equation}
\end{proposition}

Theorem \ref{Main1} (a) follows easily from Proposition \ref{Lemmaw2}:
by scaling (i.e. replace $\phi$ by $\phi_\lambda(x)=\lambda
^\alpha\phi(\lambda x)$ and $u$ by $u_\lambda (x,t)=\lambda ^\alpha
u(\lambda x,\lambda^{\alpha+1}t)$, $\lambda \ll 1$) it suffices to
prove Theorem \ref{Main1} (a) for $\phi\in H^\infty_r$ with
$\|\phi\|_{H^0_r}\leq \ep_0$. Using Proposition \ref{Lemmaw1} and
Proposition \ref{Lemmaw2}, we can construct the solution
$S^\infty_1(\phi)\in C( (-1,1):H^\infty_r)$. Finally, we use the
conservation law \eqref{conserve} to extend the solution to the entire
real line.

To prove Theorem \ref{Main1} (b) we need an additional estimate on
differences of solutions. For any $\phi\in H^\infty_r$ and
$N\in[1,\infty)$ let
$\phi^{N}=\mathcal{F}^{-1}[\widehat{\phi}(\xi)\cdot
\mathbf{1}_{[-N,N]}(\xi)]\in H^\infty_r$.

\begin{proposition}\label{Lemmaw3}
  Assume $N\in [2,\infty)$, $\phi\in H^\infty_r$, and
  $\|\phi\|_{H^0_r}\leq \varepsilon_0$ (see \eqref{rg5}). Then
  \begin{equation}\label{rg8}
    \sup_{t\in(-1,1)}\|S^\infty(\phi)(t)-S^\infty(\phi^N)(t)\|_{H^0_r}\leq C\|\phi-\phi^N\|_{H^0_r}. 
  \end{equation}
\end{proposition}

We show now that Proposition \ref{Lemmaw2} and Proposition
\ref{Lemmaw3} imply Theorem \ref{Main1} (b). Assume $\phi\in H^0_r$ is
fixed,
\begin{equation*}
  \phi_n\in H^\infty_r\text{ and }\lim_{n\to\infty}\phi_n=\phi\text{ in }H^0_r. 
\end{equation*}
For Theorem \ref{Main1} (b) it suffices to prove that the sequence
$S^\infty_T(\phi_n)\in C((-T,T):H^\infty_r)$ is a Cauchy sequence in
$C((-T,T):H^0_r)$. By scaling, we may assume $\|\phi\|_{H^0_r}\leq
\ep_0/2$. Using the conservation law \eqref{conserve} it suffices to
prove that for any $\de>0$ there is $M_\de$ such that
\begin{equation}\label{rg9}
  \sup_{t\in(-1,1)}\|S^\infty(\phi_m)(t)-S^\infty(\phi_n)(t)\|_{H^0_r}\leq\de\text{ for any }m,n\geq M_\delta.
\end{equation}
We observe now that $\|\phi_n-\phi_n^N\|_{H^0_r}\leq
\|\phi-\phi^N\|_{H^0_r}+\|\phi-\phi_n\|_{H^0_r}$. Thus we can fix
$N=N(\delta,\phi)\geq 2$ and $M^1_\delta$ such that
$\|\phi_n-\phi_n^N\|_{H^0_r}\leq \delta/(4C)$ and $\|\phi_n\|\leq
\ep_0$ for any $n\geq M^1_\delta$, where $C$ is the constant in
\eqref{rg8}. Thus, using \eqref{rg8},
\begin{equation}\label{rg10}
  \sup_{t\in(-1,1)}\|S^\infty(\phi_n)(t)-S^\infty(\phi_n^N)(t)\|_{H^0_r}\leq\de/4\text{ for any }n\geq M^1_\delta.
\end{equation}

It remains to estimate
\begin{equation*}
  \sup_{t\in(-1,1)}\|S^\infty(\phi_n^N)(t)-S^\infty(\phi_m^N)(t)\|_{H^0_r}.
\end{equation*} 
Using standard energy estimates for the difference equation, we have
\begin{equation*}
  \begin{split}
    &\sup_{t\in(-1,1)}\|S^\infty(\phi_n^N)(t)-S^\infty(\phi_m^N)(t)\|_{H^0_r}\\
    &\leq \|\phi_n^N-\phi_m^N\|_{H^0_r}\cdot \exp\Big(C\int_{-1}^1\|\partial_x(S^\infty(\phi_n^N))(t)\|_{L^\infty_x}+\|\partial_x(S^\infty(\phi_m^N))(t)\|_{L^\infty_x}\,dt\Big)\\
    &\leq \|\phi_n-\phi_m\|_{H^0_r}\cdot
    \exp\Big(C\sup_{t\in(-1,1)}\big[\|S^\infty(\phi_n^N)(t)\|_{H^2_r}+\|S^\infty(\phi_m^N)(t)\|_{H^2_r}\big]\Big).
  \end{split}
\end{equation*}
Using Proposition \ref{Lemmaw2}, it follows that
\begin{equation}\label{rg11}
  \begin{split}
    \sup_{t\in(-1,1)}&\|S^\infty(\phi_n^N)(t)-S^\infty(\phi_m^N)(t)\|_{H^0_r}\\
    &\leq \|\phi_n-\phi_m\|_{H^0_r}\cdot \exp\Big(C[ \|\phi_n^N\|_{H^2_r}+\|\phi_m^N\|_{H^2_r}]\Big)\\
    &\leq \|\phi_n-\phi_m\|_{H^0_r}\cdot \exp(CN^2).
  \end{split}
\end{equation}
The bound \eqref{rg9} follows from \eqref{rg10} and \eqref{rg11}.

This completes the proof of Theorem \ref{Main1}. The rest of the paper
is concerned with proving Proposition \ref{Lemmaw2} and Proposition
\ref{Lemmaw3}.

\section{The main renormalization}\label{construct}

The initial-value problem \eqref{eq-1} cannot be analyzed
perturbatively, due to the strong interactions between very low and
high frequencies. In this section we construct a para-differential
renormalization which allows us to recover information about the
solution $u$ of \eqref{eq-1} indirectly, by analyzing perturbatively a
system of equations satisfied by suitably renormalized frequency
components of $u$, see \eqref{rh60} and \eqref{rh50}.

Assume in this section that $T\in(0,1]$ and $u\in C(
(-T,T):H^\infty_r)$ is a solution of the initial-value problem
\begin{equation}\label{rh0}
  \begin{cases}
    \partial_tu+D^\al\partial_xu+\partial_x(u^2/2)=0\text{ on }\mathbb{R}_x\times(-T,T);\\
    u(0)=\phi.
  \end{cases}
\end{equation}
Assume, in addition, that $\|\phi\|_{H^0_r}\leq \ep_0$, for some
sufficiently small constant $\ep_0$ (compare with \eqref{rg5}). Let
$\phi_{\mathrm{low}}=\mathcal{F}^{-1}[\widehat{\phi}(\xi)\cdot
\mathbf{1}_{[-1/2,1/2]}(\xi)]\in H^\infty_r$ and
$\phi_{\mathrm{high}}=\phi-\phi_{\mathrm{low}}$. Let
$v(x,t)=u(x,t)-\phi_{\mathrm{low}}(x)$, so
\begin{equation}\label{rh1}
  \begin{cases}
    \partial_tv+D^\al\partial_xv+\phi_{\mathrm{low}}\cdot \partial_xv=-\partial_x(v^2/2)-v\cdot \partial_x\phi_{\mathrm{low}}-D^\al\partial_x\phi_{\mathrm{low}}-\partial_x(\phi_{\mathrm{low}}^2/2);\\
    v(0)=\phi_{\mathrm{high}}.
  \end{cases}
\end{equation}
on $\mathbb{R}_x\times(-T,T)$.

We fix the increasing sequence $\{n_k\}_{k\in\Z}$
\begin{equation}\label{rh30}
  \begin{cases}
    n_0=0,\,\,n_1=4,\,\,n_{k+1}=n_k+n_k^{1/2}\text{ for }k=1,2,\ldots;\\
    n_{-k}=-n_k\text{ for }k=1,2,\ldots.
  \end{cases}
\end{equation}
Clearly, $|n_k|\approx |k|^2$ for $|k|\geq 1$. We fix also smooth
functions $\chi_k:\R\to[0,1]$, $k\in\Z$, $\chi_k$ supported in
$[(2n_{k-1}+n_k)/3,(2n_{k+1}+n_k)/3]$ such that
\begin{equation}\label{rh31}
  \begin{cases}
    &\sum_{k\in\Z}\chi_k\equiv 1;\\
    &|\partial_\xi^\sigma\chi_k|\leq C(1+|n_k|)^{-\sigma/2}\text{ for
      any } k\in\Z\text{ and }\sigma=0,1,2.
  \end{cases}
\end{equation}
For $k\in\Z$ let
\begin{equation}\label{rh31.1}
  I_k=[(5n_{k-1}+n_k)/6,(5n_{k+1}+n_k)/6].
\end{equation}
Let $P_k$ and $\widetilde{P}_k$, $k\in\Z$, denote the operators
defined by the Fourier multipliers $\chi_k$ and $\mathbf{1}_{I_k}$
respectively.

We apply the operators $P_k$, $k\in\Z\setminus\{0\}$, to the identity
\eqref{rh1}; the result is
\begin{equation}\label{rh2}
  \begin{cases}
    \partial_t(P_{k}v)+D^\al\partial_x(P_{k}v)+\phi_{\mathrm{low}}\cdot \partial_x(P_{k}v)=E_{k}\text{ on }\R_x\times(-T,T);\\
    (P_{k}v)(0)=P_{k}(\phi_{\mathrm{high}}),
  \end{cases}
\end{equation}
where
\begin{equation}\label{rh3}
  E_{k}=[\phi_{\mathrm{low}}\cdot \partial_x(P_{k}v)-P_{k}(\phi_{\mathrm{low}}\cdot \partial_xv)]-P_{k}\partial_x(v^2/2)-P_{k}(v\cdot \partial_x\phi_{\mathrm{low}}).
\end{equation}
We apply also the operator $P_0$ to the identity \eqref{rh1}; the
result is
\begin{equation}\label{rh4}
  \begin{cases}
    \partial_t(P_0v)+D^\al\partial_x (P_0v)=R_0\text{ on }\R_x\times(-T,T);\\
    (P_0v)(0)=P_0(\phi_{\mathrm{high}}),
  \end{cases}
\end{equation}
where
\begin{equation}\label{rh5}
  \begin{split}
    R_0=&-P_0(\phi_{\mathrm{low}}\cdot \partial_xv)-P_0\partial_x(v^2/2)\\
    &-P_0(v\cdot
    \partial_x\phi_{\mathrm{low}})-D^\al\partial_xP_0(\phi_{\mathrm{low}})-P_0\partial_x(\phi_{\mathrm{low}}^2/2).
  \end{split}
\end{equation}

We define the smooth function $\Psi:\mathbb{R}\to\mathbb{R}$ as the
anti-derivative of $\phi_{\mathrm{low}}$,
\begin{equation}\label{fg9}
  \Psi'(x)=\phi_{\mathrm{low}}(x)\text{ and }\Psi(0)=0.
\end{equation}
For $k\in\Z\setminus\{0\}$ we define the functions
\begin{equation}\label{rh18}
  v_{k}(x,t)=P_{k}(v)(x,t)\cdot e^{-ia_{k}\Psi(x)} \text{ where }a_k=-n_k|n_k|^{-\al}/(\alpha+1).
\end{equation} 
We substitute the identity $P_{k}(v)=e^{ia_{k}\Psi}v_{k}$ into
\eqref{rh2}; the result is
\begin{equation}\label{rh8}
  e^{ia_k\Psi}\partial_tv_k+e^{ia_k\Psi}D^\al\partial_x(v_k)+(\alpha+1)D^\alpha v_k\cdot (ia_k\Psi')e^{ia_k\Psi}+\phi_{\mathrm{low}}e^{ia_k\Psi}\partial_x(v_k)=E'_{k}
\end{equation}
where
\begin{equation*}
  \begin{split}
    E'_k&=[e^{ia_k\Psi}D^\al\partial_x(v_k)+(\alpha+1)D^\alpha v_k\cdot (ia_k\Psi')e^{ia_k\Psi}-D^\alpha\partial_x(e^{ia_k\Psi}v_k)]\\
    &+E_k-\phi_{\mathrm{low}}(ia_k\Psi')e^{ia_k\Psi}\cdot v_k.
  \end{split}
\end{equation*}
We multiply \eqref{rh8} by $e^{-ia_k\Psi}$. Using the definition
\eqref{rh18} of the coefficients $a_k$, it follows that
\begin{equation}\label{rh20}
  \begin{cases}
    \partial_tv_k+D^\al\partial_x(v_k)=R_k\text{ on }\R_x\times(-T,T);\\
    v_k(0)=e^{-ia_k\cdot \Psi}\cdot P_k(\phi_{\mathrm{high}}),
  \end{cases}
\end{equation}
where
\begin{equation}\label{rh21}
  \begin{split}
    R_k=&-e^{-ia_k\Psi}P_k\partial_x(v^2/2)\\
    &-\phi_{\mathrm{low}}[\partial_xv_k-D^\alpha v_k\cdot (in_k|n_k|^{-\alpha})]\\
    &-[e^{-ia_k\Psi}D^\alpha\partial_x(e^{ia_k\Psi}v_k)-D^\al\partial_x(v_k)-(\alpha+1)D^\alpha v_k\cdot (ia_k\Psi')]\\
    &-e^{-ia_k\Psi}[P_k(\phi_{\mathrm{low}}\cdot \partial_xv)-\phi_{\mathrm{low}}\cdot \partial_x(P_kv)]\\
    &-[ia_k\phi^2_{\mathrm{low}}\cdot v_k+e^{-ia_k\Psi}P_k(v\cdot
    \partial_x\phi_{\mathrm{low}})].
  \end{split}
\end{equation}

To summarize, given a solution $u\in C( (-T,T):H^\infty_r)$ of the
initial-value problem \eqref{rh0} we constructed functions $v_k\in C(
(-T,T):H^\infty)$, $k\in \Z$, which solve the initial-value problems
\begin{equation}\label{rh60}
  \begin{cases}
    \partial_tv_k+D^\al\partial_x(v_k)=R_k\text{ on }\R_x\times(-T,T);\\
    v_k(0)=e^{-ia_k\cdot \Psi}\cdot P_k(\phi_{\mathrm{high}}),
  \end{cases}
\end{equation}
where, for simplicity of notation, $a_0=0$. The functions $R_k\in C(
(-T,T):H^\infty)$ are defined in \eqref{rh5} for $k=0$, and
\eqref{rh21} for $k\neq 0$. In addition, $u=v+\phi_{\mathrm{low}}$,
\begin{equation}\label{rh50}
  v=\sum_{k\in\Z}e^{ia_{k}\Psi}v_{k},\text{ and }v_k=e^{-ia_k\Psi}\widetilde{P}_k(e^{ia_{k}\Psi}v_{k})\text{ for any }k\in\Z.
\end{equation}

\section{Proof of the main theorem}\label{linear}

The rest of the argument is based on perturbative analysis of the
system of equations \eqref{rh60}, for fixed $\phi_{\mathrm{low}}$. In
this section we define our main normed spaces used for this
perturbative analysis and show how to reduce Propositions
\ref{Lemmaw2} and \ref{Lemmaw3} to the more technical Proposition
\ref{Lemmat2} below. Proposition \ref{Lemmat2} will be proved in the
remaining sections of the paper.

The normed spaces constructed in this section are very similar to
those used by two of the authors in \cite{IoKe} for the analysis of
the Benjamin--Ono equation. However, our spaces are adapted to the
frequency intervals $I_k$ constructed in section \ref{construct},
instead of dyadic intervals, since they are used to measure the
components $v_k$ and their renormalizations.

Let $\eta_0:\mathbb{R}\to[0,1]$ denote an even smooth function
supported in $[-8/5,8/5]$ and equal to $1$ in $[-5/4,5/4]$. For
$k\in\{1,2,\ldots\}$ let $\eta_k(\nu)=\eta_0(\nu/2^k)-\eta_0(\nu
/2^{k-1})$. For $k\in\Z$ let $\widetilde{\eta}_k(\nu )=\eta_0(\nu
/2^k)-\eta_0(\nu /2^{k-1})$. We define the sets $J_0=[-2,2]$,
$J_k=\{\nu\in\R:|\nu|\in[2^{k-1},2^{k+1}]$, $k=1,2,\ldots $, and
$\widetilde{J}_k=\{\nu\in\R:|\nu|\in[2^{k-1},2^{k+1}]$, $k\in\Z$. For
$\xi\in\mathbb{R}$ let
\begin{equation}\label{omega}
  \omega(\xi)=-\xi|\xi|^{\al}.
\end{equation}

Recall the sequence $n_k$, the functions $\chi _k$, and the intervals
$I_k$, $k\in\Z$, defined in \eqref{rh30}--\eqref{rh31.1}. We fix
$\delta=(\alpha-1)/100\in(0,1/100)$. We define the normed spaces
$Z_k=Z_k(\mathbb{R}\times\mathbb{R})$, $k\in\mathbb{Z}$: for $|k|\geq
1$ (high frequencies) we define
\begin{equation}\label{def1}
  \begin{split}
    Z_k=&\{f\in L^2:\, f \text{ supported in }I_k\times\mathbb{R}\text{ and }\\
    &\|f\|_{Z_k}=\sum_{j=0}^\infty
    2^{j/2}\beta_{k,j}\|\eta_j(\tau-\omega(\xi))f(\xi,\tau)\,\|_{L^2_{\xi,\tau}}<\infty\},
  \end{split}
\end{equation}
where
\begin{equation}\label{def1'}
  \beta_{k,j}=1+(2^j/|n_k|^{\al+1})^{1/2-\delta}.
\end{equation}
Notice that $2^{j/2}\beta_{k,j}\approx 2^{j(1-\delta)}$ when $k$ is
small. For $k=0$ (low frequencies) we define
\begin{equation}\label{def1''}
  \begin{split}
    X_0=\{&f\in L^2:\, f \text{ supported in }I_0\times\mathbb{R}\text{ and }\\
    &\|f\|_{X_0}=\sum_{j=0}^\infty\sum_{k'=-\infty}^22^{j(1-\delta)}2^{-k'}\|\eta_j(\tau)\widetilde{\eta}_{k'}(\xi)f(\xi,\tau)\,\|_{L^2_{\xi,\tau}}<\infty\}
  \end{split}
\end{equation}
and
\begin{equation}\label{def1c}
  \begin{split}
    Y_0=\{&f\in L^2:\, f \text{ supported in }I_0\times\mathbb{R}\text{ and }\\
    &\|f\|_{Y_0}=\sum_{j=0}^\infty2^{j(1-\delta)}\|\mathcal{F}^{-1}[\eta_j(\tau)f(\xi,\tau)]\|_{L^1_xL^2_t}<\infty\}.
  \end{split}
\end{equation}
Finally, we define
\begin{equation}\label{def1d}
  Z_0=X_0+Y_0.
\end{equation}
The normed spaces $Z_k$, $k\in\Z$, are our main spaces of functions
defined in the Fourier space. They are similar to the spaces used in
\cite{IoKe}, but slightly simpler (we do not need the spaces $Y_k$ for
$|k|\geq 1$, compare to \cite[Section 3]{IoKe}, since the local
smoothing phenomenon is not essential if $\alpha>1$).

For $\sigma\in[0,\infty)$ we define the normed spaces
$\widetilde{H}^{\sigma}=\widetilde{H}^{\sigma}(\mathbb{R})$,
$\mathbf{F}^{\sigma}=\mathbf{F}^{\sigma}(\mathbb{R}\times\mathbb{R})$,
and
$\mathbf{N}^{\sigma}=\mathbf{N}^{\sigma}(\mathbb{R}\times\mathbb{R})$. We
define first
\begin{equation}\label{def4}
  \begin{split}
    &\widetilde{H}^{\sigma}=\Big\{\phi\in H^\infty:\|\phi\|_{\widetilde{H}^{\sigma}}^2=\|\chi_0\mathcal{F}(\phi)\|^2_{B_0}+\sum_{|k|\geq 1} (1+|n_k|)^{2\sigma}\|\chi_k\mathcal{F}(\phi)\|_{L^2_\xi}^2<\infty\Big\},\\
    &\|f\|_{B_0}=\inf_{f=g+h}\|\mathcal{F}^{-1}(g)\|_{L^1}+\sum_{k'=-\infty}^22^{-k'}\|\widetilde{\eta}_{k'}
    h\,\|_{L^2_{\xi}}.
  \end{split}
\end{equation}
Then we define
\begin{equation}\label{def5}
  \begin{split}
    \mathbf{F}^{\sigma}=\Big\{&u\in
    C(\R:\widetilde{H}^{\sigma}):u\text{ supported in }
    \R_x\times [-4,4]\\
    &\text{ and }\|u\|_{\mathbf{F}^{\sigma}}^2=\sum_{k\in\Z} (1+|n_k| )^{2\sigma}\|P_k u\|^2_{F_k}<\infty\\
    &\text{ where } \|P_k
    u\|_{F_k}=\|\chi_k(\xi)\cdot\mathcal{F}(u)\|_{Z_k}\Big\},
  \end{split}
\end{equation}
and
\begin{equation}\label{def6}
  \begin{split}
    \mathbf{N}^{\sigma}=\Big\{& u\in C(\R:\widetilde{H}^{\sigma}):u\text{ supported in }\R_x\times [-4,4]\\
    &\text{ and }\|u\|_{\mathbf{N}^{\sigma}}^2=\sum_{k\in\Z} (1+|n_k| )^{2\sigma}\|P_k u\|^2_{N_k}<\infty\\
    &\text{ where }
    \|P_ku\|_{N_k}=\|\chi_k(\xi)(\tau-\omega(\xi)+i)^{-1}\cdot
    \mathcal{F}(u)\|_{Z_k}\Big\}.
  \end{split}
\end{equation}
Finally, for any $T'\in(0,1]$, $\sigma\in[0,2]$ and $f\in
C((-T',T'):\widetilde{H}^\sigma)$ we define\begin{equation*}
  \|f\|_{\F^\sigma(T')}=\inf_{\widetilde{f}=f\text{ in
    }\mathbb{R}\times(-T',T')}\|\widetilde{f}\|_{\F^\sigma},
\end{equation*}
and
\begin{equation*}
  \|f\|_{\N^\sigma(T')}=\inf_{\widetilde{f}=f\text{ in }\mathbb{R}\times(-T',T')}\|\widetilde{f}\|_{\N^\sigma}.
\end{equation*}

It follows easily from the definitions that
\begin{equation}\label{hh80}
  \sup_{t\in\R}\|u(t)\|_{\widetilde{H}^{\sigma}}\leq C \|u\|_{\mathbf{F}^\sigma}
\end{equation}
if $\sigma\in[0,2]$ and $u\in \mathbf{F}^\sigma$, see \cite[Lemma
4.2]{IoKe} for a similar proof. Thus $\F^\sigma\hookrightarrow
C(\R:\widetilde{H}^{\sigma})$.

For any $u\in C(\mathbb{R}:L^2)$ let $\widetilde{u}(.,t)\in
C(\mathbb{R}:L^2)$ denote its partial Fourier transform with respect
to the variable $x$. For $\phi\in H^\infty $ let $W(t)\phi\in
C(\mathbb{R}:H^\infty )$ denote the solution of the free evolution
given by
\begin{equation}\label{ni1}
  [W(t)\phi]\,\,\widetilde{}\,\,(\xi,t)=e^{it\omega(\xi)}\widehat{\phi}(\xi),
\end{equation}
where $\omega(\xi)$ is defined in \eqref{omega}. We record first an
$\widetilde{H}^{\sigma}\to \mathbf{F}^\sigma$ homogeneous linear
estimate.

\begin{proposition}\label{Lemmab1}
  If $\sigma\in[0,2]$ and $\phi\in \widetilde{H}^{\sigma}$ then
  \begin{equation}\label{ty3}
    \|\mathbf{1}_{(-1,1)}(t)\cdot (W(t)\phi)\|_{\mathbf{F}^{\sigma}(1)}\leq C\|\phi\|_{\widetilde{H}^{\sigma}}.
  \end{equation}
\end{proposition}

See, for example, \cite[Lemma 5.1]{IoKe} for a similar proof. We need
also an inhomogeneous $\mathbf{N}^\sigma\to \mathbf{F}^\sigma$ linear
estimate, see, for example, \cite[Lemma 5.2]{IoKe} for a similar
proof.

\begin{proposition}\label{Lemmab3}
  If $\sigma\in[0,2]$, $T\in(0,1]$, and $u\in \mathbf{N}^{\sigma}(T)$
  then
  \begin{equation*}
    \Big\|\int_0^tW(t-s)(u(s))\,ds\Big\|_{\mathbf{F}^{\sigma}(T)}\leq C\|u\|_{\mathbf{N}^{\sigma}(T)}.
  \end{equation*}
\end{proposition}

In the rest of this section we use the notation in section
\ref{construct}. In particular, given a solution $u\in
C((-T,T):H^\infty_r)$, $T\in(0,1]$, of the initial-value problem
\eqref{rh0} with $\|\phi\|_{H^0_r}\leq\ep_0$, we constructed the
functions $v_k\in C( (-T,T):H^{\infty})$, $k\in \Z$, which solve the
equations
\begin{equation}\label{ro20}
  \begin{cases}
    \partial_tv_k+D^\al\partial_x(v_k)=R_k\text{ on }\R_x\times(-T,T);\\
    v_k(0)=e^{-ia_k\cdot \Psi}\cdot P_k(\phi_{\mathrm{high}}).
  \end{cases}
\end{equation}
Here $R_k\in C( (-T,T):H^{\infty})$, $k\in\Z$, are as in \eqref{rh5}
and \eqref{rh21}, and $a_k$ are defined in \eqref{rh18}. Assume also
that $u'\in C((-T,T):H^\infty_r)$ is a solution of the initial-value
problem
\begin{equation*}
  \begin{cases}
    \partial_t{u'}+D^\al\partial_x{u'}+\partial_x({u'}^2/2)=0\text{ on }\mathbb{R}_x\times(-T,T);\\
    u'(0)=\phi'.
  \end{cases}
\end{equation*}
Assume, in addition, that
\begin{equation}\label{ly1}
  \|\phi'\|_{H^0_r}\leq\ep_0\quad\text{ and }\quad\phi'_{\mathrm{low}}=\phi_{\mathrm{low}},
\end{equation}
where, as in section \ref{construct},
$\phi'_{\mathrm{low}}=\mathcal{F}^{-1}[\widehat{\phi'}(\xi)\cdot\mathbf{1}_{[-1/2,1/2]}(\xi)]$. Let
$v'_k$, $R'_k$, $n_k$, $a_k$, $\Psi$ be defined as in section
\ref{construct}, so
\begin{equation}\label{ro2.2}
  \begin{cases}
    \partial_tv'_k+D^\al\partial_x(v'_k)=R'_k\text{ on }\R_x\times(-T,T);\\
    v'_k(0)=e^{-ia_k\cdot \Psi}\cdot P_k(\phi'_{\mathrm{high}}).
  \end{cases}
\end{equation}
Our main proposition concerning the nonlinearities $R_k$ and $R'_k$ is
the following:

\begin{proposition}\label{Lemmat2}
  (a) For any $\sigma\in[0,2]$
  \begin{equation*}
    \sum_{k\in\Z}\|R_k\|_{\N^\sigma(T)}^2<\infty.
  \end{equation*}
  In addition, the mapping $\mathcal{N}:(0,T]\to[0,\infty)$,
  \begin{equation*}
    \mathcal{N}(T')=\sum_{k\in\Z}\|R_k\|_{\N^0(T')}^2
  \end{equation*}
  is continuous and increasing on the interval $(0,T]$ and
  \begin{equation*}
    \lim_{T'\to 0}\mathcal{N}(T')=0.
  \end{equation*}

  (b) Assume $\sigma\in[0,2]$ and $T'\in[0,T]$. Then
  \begin{equation}\label{toshow1}
    \sum_{k\in\mathbb{Z}}\|R_k\|_{\N^\sigma(T')}^2\leq
    C\big(\sum_{k\in\mathbb{Z}}\|v_k\|_{\F^\sigma(T')}^2\big)
    \big(\sum_{k\in\mathbb{Z}}\|v_k\|_{\F^0(T')}^2+\ep_0^2\big)+C\|\phi\|_{H^\sigma}^2.
  \end{equation}
  In addition,
  \begin{equation}\label{toshow2}
    \sum_{k\in\mathbb{Z}}\|R'_k-R_k\|_{\N^0(T')}^2\leq
    C\big(\sum_{k\in\mathbb{Z}}\|v'_k-v_k\|_{\F^0(T')}^2\big)
    \big(\sum_{k\in\mathbb{Z}}\big(\|v_k\|_{\F^0(T')}^2+\|v'_k\|_{\F^0(T')}^2\big)+\ep_0^2\big).
  \end{equation}
\end{proposition}

The proof of Proposition \ref{Lemmat2} will cover sections
\ref{locL2}--\ref{lastsection}. In the rest of this section we show
how to use this proposition to complete the proof of Theorem
\ref{Main1}.

\begin{proof}[Proof of Proposition \ref{Lemmaw2}] It follows easily
  from the definitions that for $\sigma\in[0,2]$
  \begin{equation}\label{ro60}
    \sum_{k\in\Z}\|e^{-ia_k\cdot\Psi}\cdot P_k(\phi_{\mathrm{high}})\|^2_{\widetilde{H}^\sigma}\leq C\|\phi_{\mathrm{high}}\|^2_{H^\sigma}.
  \end{equation}
  See also \cite[Lemma10.1]{IoKe} for a similar proof. Using
  Propositions \ref{Lemmab1} and \eqref{Lemmab3} and the equations
  \eqref{ro20} it follows that for any $\sigma\in[0,2]$ and
  $T'\in(0,T]$
  \begin{equation}\label{ro61}
    \sum_{k\in\Z}\|v_k\|^2_{\F^\sigma(T')}\leq C\|\phi_{\mathrm{high}}\|^2_{H^\sigma}+C\sum_{k\in\Z}\|R_k\|^2_{\N^\sigma(T')}.
  \end{equation}
  We set $\sigma=0$ and combine \eqref{ro61} and \eqref{toshow1}; it
  follows that
  \begin{equation*}
    \mathcal{N}(T')\leq C(\ep_0^2+\mathcal{N}(T'))^2+C\ep_0^2.
  \end{equation*}
  Using Proposition \ref{Lemmat2} (a) it follows that
  $\mathcal{N}(T')\leq C\ep_0^2$ for any $T'\in(0,T]$, provided that
  the constant $\ep_0$ is taken sufficiently small. In particular
  \begin{equation*}
    \sum_{k\in\Z}\|R_k\|_{\N^0(T)}^2\leq C\ep_0^2.
  \end{equation*}
  It follows from \eqref{ro61} with $\sigma=0$ that
  \begin{equation}\label{ro62}
    \sum_{k\in\Z}\|v_k\|_{\F^0(T)}^2\leq C\ep_0^2.
  \end{equation}
  Thus, using \eqref{toshow1} with $\sigma=2$
  \begin{equation*}
    \sum_{k\in\Z}\|R_k\|_{\N^2(T)}^2\leq C\ep_0^2\big(\sum_{k\in\Z}\|v_k\|_{\F^2(T)}^2\big)+C\|\phi\|_{H^2}^2.
  \end{equation*}
  Using \eqref{ro61} with $\sigma=2$ it follows that
  \begin{equation*}
    \sum_{k\in\Z}\|v_k\|^2_{\F^2(T)}\leq C\|\phi\|^2_{H^2}+C\ep_0^2\big(\sum_{k\in\Z}\|v_k\|_{\F^2(T)}^2\big),
  \end{equation*}
  therefore, assuming $\ep_0$ sufficiently small,
  \begin{equation*}
    \sum_{k\in\Z}\|v_k\|^2_{\F^2(T)}\leq C\|\phi\|^2_{H^2}.
  \end{equation*}
  Using \eqref{hh80} it follows that for any $t\in(-T,T)$
  \begin{equation}\label{ro65}
    \sum_{k\in\Z}\|v_k(t)\|^2_{\widetilde{H}^2}\leq C\|\phi\|^2_{H^2}.
  \end{equation}
  We use now \eqref{rh50}
  \begin{equation*}
    u=\phi_{\mathrm{low}}+\sum_{k\in\Z}e^{ia_{k}\Psi}v_{k},\text{ and }v_k=e^{-ia_k\Psi}\widetilde{P}_k(e^{ia_{k}\Psi}v_{k})\text{ for any }k\in\Z.
  \end{equation*}
  to prove a bound on $u$. For any $t\in(-T,T)$ we have, using
  \eqref{ro65},
  \begin{align*}
    &\|u(t)\|_{H^2}^2\leq C\|\phi_{\mathrm{low}}\|_{H^2}^2+C\sum_{k\in\Z}\|\widetilde{P}_k(e^{ia_k\Psi}v_k(t))\|_{H^2}^2\\
    \leq &
    C\|\phi_{\mathrm{low}}\|_{H^2}^2+\sum_{k\in\Z}\|v_k(t)\|^2_{H^2}\leq
    C\|\phi\|^2_{H^2}.
  \end{align*}
  This completes the proof of Proposition \ref{Lemmaw2}.
\end{proof}

\begin{proof}[Proof of Proposition \ref{Lemmaw3}] Let $\phi'=\phi^N$,
  so \eqref{ly1} is verified, and subtract equations \eqref{ro2.2} and
  \eqref{ro20}. The result is
  \begin{equation*}
    \begin{cases}
      \partial_t(v'_k-v_k)+D^\al\partial_x(v'_k-v_k)=R'_k-R_k\text{ on }\R_x\times(-T,T);\\
      (v'_k-v_k)(0)=e^{-ia_k\cdot \Psi}\cdot
      P_k((\phi'-\phi)_{\mathrm{high}}).
    \end{cases}
  \end{equation*}
  Using \eqref{ro60} we have
  \begin{equation*}
    \sum_{k\in\Z}\|e^{-ia_k\cdot\Psi}\cdot P_k((\phi'-\phi)_{\mathrm{high}})\|^2_{\widetilde{H}^0}\leq C\|(\phi'-\phi)_{\mathrm{high}}\|^2_{H^0}.
  \end{equation*}
  Using Proposition \ref{Lemmab1} and \ref{Lemmab3} it follows that
  \begin{equation*}
    \sum_{k\in\Z}\|v'_k-v_k\|^2_{\F^0(T)}\leq C\|(\phi'-\phi)_{\mathrm{high}}\|^2_{H^0}+\sum_{k\in\Z}\|R'_k-R_k\|^2_{\N^0(T)}.
  \end{equation*}
  Using \eqref{toshow2} with $T'=T$ and \eqref{ro62},
  \begin{equation*}
    \sum_{k\in\Z}\|R'_k-R_k\|^2_{\N^0(T)}\leq C\ep_0^2\sum_{k\in\Z}\|v'_k-v_k\|^2_{\F^0(T)}.
  \end{equation*}
  It follows from the last two inequalities that
  \begin{equation*}
    \sum_{k\in\Z}\|v'_k-v_k\|^2_{\F^0(T)}\leq C\|\phi'-\phi\|^2_{L^2},
  \end{equation*}
  provided that $\ep_0$ is sufficiently small. Using \eqref{hh80} it
  follows that
  \begin{equation}\label{ro70}
    \sum_{k\in\Z}\|v'_k(t)-v_k(t)\|^2_{L^2}\leq C\|\phi'-\phi\|^2_{L^2}
  \end{equation}
  for any $t\in(-T,T)$. As in the proof of Proposition \ref{Lemmaw2},
  it follows from \eqref{ro70} and \eqref{rh50} that
  $\|u'(t)-u(t)\|_{L^2}^2\leq C\|\phi'-\phi\|^2_{L^2}$, which
  completes the proof of Proposition \ref{Lemmaw3}.
\end{proof}

In view of the results in section \ref{estimates}, the main theorem
follows from Propositions \ref{Lemmaw2} and \ref{Lemmaw3}. Thus it
remains to prove Proposition \ref{Lemmat2}, which is the goal of the
rest of the paper.

\section{Localized $L^2$ estimates}\label{locL2}

In this section we prove the localized $L^2$ estimates in Corollary
\ref{Lemmad2}. These estimates are similar to the estimates proved in
\cite[Section 6]{IoKe} in the study of the Benjamin-Ono equation. More
general $L^2$ estimates of this type can be found in \cite{Tao4}.

For $\xi_1,\xi_2\in\mathbb{R}$ let
\begin{equation}\label{om1}
  \Omega(\xi_1,\xi_2)=-\omega(\xi_1+\xi_2)+\omega(\xi_1)+\omega(\xi_2),
\end{equation}
where, as before, $\omega(\xi)=-\xi|\xi|^\alpha$. For compactly
supported functions $f,g,h\in L^2(\mathbb{R}\times\mathbb{R})$ let
\begin{equation}\label{om2}
  J(f,g,h)=\int_{\mathbb{R}^4}f(\xi_1,\mu_1)g(\xi_2,\mu_2)h(\xi_1+\xi_2,\mu_1+\mu_2+\Omega(\xi_1,\xi_2))\,d\xi_1d\xi_2d\mu_1d\mu_2.
\end{equation}
Given a triplet of real numbers $(\alpha_1, \alpha_2,\alpha_3)$ let
$\mathrm{min}\,(\alpha_1, \alpha_2,\alpha_3)$,
$\mathrm{max}\,(\alpha_1, \alpha_2,\alpha_3)$, and
$\mathrm{med}\,(\alpha_1, \alpha_2,\alpha_3)$ denote the minimum, the
maximum, and the median (more precisely, $\mathrm{med}\,(\alpha_1,
\alpha_2,\alpha_3)=\alpha_1+\alpha_2+\alpha_3-\mathrm{max}\,(\alpha_1,
\alpha_2,\alpha_3)-\mathrm{min}\,(\alpha_1, \alpha_2,\alpha_3)$) of
the numbers $\alpha_1$, $\alpha_2$, and $\alpha_3$. We define the sets
$U_k$, $k\in\Z$,
\begin{equation}\label{om0.1}
  \begin{cases}
    &U_k=\{\nu\in\R:|\nu|\in[(5n_{k-1}+n_k)/6,(5n_{k+1}+n_k)/6]\}\text{ if }k\in[1,\infty)\cap\Z;\\
    &U_k=\{\nu\in\R:|\nu|\in[2^{k+1},2^{k+3}]\}\text{ if
    }k\in(-\infty,0]\cap\Z.
  \end{cases}
\end{equation}
Clearly $U_k=I_k\cup I_{-k}$ if $k\geq 1$ and
$U_k=\widetilde{J}_{k+2}$ if $k\leq 0$. For $k_1,k_2,k_3\in\Z$ let
\begin{equation}\label{hj100}
  d_\al(k_1,k_2;k_3)=\inf\{\big\|\xi_1|^\al-|\xi_2|^\al\big|:\xi_1\in U_{k_1}, \xi_2\in U_{k_2}, \xi_1+\xi_2\in U_{k_3}\}.
\end{equation}

\begin{lemma}\label{Lemmad1}
  Assume $k_1,k_2,k_3\in\mathbb{Z}$, $j_1,j_2,j_3\in\mathbb{Z}_+$, and
  $f_{k_i}^{j_i}\in L^2(\mathbb{R}\times\mathbb{R})$ are functions
  supported in $U_{k_i}\times J_{j_i}$, $i=1,2,3$.

  (a) For any $k_1,k_2,k_3\in\mathbb{Z}$ and
  $j_1,j_2,j_3\in\mathbb{Z}_+$,
  \begin{equation}\label{om31}
    |J(f_{k_1}^{j_1},f_{k_2}^{j_2},f_{k_3}^{j_3})|\leq C\mathrm{min}\,( |U_{k_1}|,|U_{k_2}|,|U_{k_3}| )^{1/2}2^{\mathrm{min}\,(j_1,j_2,j_3)/2}
    \prod_{i=1}^3\|f_{k_i}^{j_i}\|_{L^2}.
  \end{equation}

  (b) Assume that $\{i_1,i_2,i_3\}$ is a permutation of $\{1,2,3\}$.
  Then
  \begin{equation}\label{om32}
    |J(f_{k_1}^{j_1},f_{k_2}^{j_2},f_{k_3}^{j_3})|\leq
    C2^{(j_1+j_2+j_3)/2}[2^{j_{i_3}}d_\al(k_{i_1},k_{i_2};k_{i_3})]^{-1/2}\prod_{i=1}^3\|f_{k_i}^{j_i}\|_{L^2}.
  \end{equation}

  (c) For any $k_1,k_2,k_3\in\mathbb{Z}$ and
  $j_1,j_2,j_3\in\mathbb{Z}_+$,
  \begin{equation}\label{om3}
    |J(f_{k_1}^{j_1},f_{k_2}^{j_2},f_{k_3}^{j_3})|\leq
    C2^{\mathrm{min}\,(j_1,j_2,j_3)/2+\mathrm{med}\,(j_1,j_2,j_3)/4}\prod_{i=1}^3\|f_{k_i}^{j_i}\|_{L^2}.
  \end{equation}
\end{lemma}

\begin{proof}[Proof of Lemma \ref{Lemmad1}] Let
  $A_{k_i}(\xi)=\Big[\int_{\mathbb{R}}|f_{k_i}^{j_i}(\xi,\mu)|^2\,d\mu\Big]^{1/2}$,
  $i=1,2,3$. Using the H\"{o}lder inequality and the support
  properties of the functions $f_{k_i}^{j_i}$,
  \begin{equation}\label{ela3}
    \begin{split}
      |J(f_{k_1}^{j_1},f_{k_2}^{j_2},f_{k_3}^{j_3})|&\leq
      C2^{\mathrm{min}\,(j_1,j_2,j_3)/2}\int_{\mathbb{R}^2}A_{k_1}(\xi_1)A_{k_2}(\xi_2)
      A_{k_3}(\xi_1+\xi_2)\,d\xi_1d\xi_2\\
      &\leq C\mathrm{min}\,( |U_{k_1}|,|U_{k_2}|,|U_{k_3}|
      )^{1/2}2^{\mathrm{min}\,(j_1,j_2,j_3)/2}
      \prod_{i=1}^3\|f_{k_i}^{j_i}\|_{L^2},
    \end{split}
  \end{equation}
  which is part (a).

  For part (b), using simple changes of variables and the fact that
  $\omega$ is odd,
  \begin{equation}\label{ela1}
    |J(f,g,h)|=|J(g,f,h)|\text{ and }|J(f,g,h)|=|J(\widetilde{f},h,g)|,
  \end{equation}
  where $\widetilde{f}(\xi,\mu)=f(-\xi,-\mu)$. Thus, by symmetry, in
  proving \eqref{om32} we may assume that $i_1=1$, $i_2=2$, and
  $i_3=3$. Let
  \begin{equation*}
    B_{k_3}(\xi,\mu)=\Big[\frac{1}{2^{j_1}2^{j_2}}\int_{\mathbb{R}^2}
    |f_{k_3}^{j_3}(\xi,\mu+\alpha+\beta)|^2(1+\alpha/2^{j_1})^{-2}(1+\beta/2^{j_2})^{-2}\,d\alpha
    d\beta\Big]^{1/2}.
  \end{equation*}
  Clearly,
  \begin{equation}\label{om6}
    \|B_{k_3}\|_{L^2}=C\|f_{k_3}^{j_3}\|_{L^2}\text{ and }B_{k_3}\text{ is supported in }U_{k_3}\times\mathbb{R}.
  \end{equation}
  Also, using H\"{o}lder's inequality,
  \begin{equation}\label{om7}
    \begin{split}
      &|J(f_{k_1}^{j_1},f_{k_2}^{j_2},f_{k_3}^{j_3})|\\
      &\leq
      C2^{(j_1+j_2)/2}\int_{\mathbb{R}^2}A_{k_1}(\xi_1)A_{k_2}(\xi_2)B_{k_3}(\xi_1+\xi_2,\Omega(\xi_1,\xi_2))\,d\xi_1d\xi_2\\
      &\leq
      C2^{(j_1+j_2)/2}\|A_{k_1}\|_{L^2}\|A_{k_2}\|_{L^2}\Big[\int_{U_{k_1}\times
        U_{k_2}}|B_{k_3}(\xi_1+\xi_2,\Omega(\xi_1,\xi_2))|^2\,d\xi_1d\xi_2\Big]^{1/2}.
    \end{split}
  \end{equation}
  Thus, for \eqref{om32}, it suffices to prove that
  \begin{equation}\label{om200}
    \Big[\int_{U_{k_1}\times U_{k_2}}|B_{k_3}(\xi_1+\xi_2,\Omega(\xi_1,\xi_2))|^2\,d\xi_1d\xi_2\Big]^{1/2}\leq C[d_\al(k_1,k_2;k_3)]^{-1/2}\|B_{k_3}\|_{L^2}.
  \end{equation}
  Let $B'_{k_3}(\xi,\mu)=B_{k_3}(\xi,-\omega(\xi)+\mu)$,
  $\|B'_{k_3}\|_{L^2}=\|B_{k_3}\|_{L^2}$, $B'_{k_3}$ supported in
  $U_{k_3}\times\mathbb{R}$. For \eqref{om200} it suffices to prove
  that
  \begin{equation}\label{om201}
    \Big[\int_{U_{k_1}\times U_{k_2}}|B'_{k_3}(\xi_1+\xi_2,\omega(\xi_1)+\omega(\xi_2))|^2\,d\xi_1d\xi_2\Big]^{1/2}\leq C[d_\al(k_1,k_2;k_3)]^{-1/2}\|B'_{k_3}\|_{L^2}.
  \end{equation}
  We observe now that $|\omega'(\xi_1)-\omega'(\xi_2)|\geq
  C^{-1}d_\al(k_1,k_2;k_3)$ if $\xi_1\in U_{k_1}$, $\xi_2\in U_{k_2}$,
  and $\xi_1+\xi_2\in U_{k_3}$. The bound \eqref{om201} follows.

  For part (c), using part (a), we may assume
  \begin{equation}\label{ela2}
    2^{\mathrm{med}\,(j_1,j_2,j_3)/2}\leq C^{-1}\mathrm{min}\,( |U_{k_1}|, |U_{k_2}|, |U_{k_3}| ).
  \end{equation}
  Using \eqref{ela1}, we may also assume
  $j_1=\mathrm{min}\,(j_1,j_2,j_3)$ and
  $j_2=\mathrm{med}\,(j_1,j_2,j_3)$. Let
  \begin{equation*}
    R_{j_2}=\{(\xi_1,\xi_2):|\xi_1-\xi_2|\geq 2^{j_2/2}\}.
  \end{equation*}
  For the integral over
  $(\xi_1,\xi_2)\in{}^c\negmedspace\,R_{j_2}=\mathbb{R}^2\setminus
  R_{j_2}$ we use a bound similar to \eqref{ela3}:
  \begin{equation*}
    \begin{split}
      \Big|\int_{{}^c\negmedspace R_{j_2}\times\mathbb{R}^2}&f_{k_1}^{j_1}(\xi_1,\mu_1)f_{k_2}^{j_2}(\xi_2,\mu_2)f_{k_3}^{j_3}(\xi_1+\xi_2,\mu_1+\mu_2+\Omega(\xi_1,\xi_2))\,d\xi_1d\xi_2d\mu_1d\mu_2\Big|\\
      &\leq C2^{j_1/2}\int_{{}^c\negmedspace R_{j_2}}A_{k_1}(\xi_1)A_{k_2}(\xi_2)A_{k_3}(\xi_1+\xi_2)\,d\xi_1d\xi_2 \\
      &\leq C2^{j_1/2}\iint_{|\mu|\leq 2^{j_2/2}}A_{k_1}(\xi_2+\mu)A_{k_2}(\xi_2)A_{k_3}(2\xi_2+\mu)\,d\xi_2d\mu\\
      &\leq C2^{j_1/2}\int_{|\mu|\leq 2^{j_2/2}}\Big(\int_\mathbb{R}|A_{k_1}(\xi_2+\mu)|^2|A_{k_2}(\xi_2)|^2\,d\xi_2\Big)^{1/2}\|A_{k_3}\|_{L^2}d\mu\\
      &\leq
      C2^{j_1/2}2^{j_2/4}\|A_{k_1}\|_{L^2}\|A_{k_2}\|_{L^2}\|A_{k_3}\|_{L^2},
    \end{split}
  \end{equation*}
  which suffices for \eqref{om3}. For the integral over
  $(\xi_1,\xi_2)\in R_{j_2}$ we use a bound similar to \eqref{om7}
  \begin{equation}\label{ela4}
    \begin{split} &\Big|\int\limits_{R_{j_2}\times\mathbb{R}^2}f_{k_1}^{j_1}(\xi_1,\mu_1)f_{k_2}^{j_2}(\xi_2,\mu_2)f_{k_3}^{j_3}(\xi_1+\xi_2,\mu_1+\mu_2+\Omega(\xi_1,\xi_2))\,d\xi_1d\xi_2d\mu_1d\mu_2\Big|\\
      \leq&
      C2^{(j_1+j_2)/2}\|A_{k_1}\|_{L^2}\|A_{k_2}\|_{L^2}
      \Big[\int\limits_{R_{j_2}\cap
        (U_{k_1}\times
        U_{k_2})}|B_{k_3}(\xi_1+\xi_2,\Omega(\xi_1,\xi_2))|^2\,d\xi_1d\xi_2\Big]^{1/2}.
    \end{split}
  \end{equation}
  Using \eqref{ela2} and the identity
  $\omega'(\xi)=-(\alpha+1)|\xi|^\alpha$, we observe that
  \begin{equation*}
    |\omega'(\xi_1)-\omega'(\xi_2)|\geq C^{-1}2^{j_2/2}\text{ if }(\xi_1,\xi_2)\in R_{j_2}\cap (U_{k_1}\times U_{k_2})\text{ and }\xi_1+\xi_2\in U_{k_3}.
  \end{equation*}
  As before, it follows that
  \begin{equation*}
    \Big[\int_{R_{j_2}\cap (U_{k_1}\times U_{k_2})}|B_{k_3}(\xi_1+\xi_2,\Omega(\xi_1,\xi_2))|^2\,d\xi_1d\xi_2\Big]^{1/2}\leq C2^{-j_2/4}\|B_{k_3}\|_{L^2}.
  \end{equation*}
  The bound \eqref{om3} follows by substituting this bound into
  \eqref{ela4}.
\end{proof}

We restate now Lemma \ref{Lemmad1} in a form that is suitable for the
bilinear estimates in the next section. For $k\in\Z$ and $j\in\Z_+$
let $V_k^j=\{(\xi,\tau)\in\R\times\R:\xi\in U_k\text{ and
}\tau-\omega(\xi)\in J_j\}$.

\begin{corollary}\label{Lemmad2}
  Assume $k_1,k_2,k_3\in\mathbb{Z}$, $j_1,j_2,j_3\in\mathbb{Z}_+$, and
  $f_{k_i}^{j_i}\in L^2(\mathbb{R}^2)$ are functions supported in
  $V_{k_i}^{j_i}$, $i=1,2$.

  (a) For any $k_1,k_2,k_3\in\mathbb{Z}$ and
  $j_1,j_2,j_3\in\mathbb{Z}_+$,
  \begin{equation}\label{on31}
    \|\mathbf{1}_{V_{k_3}^{j_3}}\cdot (f_{k_1}^{j_1}\ast f_{k_2}^{j_2})\|_{L^2}\leq C
    \mathrm{min}\,( |U_{k_1}|,|U_{k_2}|,|U_{k_3}| )^{1/2}2^{\mathrm{min}\,(j_1,j_2,j_3)/2}
    \prod_{i=1}^2\|f_{k_i}^{j_i}\|_{L^2}.
  \end{equation}

  (b) If $\{i_1,i_2,i_3\}$ is a permutation of $\{1,2,3\}$ then
  \begin{equation}\label{on32}
    \|\mathbf{1}_{V_{k_3}^{j_3}}\cdot (f_{k_1}^{j_1}\ast f_{k_2}^{j_2})\|_{L^2}\leq
    C2^{(j_1+j_2+j_3)/2}[2^{j_{i_3}}d_\al(k_{i_1},k_{i_2};k_{i_3})]^{-1/2}\prod_{i=1}^2\|f_{k_i}^{j_i}\|_{L^2}.
  \end{equation}

  (c) For any $k_1,k_2,k_3\in\mathbb{Z}$ and
  $j_1,j_2,j_3\in\mathbb{Z}_+$,
  \begin{equation}\label{on3}
    \|\mathbf{1}_{V_{k_3}^{j_3}}\cdot (f_{k_1}^{j_1}\ast f_{k_2}^{j_2})\|_{L^2}\leq
    C2^{\mathrm{min}\,(j_1,j_2,j_3)/2+\mathrm{med}\,(j_1,j_2,j_3)/4}\prod_{i=1}^2\|f_{k_i}^{j_i}\|_{L^2}.
  \end{equation}
\end{corollary}

\begin{proof}[Proof of Corollary \ref{Lemmad2}] Clearly,
  \begin{equation*}
    \|\mathbf{1}_{V_{k_3}^{j_3}}\cdot (f_{k_1}^{j_1}\ast f_{k_2}^{j_2})\|_{L^2}=\sup_{\|f\|_{L^2}=1}\Big|\int_{V_{k_3}^{j_3}}f\cdot(f_{k_1}^{j_1}\ast f_{k_2}^{j_2})\,d\xi d\tau\Big|.
  \end{equation*}
  Let $f_{k_3}^{j_3}=\mathbf{1}_{V_{k_3}^{j_3}}\cdot f$, and then
  $g_{k_i}^{j_i}(\xi,\mu)=f_{k_i}^{j_i}(\xi,\mu+\omega(\xi))$,
  $i=1,2,3$. The functions $g_{k_i}^{j_i}$ are supported in
  $U_{k_i}\times J_{j_i}$,
  $\|g_{k_i}^{j_i}\|_{L^2}=\|f_{k_i}^{j_i}\|_{L^2}$, and, using simple
  changes of variables,
  \begin{equation*}
    \int_{V_{k_3}^{j_3}}f\cdot(f_{k_1}^{j_1}\ast f_{k_2}^{j_2})\,d\xi d\tau=J(g_{k_1}^{j_1},g_{k_2}^{j_2},g_{k_3}^{j_3}).
  \end{equation*}
  Corollary \ref{Lemmad2} follows from Lemma \ref{Lemmad1}.
\end{proof}

\section{Bilinear estimates}\label{Dyadic2}

In this section we prove several $L^2$-based bilinear estimates. All
of our estimates are based on Corollary \ref{Lemmad2}. For
$\rho\in[-1,1]$ we define the family of normed spaces $X_0^\rho$,
\begin{equation}\label{ro80}
  \begin{split}
    X_0^\rho=\{&f\in L^2:\, f \text{ supported in }I_0\times\mathbb{R}\text{ and }\\
    &\|f\|_{X_0^\rho}=\sum_{j=0}^\infty\sum_{k'=-\infty}^22^{j(1-\delta)}2^{\rho
      k'}\|\eta_j(\tau-\omega(\xi))\widetilde{\eta}_{k'}(\xi)f(\xi,\tau)\,\|_{L^2_{\xi,\tau}}<\infty\}.
  \end{split}
\end{equation}
Clearly, $X_0^{-1}=X_0$ (compare with the definition \eqref{def1''})
and $X_0^{\rho}\hookrightarrow X_0^{\rho'}$ if $\rho\leq \rho'$. In
addition, it follows easily that
\begin{equation}\label{ro81}
  X_0^{-1}\hookrightarrow Z_0\hookrightarrow X_0^{-1/2+\delta}.
\end{equation}

For $j\in\Z_+$ and $k\in\Z\setminus \{0\}$ we define the sets
\begin{equation}\label{DL}
  D_{k}^j=\{(\xi,\tau)\in\R\times\R: \xi\in I_k\text{ and }\tau-\omega(\xi)\in J_j\}\subseteq V_{|k|}^j.
\end{equation}
For $j\in\Z_+$ and $k'\in(-\infty,2]\cap\Z$ we define the sets
\begin{equation}\label{DS}
  D_{0,k'}^j=\{(\xi,\tau)\in\R\times\R: \xi\in I_0\cap \widetilde{J}_{k'}\text{ and }\tau-\omega(\xi)\in J_j\}\subseteq V_{k'-2}^j.
\end{equation} 
Using the definition, if $|k|\geq 1$ and $f_k\in Z_k$ then $f_k$ can
be written in the form
\begin{equation}\label{repr1}
  \begin{cases}
    &f_k=\sum\limits_{j=0}^\infty f_{k}^{j};\\
    &\sum\limits_{j=0}^\infty
    2^{j/2}\beta_{k,j}\|f_{k}^{j}\|_{L^2}=\|f_k\|_{Z_k},
  \end{cases}
\end{equation}
such that $f_{k}^{j}$ is supported in $D_{k}^j$. If $f_0\in X_0^\rho$
then $f_0$ can be written in the form
\begin{equation}\label{repr2}
  \begin{cases}
    &f_0=\sum\limits_{j=0}^\infty \sum\limits_{k'=-\infty}^2f^{j}_{0,k'};\\
    &\sum\limits_{j=0}^\infty \sum\limits_{k'=-\infty}^2
    2^{j(1-\delta)}2^{\rho
      k'}\|f^{j}_{0,k'}\|_{L^2}=\|f_0\|_{X_0^\rho},
  \end{cases}
\end{equation}
such that $f^{j}_{0,k'}$ is supported in $D_{0,k'}^{j}$. The
identities \eqref{repr1} and \eqref{repr2} are our main atomic
decompositions of functions in $Z_k$, $k\geq 1$, and $X_0^\rho$.

We consider first $\mathrm{Low}\times\mathrm{High}\to\mathrm{High}$
interactions.

\begin{lemma}\label{Lemmak1}
  (a) Assume $k,k_1,k_2\in\Z\setminus\{0\}$, $|n_k|\geq 2^{20}$,
  $|n_{k_1}|\leq |n_k|/ 2^{10}$, $f_{k_1}\in Z_{k_1}$, and $f_{k_2}\in
  Z_{k_2}$. Then
  \begin{equation}\label{hj1}
    \begin{split}
      (1+|n_k|)\cdot &\|\chi_k(\xi)(\tau-\omega(\xi)+i)^{-1}\cdot (f_{k_1}\ast f_{k_2})\|_{Z_k}\\
      &\leq C(1+|n_{k_1}| )^{-1/2}(1+|n_k| )^{-\delta}\cdot
      \|f_{k_1}\|_{Z_{k_1}}\|f_{k_2}\|_{Z_{k_2}}.
    \end{split}
  \end{equation}

  (b) Assume $k,k_2\in\Z\setminus\{0\}$, $|n_k|\geq 2^{20}$,
  $f_{k_2}\in Z_{k_2}$, and $f_{0}\in X_{0}^\rho$,
  $\rho\in\{-1/2+\delta,\delta\}$. Then
  \begin{equation}\label{hj1.2}
    \begin{split}
      (1+|n_k|)^{1/2-\rho+\delta}\cdot &\|\chi_k(\xi)(\tau-\omega(\xi)+i)^{-1}\cdot (f_{0}\ast f_{k_2})\|_{Z_k}\\
      &\leq C(1+|n_k| )^{-\delta}\cdot
      \|f_{0}\|_{X^\rho_0}\|f_{k_2}\|_{Z_{k_2}}.
    \end{split}
  \end{equation}
\end{lemma}

\begin{proof}[Proof of Lemma \ref{Lemmak1}] For part (a), we may
  assume $|n_{k_2}-n_k|\leq |n_k|/2^5$. Using \eqref{repr1}, we may
  assume $f_{k_1}=f_{k_1}^{j_1}$ is supported in $D_{k_1}^{j_1}$ and
  $f_{k_2}=f_{k_2}^{j_2}$ is supported in $D_{k_2}^{j_2}$. For
  \eqref{hj1} it suffices to prove that
  \begin{equation}\label{hj2}
    \begin{split}
      &|n_k|\cdot \sum_{j\in\Z_+}2^{-j/2}\beta_{k,j}\|\mathbf{1}_{D_k^j}\cdot (f_{k_1}^{j_1}\ast f_{k_2}^{j_2})\|_{L^2}\\
      &\leq C|n_{k_1}|^{-1/2}|n_k|^{-\delta}\cdot
      2^{j_1/2}\beta_{k_1,j_1}\|f_{k_1}^{j_1}\|_{L^2}\cdot
      2^{j_2/2}\beta_{k_2,j_2}\|f_{k_2}^{j_2}\|_{L^2}.
    \end{split}
  \end{equation}
  Using \eqref{on32} and \eqref{hj100},
  \begin{equation}\label{hj3}
    \|\mathbf{1}_{D_k^j}\cdot (f_{k_1}^{j_1}\ast f_{k_2}^{j_2})\|_{L^2}\leq C2^{(j_1+j_2+j)/2}[(2^j+2^{j_2})|n_k|^{\alpha}+2^{j_1}|n_{k_1}| |n_{k}|^{\alpha-1}]^{-1/2}\Pi,
  \end{equation} 
  where $\Pi=\|f_{k_1}^{j_1}\|_{L^2}\cdot \|f_{k_2}^{j_2}\|_{L^2}$.
  Let $\Pi'=2^{j_1/2}\beta_{k_1,j_1}\|f_{k_1}^{j_1}\|_{L^2}\cdot
  2^{j_2/2}\beta_{k_2,j_2}\|f_{k_2}^{j_2}\|_{L^2}$. Using \eqref{hj3},
  if $j=\max(j_1,j_2,j)$ then
  \begin{equation}\label{hj4}
    |n_k|\cdot 2^{-j/2}\beta_{k,j}\|\mathbf{1}_{D_k^j}\cdot (f_{k_1}^{j_1}\ast f_{k_2}^{j_2})\|_{L^2}\leq C\frac{2^{-j/2}\beta_{k,j}}{\beta_{k_1,j_1}\beta_{k_2,j_2}}|n_k|^{1-\alpha/2}\cdot \Pi'.
  \end{equation}
  If $j_2=\max(j_1,j_2,j)$ then
  \begin{equation}\label{hj5}
    |n_k|\cdot 2^{-j/2}\beta_{k,j}\|\mathbf{1}_{D_k^j}\cdot (f_{k_1}^{j_1}\ast f_{k_2}^{j_2})\|_{L^2}\leq C\frac{\beta_{k,j}}{2^{j_2/2}\beta_{k_1,j_1}\beta_{k_2,j_2}}|n_k|^{1-\alpha/2}\cdot \Pi'.
  \end{equation}
  If $j_1=\max(j_1,j_2,j)$ then
  \begin{equation}\label{hj6}
    |n_k|\cdot 2^{-j/2}\beta_{k,j}\|\mathbf{1}_{D_k^j}\cdot (f_{k_1}^{j_1}\ast f_{k_2}^{j_2})\|_{L^2}\leq C\frac{\beta_{k,j}}{2^{j_1/2}\beta_{k_1,j_1}\beta_{k_2,j_2}}|n_{k_1}|^{-1/2}|n_k|^{1-(\alpha-1)/2}\cdot \Pi'.
  \end{equation}

  We observe now that for any $\xi_1,\xi_2\in\mathbb{R}$
  \begin{equation}\label{om20}
    \frac{|\Omega(\xi_1,\xi_2)|}{\mathrm{min}\,(|\xi_1|,|\xi_2|,|\xi_1+\xi_2|)\cdot
      \mathrm{max}\,( |\xi_1|^\alpha ,|\xi_2|^\alpha ,|\xi_1+\xi_2|^\alpha )}\in[2^{-4},2^{4}].
  \end{equation}
  Thus, by examining the supports of the functions,
  $\mathbf{1}_{D_k^j}\cdot (f_{k_1}^{j_1}\ast f_{k_2}^{j_2})\equiv 0$
  unless
  \begin{equation}\label{new2}
    2^{\mathrm{max}\,(j_1,j_2,j)}\geq C^{-1}|n_{k_1}| |n_k|^\alpha.
  \end{equation}
  Thus, using \eqref{hj4} and \eqref{def1'},
  \begin{equation*}
    \begin{split}
      &|n_k|\cdot \sum_{j\geq \max(j_1,j_2)}2^{-j/2}\beta_{k,j}\|\mathbf{1}_{D_k^j}\cdot (f_{k_1}^{j_1}\ast f_{k_2}^{j_2})\|_{L^2}\\
      &\leq C\sum_{2^j\geq C^{-1}|n_{k_1}|
        |n_k|^\alpha}\frac{2^{-j/2}\beta_{k,j}}{\beta_{k_1,j_1}\beta_{k_2,j_2}}|n_k|^{1-\alpha/2}\cdot
      \Pi'\leq C|n_{k_1}|^{-1/2}|n_k|^{1-\alpha}\cdot \Pi'.
    \end{split}
  \end{equation*}
  If $j_2\geq j_1$ then, using \eqref{hj5}, \eqref{new2}, and
  \eqref{def1'}
  \begin{equation*}
    \begin{split}
      &|n_k|\cdot \sum_{j\leq j_2}2^{-j/2}\beta_{k,j}\|\mathbf{1}_{D_k^j}\cdot (f_{k_1}^{j_1}\ast f_{k_2}^{j_2})\|_{L^2}\\
      &\leq C\sum_{j\leq
        j_2}\frac{\beta_{k,j}}{2^{j_2/2}\beta_{k_1,j_1}\beta_{k_2,j_2}}|n_k|^{1-\alpha/2}\cdot
      \Pi'\leq C|n_{k_1}|^{-1/2}|n_k|^{1-\alpha}\ln(2+|n_k| )\cdot
      \Pi'.
    \end{split}
  \end{equation*}
  Finally, if $j_1\geq j_2$ then, using \eqref{hj6}, \eqref{new2}, and
  \eqref{def1'}
  \begin{equation*}
    \begin{split}
      |n_k|\cdot \sum_{j\leq j_1}2^{-j/2}&\beta_{k,j}\|\mathbf{1}_{D_k^j}\cdot (f_{k_1}^{j_1}\ast f_{k_2}^{j_2})\|_{L^2}\\
      &\leq C\sum_{j\leq j_1}\frac{\beta_{k,j}}{2^{j_1/2}\beta_{k_1,j_1}\beta_{k_2,j_2}}|n_{k_1}|^{-1/2}|n_k|^{1-(\alpha-1)/2}\cdot \Pi'\\
      &\leq C|n_{k_1}|^{-1/2}|n_k|^{1-\alpha}\ln(2+|n_k| )\cdot \Pi'.
    \end{split}
  \end{equation*}
  The estimate \eqref{hj2} follows from the last three bounds.
  \medskip
 
  For part (b), using \eqref{repr1} and \eqref{repr2} we may assume
  $f_{0}=f_{0,k'}^{j_1}$ is supported in $D_{0,k'}^{j_1}$ and
  $f_{k_2}=f_{k_2}^{j_2}$ is supported in $D_{k_2}^{j_2}$. For
  \eqref{hj1.2} it suffices to prove that
  \begin{equation}\label{hj20}
    \begin{split}
      &(|n_k|+2^{-k'/2}|n_k|^{1/2}|n_k|^{-19\delta})\cdot \sum_{j\in\Z_+}2^{-j/2}\beta_{k,j}\|\mathbf{1}_{D_k^j}\cdot (f_{0,k'}^{j_1}\ast f_{k_2}^{j_2})\|_{L^2}\\
      &\leq C|n_k|^{-20\delta}\cdot
      2^{j_1(1-\delta)}2^{-k'(1/2-\delta)
      }\|f_{0,k'}^{j_1}\|_{L^2}\cdot
      2^{j_2/2}\beta_{k_2,j_2}\|f_{k_2}^{j_2}\|_{L^2}.
    \end{split}
  \end{equation}
  Using \eqref{on31},
  \begin{equation}\label{hj25}
    \|\mathbf{1}_{D_k^j}\cdot (f_{0,k'}^{j_1}\ast f_{k_2}^{j_2})\|_{L^2}\leq C2^{j_2/2}2^{k'/2}\cdot \|f_{0,k'}^{j_1}\|_{L^2}\cdot \|f_{k_2}^{j_2}\|_{L^2},
  \end{equation}
  and the bound \eqref{hj20} follows easily if
  $2^{k'}|n_k|^{1+30\delta}\leq 1$.

  Assume that
  \begin{equation}\label{hj34}
    2^{k'}|n_k|^{1+30\delta} \geq  1.
  \end{equation}
  In this case $|n_k|\geq2^{-k'/2}|n_k|^{1/2}|n_k|^{-19\delta}$. Using
  \eqref{on32} and \eqref{hj100},
  \begin{equation}\label{hj21}
    \|\mathbf{1}_{D_k^j}\cdot (f_{0,k'}^{j_1}\ast f_{k_2}^{j_2})\|_{L^2}\leq C2^{(j_1+j_2+j)/2}[(2^j+2^{j_2})|n_k|^{\alpha}]^{-1/2}\Pi,
  \end{equation} 
  where $\Pi=\|f_{0,k'}^{j_1}\|_{L^2}\cdot \|f_{k_2}^{j_2}\|_{L^2}$.
  Using \eqref{om20}, $\mathbf{1}_{D_k^j}\cdot (f_{0,k'}^{j_1}\ast
  f_{k_2}^{j_2})\equiv 0$ unless
  \begin{equation}\label{new3}
    2^{\mathrm{max}\,(j_1,j_2,j)}\geq C^{-1}2^{k'}|n_k|^\alpha.
  \end{equation}
  Let
  $\Pi'=2^{j_1(1-\delta)}2^{-k'(1/2-\delta)}\|f_{0,k'}^{j_1}\|_{L^2}\cdot
  2^{j_2/2}\beta_{k_2,j_2}\|f_{k_2}^{j_2}\|_{L^2}$. If
  $j=\max(j_1,j_2,j)$ then, using \eqref{hj21},
  \begin{equation}\label{hj40}
    |n_k|\cdot 2^{-j/2}\beta_{k,j}\|\mathbf{1}_{D_k^j}\cdot (f_{0,k'}^{j_1}\ast f_{k_2}^{j_2})\|_{L^2}\leq C\frac{2^{-j/2}\beta_{k,j}}{2^{j_1(1/2-\delta)}\beta_{k_2,j_2}}2^{k'(1/2-\delta)}|n_k|^{1-\alpha/2}\cdot \Pi'.
  \end{equation}
  If $j_2=\max(j_1,j_2,j)$ then, using \eqref{hj21},
  \begin{equation}\label{hj50}
    |n_k|\cdot 2^{-j/2}\beta_{k,j}\|\mathbf{1}_{D_k^j}\cdot (f_{0,k'}^{j_1}\ast f_{k_2}^{j_2})\|_{L^2}\leq C\frac{\beta_{k,j}}{2^{j_1(1/2-\delta)}2^{j_2/2}\beta_{k_2,j_2}}2^{k'(1/2-\delta)}|n_k|^{1-\alpha/2}\cdot \Pi'.
  \end{equation}
  If $j_1=\max(j_1,j_2,j)$ then, using \eqref{hj25}
  \begin{equation}\label{hj60}
    |n_k|\cdot 2^{-j/2}\beta_{k,j}\|\mathbf{1}_{D_k^j}\cdot (f_{0,k'}^{j_1}\ast f_{k_2}^{j_2})\|_{L^2}\leq C\frac{2^{-j/2}\beta_{k,j}}{2^{j_1(1-\delta)}\beta_{k_2,j_2}}2^{k'(1-\delta)}|n_k|\cdot \Pi'.
  \end{equation}
  Thus, using \eqref{hj40}, \eqref{new3}, and \eqref{def1'},
  \begin{equation*}
    \begin{split}
      &|n_k|\cdot \sum_{j\geq \max(j_1,j_2)}2^{-j/2}\beta_{k,j}\|\mathbf{1}_{D_k^j}\cdot (f_{0,k'}^{j_1}\ast f_{k_2}^{j_2})\|_{L^2}\\
      &\leq C\sum_{2^j\geq
        C^{-1}2^{k'}|n_k|^\alpha}\frac{2^{-j/2}\beta_{k,j}}{2^{j_1(1/2-\delta)}\beta_{k_2,j_2}}2^{k'(1/2-\delta)}|n_k|^{1-\alpha/2}\cdot
      \Pi'\leq C2^{-\delta k'}|n_k|^{1-\alpha}\cdot \Pi'.
    \end{split}
  \end{equation*}
  If $j_2\geq j_1$ then, using \eqref{hj50}, \eqref{new3}, and
  \eqref{def1'}
  \begin{equation*}
    \begin{split}
      |n_k|\cdot \sum_{j\leq j_2}2^{-j/2}\beta_{k,j}&\|\mathbf{1}_{D_k^j}\cdot (f_{0,k'}^{j_1}\ast f_{k_2}^{j_2})\|_{L^2}\\
      &\leq C\sum_{j\leq j_2}\frac{\beta_{k,j}}{2^{j_1(1/2-\delta)}2^{j_2/2}\beta_{k_2,j_2}}2^{k'(1/2-\delta)}|n_k|^{1-\alpha/2}\cdot \Pi'\\
      &\leq C2^{-\delta k'}|n_k|^{1-\alpha}\ln(2+|n_k| )\cdot \Pi'.
    \end{split}
  \end{equation*}
  Finally, if $j_1\geq j_2$ then, using \eqref{hj60}, \eqref{new3},
  and \eqref{def1'}
  \begin{equation*}
    \begin{split}
      &|n_k|\cdot \sum_{j\leq j_1}2^{-j/2}\beta_{k,j}\|\mathbf{1}_{D_k^j}\cdot (f_{0,k'}^{j_1}\ast f_{k_2}^{j_2})\|_{L^2}\\
      &\leq C\sum_{j\leq
        j_1}\frac{2^{-j/2}\beta_{k,j}}{2^{j_1(1-\delta)}\beta_{k_2,j_2}}2^{k'(1-\delta)}|n_k|\cdot\Pi'\leq
      C|n_k|^{1-\alpha+\alpha\delta}\cdot \Pi'.
    \end{split}
  \end{equation*}
  The estimate \eqref{hj20} follows from the last three bounds and
  \eqref{hj34}.
\end{proof}

We consider now $\mathrm{High}\times\mathrm{High}\to\mathrm{Low}$
interactions.

\begin{lemma}\label{Lemmak2}
  Assume $k_1,k_2,k\in\Z\setminus\{0\}$, $\min( |n_{k_1}|,|n_{k_2}|
  )\geq 2^{10}(1+|n_k| )$, $f_{k_1}\in Z_{k_1}$, and $f_{k_2}\in
  Z_{k_2}$.  Then
  \begin{equation}\label{hw1}
    \begin{split}
      (1+|n_k|)\cdot &\|\chi_k(\xi)(\tau-\omega(\xi)+i)^{-1}\cdot (f_{k_1}\ast f_{k_2})\|_{Z_k}\\
      &\leq C(1+|n_{k}| )^{-1/2}(1+ |n_{k_1}|+|n_{k_2}|
      )^{-\delta}\cdot \|f_{k_1}\|_{Z_{k_1}}\|f_{k_2}\|_{Z_{k_2}}.
    \end{split}
  \end{equation}
  In addition, if $\min( |n_{k_1}|,|n_{k_2}| )\geq 2^{10}$,
  $f_{k_1}\in Z_{k_1}$, and $f_{k_2}\in Z_{k_2}$ then
  \begin{equation}\label{hw1.6}
    \begin{split}
      \|\chi_0(\xi)(\tau-&\omega(\xi)+i)^{-1}\cdot (f_{k_1}\ast f_{k_2})\|_{X_0^{-1/2+\delta}}\\
      &\leq C(1+ |n_{k_1}|+|n_{k_2}| )^{-\delta}\cdot
      \|f_{k_1}\|_{Z_{k_1}}\|f_{k_2}\|_{Z_{k_2}}.
    \end{split}
  \end{equation}
\end{lemma}

\begin{proof}[Proof of Lemma \ref{Lemmak2}] Clearly, we may assume
  $|n_{k_1}/n_{k_2}|\in[1/2,2]$ and $n_{k_1}\cdot n_{k_2}<0$. Using
  \eqref{repr1}, we may assume $f_{k_1}=f_{k_1}^{j_1}$ is supported in
  $D_{k_1}^{j_1}$ and $f_{k_2}=f_{k_2}^{j_2}$ is supported in
  $D_{k_2}^{j_2}$.

  For \eqref{hw1} it suffices to prove that
  \begin{equation}\label{hw2}
    \begin{split}
      &|n_k|\cdot \sum_{j\in\Z_+}2^{-j/2}\beta_{k,j}\|\mathbf{1}_{D_k^j}\cdot (f_{k_1}^{j_1}\ast f_{k_2}^{j_2})\|_{L^2}\\
      &\leq C|n_{k}|^{-1/2}|n_{k_1}|^{-\delta}\cdot
      2^{j_1/2}\beta_{k_1,j_1}\|f_{k_1}^{j_1}\|_{L^2}\cdot
      2^{j_2/2}\beta_{k_2,j_2}\|f_{k_2}^{j_2}\|_{L^2}.
    \end{split}
  \end{equation}
  In view of \eqref{om20}, $\mathbf{1}_{D_k^j}\cdot (f_{k_1}^{j_1}\ast
  f_{k_2}^{j_2})\equiv 0$ unless
  \begin{equation}\label{hw3}
    2^{\mathrm{max}\,(j_1,j_2,j)}\geq C^{-1}|n_{k}| |n_{k_1}|^\alpha.
  \end{equation}
  Using \eqref{on32},
  \begin{equation}\label{hw4}
    \|\mathbf{1}_{D_k^j}\cdot (f_{k_1}^{j_1}\ast f_{k_2}^{j_2})\|_{L^2}\leq C2^{(j_1+j_2+j)/2}[2^{\max(j_1,j_2)}|n_{k_1}|^\alpha+2^j|n_k| |n_{k_1}|^{\alpha-1}]^{-1/2}\Pi,
  \end{equation}
  where $\Pi=\|f_{k_1}^{j_1}\|_{L^2}\cdot \|f_{k_2}^{j_2}\|_{L^2}$.
  Using \eqref{hw3} and \eqref{hw4},
  \begin{equation*}
    \begin{split}
      |n_k|\cdot &\sum_{j\geq \max(j_1,j_2)}2^{-j/2}\beta_{k,j}\|\mathbf{1}_{D_k^j}\cdot (f_{k_1}^{j_1}\ast f_{k_2}^{j_2})\|_{L^2}\\
      &\leq C\sum_{2^j\geq C^{-1}|n_{k}| |n_{k_1}|^\alpha}2^{-j/2}\beta_{k,j}\cdot |n_k|^{1/2}|n_{k_1}|^{-(\alpha-1)/2}\cdot 2^{(j_1+j_2)/2}\Pi\\
      &\leq C|n_k|^{-1/2}|n_{k_1}|^{-(\alpha-1)/2}\cdot
      2^{(j_1+j_2)/2}\Pi.
    \end{split}
  \end{equation*} 
  Using again \eqref{hw3} and \eqref{hw4},
  \begin{equation*}
    \begin{split}
      |n_k|\cdot &\sum_{j\leq  \max(j_1,j_2)}2^{-j/2}\beta_{k,j}\|\mathbf{1}_{D_k^j}\cdot (f_{k_1}^{j_1}\ast f_{k_2}^{j_2})\|_{L^2}\\
      &\leq C\sum_{j\leq \max(j_1,j_2) }\beta_{k,j}\cdot 2^{-\max(j_1,j_2)/2}|n_k\|n_{k_1}|^{-\alpha/2}\cdot 2^{(j_1+j_2)/2}\Pi\\
      &\leq C\log(2+|n_{k_1}| )|n_{k_1}|^{-\alpha/2}\cdot
      2^{(j_1+j_2)/2}\Pi,
    \end{split}
  \end{equation*} 
  since in this last estimate we may assume
  $2^{\mathrm{max}\,(j_1,j_2)}\geq C^{-1}|n_{k}| |n_{k_1}|^\alpha$.
  The bound \eqref{hw2} follows from the last two estimates.  \medskip

  For \eqref{hw1.6} it suffices to prove that
  \begin{equation}\label{hl2}
    \begin{split}
      \sum_{k'=-\infty}^2&\sum_{j\in\Z_+}2^{-j\delta }2^{-k'(1/2-\delta)}\|\mathbf{1}_{D_{0,k'}^j}\cdot (f_{k_1}^{j_1}\ast f_{k_2}^{j_2})\|_{L^2}\\
      &\leq C|n_{k_1}|^{-\delta}\cdot
      2^{j_1/2}\beta_{k_1,j_1}\|f_{k_1}^{j_1}\|_{L^2}\cdot
      2^{j_2/2}\beta_{k_2,j_2}\|f_{k_2}^{j_2}\|_{L^2}.
    \end{split}
  \end{equation}
  In view of \eqref{om20}, $\mathbf{1}_{D_{0,k'}^j}\cdot
  (f_{k_1}^{j_1}\ast f_{k_2}^{j_2})\equiv 0$ unless
  \begin{equation}\label{hl3}
    2^{\mathrm{max}\,(j_1,j_2,j)}\geq C^{-1}2^{k'} |n_{k_1}|^\alpha.
  \end{equation}
  Using \eqref{on31},
  \begin{equation}\label{hl4}
    \|\mathbf{1}_{D_{0,k'}^j}\cdot (f_{k_1}^{j_1}\ast f_{k_2}^{j_2})\|_{L^2}\leq C2^{k'/2}2^{-\max(j_1,j_2)/2}\cdot 2^{(j_1+j_2)/2}\Pi,
  \end{equation}
  where $\Pi=\|f_{k_1}^{j_1}\|_{L^2}\cdot \|f_{k_2}^{j_2}\|_{L^2}$.
  Clearly,
  \begin{equation*}
    \sum_{2^{\mathrm{max}\,(j_1,j_2,j)}\geq C^{-1}2^{k'} |n_{k_1}|^\alpha}2^{-j\delta }\cdot 2^{-\max(j_1,j_2)/2}\leq C(1+2^{k'}|n_{k_1}|^\alpha)^{-\delta},
  \end{equation*}
  and the bound \eqref{hl2} follows from \eqref{hl4}.
\end{proof}

Finally, we consider interactions of comparable frequencies.

\begin{lemma}\label{Lemmak3}
  Assume $k_1,k_2\in\Z$, $k\in\Z\setminus\{0\}$, $(1+|n_{k_i}|
  )/(1+|n_k| )\in [2^{-20},2^{20}]$, $i=1,2$, $f_{k_1}\in Z_{k_1}$,
  and $f_{k_2}\in Z_{k_2}$. Then
  \begin{equation}\label{hb1}
    \begin{split}
      (1+|n_k|)\cdot &\|\chi_k(\xi)(\tau-\omega(\xi)+i)^{-1}\cdot (f_{k_1}\ast f_{k_2})\|_{Z_k}\\
      &\leq C\Lambda(k_1,k_2,k)(1+|n_k| )^{-\delta}\cdot
      \|f_{k_1}\|_{Z_{k_1}}\|f_{k_2}\|_{Z_{k_2}},
    \end{split}
  \end{equation}
  where, with $A=\min\big( \big\|n_{k_1}|-|n_{k_2}|\big|,\big|
  |n_k|-|n_{k_1}|\big|,\big| |n_k|-|n_{k_2}|\big| \big)$,
  \begin{equation*}
    \Lambda(k_1,k_2,k)=
    \begin{cases}
      1&\text{ if }A\leq 2^{50}(1+|n_k| )^{1/2};\\
      A^{-1/2}&\text{ if }A> 2^{50}(1+|n_k| )^{1/2}.
    \end{cases}
  \end{equation*}
  In addition, if $k_1,k_2\in\Z$, $1+|n_{k_i}|\in [2^{-20},2^{20}]$,
  $i=1,2$, $f_{k_1}\in Z_{k_1}$, and $f_{k_2}\in Z_{k_2}$, then
  \begin{equation}\label{hb1.3}
    \begin{split}
      &\|\chi_0(\xi)(\tau-\omega(\xi)+i)^{-1}\cdot (f_{k_1}\ast
      f_{k_2})\|_{X_0^{-1/2+\delta}}\leq
      C\|f_{k_1}\|_{Z_{k_1}}\|f_{k_2}\|_{Z_{k_2}}.
    \end{split}
  \end{equation}
\end{lemma}

\begin{proof} [Proof of Lemma \ref{Lemmak3}] We analyze two cases:
  $|n_k|\geq 2^{100}$ and $|n_k|\leq 2^{100}$.

  {\bf{Case 1:}} $|n_k|\geq 2^{100}$. In view of the hypothesis,
  $|n_{k_i}|\geq 2^{70}$, $i=1,2$. Using \eqref{repr1}, we may assume
  $f_{k_1}=f_{k_1}^{j_1}$ is supported in $D_{k_1}^{j_1}$ and
  $f_{k_2}=f_{k_2}^{j_2}$ is supported in $D_{k_2}^{j_2}$. For
  \eqref{hb1} it suffices to prove that
  \begin{equation}\label{hb2}
    \begin{split}
      &|n_k|\cdot \sum_{j\in\Z_+}2^{-j/2}\beta_{k,j}\|\mathbf{1}_{D_k^j}\cdot (f_{k_1}^{j_1}\ast f_{k_2}^{j_2})\|_{L^2}\\
      &\leq C\Lambda(k_1,k_2,k)|n_k|^{-\delta}\cdot
      2^{j_1/2}\beta_{k_1,j_1}\|f_{k_1}^{j_1}\|_{L^2}\cdot
      2^{j_2/2}\beta_{k_2,j_2}\|f_{k_2}^{j_2}\|_{L^2}.
    \end{split}
  \end{equation}
  In view of \eqref{om20}, $\mathbf{1}_{D_k^j}\cdot (f_{k_1}^{j_1}\ast
  f_{k_2}^{j_2})\equiv 0$ unless
  \begin{equation}\label{hb3}
    2^{\max(j_1,j_2,j)}\geq  C^{-1}|n_k|^{\alpha+1}\text{ and }\beta_{k,j}\leq C\max(\beta_{k_1,j_1},\beta_{k_2,j_2}).
  \end{equation}
  Using \eqref{on3},
  \begin{equation*}
    \|\mathbf{1}_{D_k^j}\cdot (f_{k_1}^{j_1}\ast f_{k_2}^{j_2})\|_{L^2}\leq 2^{(j_1+j_2+j)/2}2^{-\max(j_1,j_2,j)/2}\cdot \Pi,
  \end{equation*}
  where $\Pi=\|f_{k_1}^{j_1}\|_{L^2}\cdot \|f_{k_2}^{j_2}\|_{L^2}$.
  Thus, using \eqref{hb3},
  \begin{equation}\label{hb4}
    \begin{split}
      &|n_k|\cdot \sum_{j\in\Z_+}2^{-j/2}\beta_{k,j}\|\mathbf{1}_{D_k^j}\cdot (f_{k_1}^{j_1}\ast f_{k_2}^{j_2})\|_{L^2}\\
      &\leq C\ln(2+|n_k| )|n_k|^{-(\alpha-1)/2}\cdot
      2^{j_1/2}\beta_{k_1,j_1}\|f_{k_1}^{j_1}\|_{L^2}\cdot
      2^{j_2/2}\beta_{k_2,j_2}\|f_{k_2}^{j_2}\|_{L^2}.
    \end{split}
  \end{equation}

  It remains to prove the bound \eqref{hb2} in the case $A>
  2^{50}(1+|n_k| )^{1/2}$. In this case, using \eqref{on32},
  \begin{equation*}
    \|\mathbf{1}_{D_k^j}\cdot (f_{k_1}^{j_1}\ast f_{k_2}^{j_2})\|_{L^2}\leq 2^{(j_1+j_2+j)/2}[2^{\max(j_1,j_2,j)}A|n_k|^{\al-1}]^{-1/2}\cdot \Pi.
  \end{equation*}
  The bound \eqref{hb2} follows from \eqref{hb3}.

  {\bf{Case 2:}} $|n_k|\leq 2^{100}$. Since $|n_{k_i}|\leq C$,
  $i=1,2$, for \eqref{hb1} we have to prove that
  \begin{equation}\label{hq1}
    \begin{split}
      \|\chi_k(\xi)(\tau-\omega(\xi)+i)^{-1}\cdot (f_{k_1}\ast
      f_{k_2})\|_{Z_k}\leq
      C\|f_{k_1}\|_{Z_{k_1}}\|f_{k_2}\|_{Z_{k_2}}.
    \end{split}
  \end{equation}
  It follows from the definitions and \eqref{ro81} that
  \begin{equation}\label{hq4}
    \sum_{j\in\Z}2^{j(1-\delta)}\|\eta_j(\tau-\omega(\xi))f_{k_l}\|_{L^2}\leq C\|f_{k_l}\|_{Z_{k_l}},
  \end{equation}
  for $l=1,2$, since $|n_{k_l}|\leq C$. The bound \eqref{hq1} follows
  easily using \eqref{on31}. The bound \eqref{hb1.3} also follows from
  \eqref{hq4} and \eqref{on31}. This completes the proof of the lemma.
\end{proof}

\section{Multiplication by smooth bounded functions}\label{mult}

In this section we consider operators on $Z_k$ given by convolutions
with Fourier transforms of certain smooth bounded functions. As in
\cite{IoKe}, for integers $N\geq 100$ we define the space of
{\it{admissible factors}}
\begin{equation}\label{ar3}
  \begin{split}
    S^\infty_N=\{&m:\mathbb{R}^2\to\mathbb{C}: m\text{ supported in }\R\times[-10,10]\text{ and }\\
    &\|m\|_{S^\infty_N}=\sum_{\sigma_1=0}^N\|\partial_t^{\sigma_1}m\|_{L^\infty}+\sum_{\sigma_1=0}^N\sum_{\sigma_2=1}^N\|\partial_t^{\sigma_1}\partial_x^{\sigma_2}m\|_{L^2}<\infty\}.
  \end{split}
\end{equation}
Notice that bounded functions such as $\eta_0(t)e^{iq\Psi}$,
$q\in\mathbb{R}$, $\Psi $ as in \eqref{fg9}, are in $S^\infty_N$. We
also define the space of {\it{restricted admissible factors}}
\begin{equation}\label{ar3'}
  \begin{split}
    S^2_N=\{&m:\mathbb{R}^2\to\mathbb{C}: m\text{ supported in }\R\times[-4,4]\text{ and }\\
    &\|m\|_{S^2_N}=\sum_{\sigma_1=0}^N\sum_{\sigma_2=0}^N\|\partial_t^{\sigma_1}\partial_x^{\sigma_2}m\|_{L^2}<\infty\}.
  \end{split}
\end{equation}
It is easy to see that bounded functions such as
$\eta_0(t)\phi_{\mathrm{low}}e^{iq\Psi}$, $q\in\mathbb{R}$ (with the
notation in section \ref{construct}) are in $S^2_N$. Using the Sobolev
embedding theorem, it is easy to verify the following properties:
\begin{equation}\label{ar3''}
  \begin{cases}
    &S^2_N\subseteq S^\infty_{N-10};\\
    &S^\infty_N\cdot S^\infty_N\subseteq S^\infty_{N-10};\\
    &S^2_N\cdot S^\infty_N\subseteq S^2_{N-10};\\
    &\partial_x S^\infty_N\subseteq S^2_{N-10}.
  \end{cases}
\end{equation}
For $k\in\mathbb{Z}$ we define
\begin{equation}\label{ar1set}
  M^{\mathrm{high}}_{k}=\bigcup_{2^{j+20}\geq |n_k|^\alpha} J_j  \text{ and } M^{\mathrm{low}}_{k}=(M^{\mathrm{high}}_{k})^c,
\end{equation}
and
\begin{equation}\label{ar1}
  Z_k^{\mathrm{high}}=\big\{f_k\in Z_k:f_k\text{ is supported in }
  \{\tau-\omega(\xi)\in M^{\mathrm{high}}_{k}\}\big\}.
\end{equation}
Clearly, $Z_k^{\mathrm{high}}=Z_k$ if $|n_k|^\al\leq 2^{20}$. The main
result in this section is the following proposition.

\begin{proposition}\label{Lemmab4} (a) Assume $k_1,k_2\in\mathbb{Z}$, $f_{k_1}^{\mathrm{high}}\in Z_{k_1}^{\mathrm{high}}$,
  $\epsilon\in\{-1,0\}$, and $m\in S^\infty_{100}$. Then
  \begin{equation}\label{ar4}
    \begin{split}
      \Big\|\chi_{k_2}(\xi_2)&(\tau_2-\omega(\xi_2)+i)^\epsilon
      \cdot(f_{k_1}^{\mathrm{high}}\ast\mathcal{F}(m))
      \Big\|_{{Z}_{k_2}}\\
      &\leq C
      (1+|k_1-k_2|)^{-60}\ln(2+|n_{k_1}|)\|m\|_{S^\infty_{100}}\cdot\|(\tau_1-\omega(\xi_1)+i)^\epsilon\cdot
      f_{k_1}^{\mathrm{high}}\|_{Z_{k_1}}.
    \end{split}
  \end{equation}

  (b) Assume $k_1,k_2\in\mathbb{Z}$, $k_1\neq 0$, $f_{k_1}\in
  Z_{k_1}$, $\epsilon\in\{-1,0\}$, and $m'\in S^2_{100}$. Then
  \begin{equation}\label{ar4.1}
    \begin{split}
      \Big\|\chi_{k_2}(\xi_2)&(\tau_2-\omega(\xi_2)+i)^\epsilon
      \cdot(f_{k_1}\ast\mathcal{F}(m'))
      \Big\|_{{Z}_{k_2}}\\
      &\leq C
      (1+|k_1-k_2|)^{-60}\ln(2+|n_{k_1}|)\|m'\|_{S^2_{100}}\cdot\|(\tau_1-\omega(\xi_1)+i)^\epsilon\cdot
      f_{k_1}\|_{Z_{k_1}}.
    \end{split}
  \end{equation}
  In addition, if $f_0\in X_0^0$ (see definition \eqref{ro80}) then
  \begin{equation}\label{ar4.2}
    \begin{split}
      \Big\|\chi_{k_2}(\xi_2)&(\tau_2-\omega(\xi_2)+i)^\epsilon
      \cdot(f_{0}\ast\mathcal{F}(m'))
      \Big\|_{{Z}_{k_2}}\\
      &\leq C
      (1+|k_2|)^{-60}\|m'\|_{S^2_{100}}\cdot\|(\tau_1-\omega(\xi_1)+i)^\epsilon\cdot
      f_0\|_{X_0^0}.
    \end{split}
  \end{equation}
\end{proposition}

The rest of this section is concerned with the proof of Proposition
\ref{Lemmab4}. The proofs in this section are similar to the proofs in
\cite[Section 9]{IoKe}. We may assume
$\|m\|_{S^\infty_{100}}=\|m'\|_{S_{100}^2}=1$. For any
$j''\in\mathbb{Z}_+$ and $k''\in\mathbb{Z}$ let
\begin{equation}\label{ar12}
  \begin{split}
    &m_{k'',j''}=\mathcal{F}^{-1}\big[\eta_{j''}(\tau)\widetilde{\eta}_{k''}(\xi)\mathcal{F}(m)\big],\\
    &m'_{k'',j''}=\mathcal{F}^{-1}\big[\eta_{j''}(\tau)\widetilde{\eta}_{k''}(\xi)\mathcal{F}(m')\big].
  \end{split}
\end{equation}
Let $m_{\leq k'',j''}=\sum_{k'''\leq k''}m_{k''',j''}$ and $m'_{\leq
  k'',j''}=\sum_{k'''\leq k''}m'_{k''',j''}$. Using \eqref{ar3} and
\eqref{ar3'}, for any $j''\in\mathbb{Z}_+$ and $k''\in\mathbb{Z}$,
\begin{equation}\label{ar10}
  \begin{cases}
    &\|m_{\leq k'',j''}\|_{L^\infty_{x,t}}\leq C2^{-80j''};\\
    &2^{k''}\|m_{k'',j''}\|_{L^2_{x,t}}+\|m_{k'',j''}\|_{L^\infty_{x,t}}\leq
    C(1+2^{k''})^{-80}2^{-80j''},
  \end{cases}
\end{equation}
and
\begin{equation}\label{mi4}
  \|m'_{\leq k'',j''}\|_{L^2_{x,t}}+(1+2^{k''})^{-80}\|m'_{k'',j''}\|_{L^2_{x,t}}\leq C2^{-80j''}.
\end{equation}
We prove the proposition in several steps.

{\bf{Step 1: proof of \eqref{ar4} in the case $k_1=k_2=0$}}. The
estimate \eqref{ar4} in this case is the main reason for defining the
space $Z_0$ as in \eqref{def1d}, instead of, for example,
$Z_0=X_0^{-1/2+\delta}$. Using the definition \eqref{def1d} it is easy
to see that
\begin{equation}\label{yu4}
  \|(\tau-\omega(\xi)+i)^\epsilon h\|_{Z_0}\approx\|(\tau+i)^\epsilon h\|_{Z_0}.
\end{equation}
Therefore, we have to prove that
\begin{equation}\label{ar4.9}
  \Big\|\chi_{0}(\xi_2)(\tau_2+i)^\epsilon
  \cdot(f_0\ast\mathcal{F}(m))
  \Big\|_{{Z}_{0}}\leq C\|(\tau_1+i)^\epsilon\cdot f_0\|_{Z_{0}}
\end{equation}
for any $f_0\in Z_0$, $\epsilon\in\{0,-1\}$.

Assume first that $(\tau_1+i)^\epsilon f_0\in X_0=X_0^{-1}$. Using the
representation \eqref{repr2}, we may assume that $f_0=f^{j_1}_{0,k'}$
is an $L^2$ function supported in $D_{0,k'}^{j_1}$, $k'\leq 2$,
$j_1\geq 0$,
\begin{equation*}
  \|(\tau_1+i)^\epsilon f_0\|_{X_0}\approx 2^{\epsilon
    j_1}2^{j_1(1-\delta)}2^{-k'}\|f_{0,k'}^{j_1}\|_{L^2}.
\end{equation*}
We decompose
\begin{equation*}
  m=\sum_{j''=0}^\infty m_{\leq k'-10,j''}+\sum_{k''=k'-9}^\infty\sum_{j''=0}^{\infty}m_{k'',j''}.
\end{equation*}
For \eqref{ar4.9} it suffices to prove that
\begin{equation}\label{yu1}
  \begin{split}
    &\sum_{j''=0}^\infty\Big\|\chi_{0}(\xi_2)(\tau_2+i)^\epsilon\cdot(f_{0,k'}^{j_1}\ast\mathcal{F}(m_{\leq k'-10,j''}))\Big\|_{X_0}\\
    +&\sum_{k''=k'-9}^{10}\sum_{j''=0}^{\infty}\Big\|\chi_{0}(\xi_2)(\tau_2+i)^\epsilon\cdot(f_{0,k'}^{j_1}\ast\mathcal{F}(m_{k'',j''}))\Big\|_{Y_0}\\
    \leq &C2^{\epsilon
      j_1}2^{j_1(1-\delta)}2^{-k'}\|f_{0,k'}^{j_1}\|_{L^2}
  \end{split}
\end{equation}
Using the definition \eqref{def1''}, the first sum in the left-hand is
dominated by
\begin{equation*}
  C\sum_{j''=0}^\infty\sum_{j_2=0}^\infty 2^{\epsilon j_2}2^{j_2(1-\delta)}2^{-k'}\|\eta_{j_2}(\tau_2)\cdot(f_{0,k'}^{j_1}\ast\mathcal{F}(m_{\leq k'-10,j''}))\|_{L^2} 
\end{equation*}
We observe that
$\eta_{j_2}(\tau_2)\cdot(f_{0,k'}^{j_1}\ast\mathcal{F}(m_{\leq
  k'-10,j''}))\equiv 0$ unless
\begin{equation}\label{yu2}
  (j_2,j'')\in L^C_{j_1}=\{(j_2,j'')\in\Z_+\times\Z_+:\,|j_1-j_2|\leq C\text{ or }j_1,j_2\leq j''+C\}
\end{equation}
for some constant $C$. Thus, using \eqref{ar10} and Plancherel
theorem, the expression above is dominated by
\begin{equation*}
  \sum_{(j_2,j'')\in L^C_{j_1}}2^{\epsilon j_2}2^{j_2(1-\delta)}2^{-k'}\|f_{0,k'}^{j_1}\|_{L^2}\|m_{\leq k'-10,j''}\|_{L^\infty}\leq C2^{\epsilon
    j_1}2^{j_1(1-\delta)}2^{-k'}\|f_{0,k'}^{j_1}\|_{L^2},
\end{equation*}
as desired. Similarly, the second sum in the left-hand side of
\eqref{yu1} is dominated by
\begin{align*}
  &C\sum_{k''=k'-9}^{10}\sum_{(j_2,j'')\in L^C_{j_1}}2^{\epsilon j_2}2^{j_2(1-\delta)}\|\mathcal{F}^{-1}(f_{0,k'}^{j_1})\cdot m_{k'',j''}))\|_{L^1_xL^2_t}\\
  \leq&
  C\|f_{0,k'}^{j_1}\|_{L^2}\sum_{k''=k'-9}^{10}\sum_{(j_2,j'')\in
    L^C_{j_1}}2^{\epsilon
    j_2}2^{j_2(1-\delta)}\|m_{k'',j''}\|_{L^2_xL^\infty_t}\\
  \leq& C2^{\epsilon
    j_1}2^{j_1(1-\delta)}2^{-k'}\|f_{0,k'}^{j_1}\|_{L^2},
\end{align*}
where in the last inequality we use
$\|m_{k'',j''}\|_{L^2_xL^\infty_t}\leq C2^{-k''}2^{-70j''}$, compare
with \eqref{ar10}.

It remains to prove \eqref{ar4.9} in the case $(\tau_1+i)^\epsilon
f_0\in Y_0$. Using the definition \eqref{def1c}, we may assume
$f_0=g_0^{j_1}$ is supported in $I_0\times J_{j_1}$ and
\begin{equation*}
  \|(\tau_1+i)^\epsilon f_0\|_{Y_0}\approx 2^{\epsilon j_1}2^{j_1(1-\delta)}\|\mathcal{F}^{-1}(g_0^{j_1})\|_{L^1_xL^2_t}.
\end{equation*}
For \eqref{ar4.9} it suffices to prove that
\begin{equation*}
  \sum_{j''\geq 0}\|\chi_0(\xi_2)(\tau_2+i)^\epsilon\cdot (g_0^{j_1}\ast\mathcal{F}(m_{\leq 10,j''})\|_{Y_0}\leq C2^{\epsilon j_1}2^{j_1(1-\delta)}\|\mathcal{F}^{-1}(g_0^{j_1})\|_{L^1_xL^2_t}.
\end{equation*}
Using Plancherel theorem and \eqref{ar10}, the left-hand side is
dominated by
\begin{equation*}
  C\sum_{(j_2,j'')\in L_{j_1}^C}2^{j_2(1-\delta)}2^{\epsilon j_2}\|\mathcal{F}^{-1}(g_0^{j_1})\cdot m_{\leq 10,j''}\|_{L^1_xL^2_t}\leq  C2^{\epsilon j_1}2^{j_1(1-\delta)}\|\mathcal{F}^{-1}(g_0^{j_1})\|_{L^1_xL^2_t},
\end{equation*}
as desired.  \medskip

{\bf{Step 2: proof of \eqref{ar4.2} in the case $k_2=0$}}. Using the
representation \eqref{repr2}, we may assume that $f_0=f^{j_1}_{0,k'}$
is an $L^2$ function supported in $D_{0,k'}^{j_1}$, $k'\leq 2$,
$j_1\geq 0$,
\begin{equation*}
  \|(\tau_1-\omega(\xi_1)+i)^\epsilon f_0\|_{X_0^0}\approx 2^{\epsilon
    j_1}2^{j_1(1-\delta)}\|f_{0,k'}^{j_1}\|_{L^2}.
\end{equation*}
In view of \eqref{yu4} it suffices to prove that
\begin{equation}\label{yu5}
  \sum_{j''=0}^\infty\|\chi_{0}(\xi_2)(\tau_2+i)^\epsilon\cdot(f_{0,k'}^{j_1}\ast\mathcal{F}(m'_{\leq 10,j''}))\|_{Y_0}\leq C2^{\epsilon j_1}2^{j_1(1-\delta)}\|f_{0,k'}^{j_1}\|_{L^2}.
\end{equation}
Using Plancherel theorem and \eqref{def1c}, the left-hand side of
\eqref{yu5} is dominated by
\begin{equation*}
  C\|f_{0,k'}^{j_1}\|_{L^2}\sum_{(j_2,j'')\in L_{j_1}^C}2^{j_2(1-\delta)}2^{\epsilon j_2}\|m'_{\leq 10,j''}\|_{L^2_xL^\infty_t},
\end{equation*}
and the bound \eqref{yu5} follows since $\|m'_{\leq
  10,j''}\|_{L^2_xL^\infty_t}\leq C2^{-70j''}$, compare with
\eqref{mi4}.  \medskip

{\bf{Step 3: proof of \eqref{ar4} in the case
    $k_1,k_2\in\Z\setminus\{0\}$, $|k_1-k_2|\leq 10$.}} In view of the
definition of $Z_{k_1}^{\mathrm{high}}$ and \eqref{repr1}, we may
assume that $f_{k_1}^{\mathrm{high}}=f_{k_1}^{j_1}$ is an $L^2$
function supported in $D_{k_1}^{j_1}$, $2^{j_1+20}\geq
|n_{k_1}|^\alpha$, $\|(\tau_1-\omega(\xi_1)+i)^\epsilon
f_{k_1}^\mathrm{high}\|_{Z_{k_1}}\approx 2^{\epsilon
  j_1}2^{j_1/2}\beta_{k_1,j_1}\|f_{k_1}^{j_1}\|_{L^2}$. We write
\begin{equation}\label{ar20}
  m=\sum_{j''=0}^\infty m_{\leq -100,j''}+
  \sum_{k''=-99}^\infty\sum_{j''=0}^\infty m_{k'',j''}.
\end{equation}
For \eqref{ar4} it suffices to prove that for $\epsilon\in\{-1,0\}$
\begin{equation}\label{ar21}
  \begin{split}
    &\sum_{j''\geq
      0}\big\|\chi_{k_2}(\xi_2)(\tau_2-\omega(\xi_2)+i)^\epsilon
    \cdot[f_{k_1}^{j_1}\ast\mathcal{F}(m_{\leq
      -100,j''})](\xi_2,\tau_2)\big\|_{Z_{k_2}}\\
    &+\sum_{k''\geq -99}\sum_{j''\geq
      0}\big\|\chi_{k_2}(\xi_2)(\tau_2-\omega(\xi_2)+i)^\epsilon\cdot
    [f_{k_1}^{j_1}\ast\mathcal{F}(m_{k'',j''})](\xi_2,\tau_2)\big\|_{Z_{k_2}}\\
    &\leq C\ln(2+|n_{k_1}| )2^{\epsilon
      j_1}2^{j_1/2}\beta_{k_1,j_1}\|f_{k_1}^{j_1}\|_{L^2}.
  \end{split}
\end{equation}

To bound the first sum in \eqref{ar21}, we use \eqref{om20} and
$2^{j_1+20}\geq |n_{k_1}|^\alpha$ to conclude that
$\mathbf{1}_{D_{k_2}^{j_2}}\cdot[f_{k_1}^{j_1}\ast\mathcal{F}(m_{\leq
  -100,j''})](\xi_2,\tau_2)\equiv 0$ unless $(j_2,j'')\in L_{j_1}^C$,
see definition \eqref{yu2}. Using Plancherel theorem and \eqref{ar10},
\begin{equation*}
  \big|\big|f_{k_1}^{j_1}\ast\mathcal{F}(m_{\leq -100,j''})
  \big|\big|_{L^2_{\xi_2,\tau_2}}\leq
  C2^{-80j''}\|f_{k_1}^{j_1}\|_{L^2}.
\end{equation*}
Thus the first sum in \eqref{ar21} is dominated by
\begin{equation*}
  C\sum_{(j_2,j'')\in L_{j_1}^C}2^{\epsilon j_2}2^{j_2/2}
  \beta_{k_2,j_2}2^{-80j''}\|f_{k_1}^{j_1}\|_{L^2},
\end{equation*}
which suffices (recall that $|k_1-k_2|\leq 10$).

To bound the second sum in \eqref{ar21} assume first that
$\epsilon=0$. As before, we use \eqref{om20} to conclude that
$\mathbf{1}_{D_{k_2}^{j_2}}\cdot[f_{k_1}^{j_1}\ast\mathcal{F}(m_{k'',j''})]
(\xi_2,\tau_2)\equiv 0$ unless
\begin{equation}\label{ar26}
  |j_1-j_2|\leq 4\,\text{ or }\,j_1,j_2\leq \log_2( |n_{k_1}|^\alpha)+k''+j''+C.
\end{equation}
Using Plancherel theorem and \eqref{ar10},
\begin{equation}\label{ng2}
  \big|\big|f_{k_1}^{j_1}\ast\mathcal{F}(m_{k'',j''})\big|\big|_{L^2_{\xi_2,\tau_2}}\leq C2^{-80k''}2^{-80j''}\|f_{k_1}^{j_1}\|_{L^2}.
\end{equation}
Thus, using $j_1+C\geq\log_2( |n_{k_1}|^\alpha)$, the second sum in
\eqref{ar21} is dominated by
\begin{align*}
  &C\sum_{k''\geq -99}\sum_{j''\geq 0}2^{-80k''}2^{-80j''}\|f_{k_1}^{j_1}\|_{L^2}\sum_{j_2\leq j_1+k''+j''+C}2^{j_2/2}\beta_{k_2,j_2}\\
  \leq & C2^{j_1/2}\beta_{k_1,j_1}\|f_{k_1}^{j_1}\|_{L^2}.
\end{align*}

We bound now the second sum in \eqref{ar21} when $\epsilon=-1$.  The
main difficulty is the presence of the indices $j_2\ll j_1$.  In fact,
for indices $j_2\geq j_1-10$, the argument above applies since the
left-hand side is multiplied by $2^{-j_2}$ and the right-hand side is
multiplied by $2^{-j_1}$. In view of \eqref{ar26}, it suffices to
prove that
\begin{equation}\label{ng1}
  \begin{split}
    &\sum_{k''+j''\geq j_1-\log_2( |n_{k_1}|^\alpha)-C}\sum_{j_2\leq
      j_1-10}2^{-j_2/2}\beta_{k_2,j_2}\big|\big|\mathbf{1}_{D_{k_2}^{j_2}}\cdot
    [f_{k_1}^{j_1}\ast\mathcal{F}(m_{k'',j''})]
    \big|\big|_{L^2}\\
    &\leq C\ln(2+|n_{k_1}|
    )2^{-j_1/2}\beta_{k_1,j_1}\|f_{k_1}^{j_1}\|_{L^2}.
  \end{split}
\end{equation}
If $j_1\geq \log_2( |n_{k_1}|^{\alpha+1})-C$ the bound \eqref{ng1}
follows easily from \eqref{ng2}. Assuming $j_1\leq \log_2(
|n_{k_1}|^{\alpha+1})-C$, the sum over $k''\geq \log_2( |n_{k_1}|)-C$
in \eqref{ng1} is bounded easily using again \eqref{ng2}. If $k'' \leq
\log_2( |n_{k_1}|)-C$ then, using Corollary \ref{Lemmad2} (b)
\begin{equation*}
  \big|\big|\mathbf{1}_{D_{k_2}^{j_2}}\cdot [f_{k_1}^{j_1}\ast\mathcal{F}(m_{k'',j''})]
  \big|\big|_{L^2}\leq C2^{10(j''+k'')}2^{j_2/2}|n_{k_1}|^{-\alpha/2}\|m_{k'',j''}\|_{L^2}\|f_{k_1}^{j_1}\|_{L^2}.
\end{equation*}
The bound \eqref{ng1} follows using \eqref{ar10}.  \medskip

{\bf{Step 4: proof of \eqref{ar4.1} in the case $k_2\neq 0$,
    $|k_1-k_2|\leq 10$.}} In view of \eqref{repr1}, we may assume that
$f_{k_1}=f_{k_1}^{j_1}$ is an $L^2$ function supported in
$D_{k_1}^{j_1}$, $\|(\tau_1-\omega(\xi_1)+i)^\epsilon
f_{k_1}\|_{Z_{k_1}}\approx 2^{\epsilon
  j_1}2^{j_1/2}\beta_{k_1,j_1}\|f_{k_1}^{j_1}\|_{L^2}$. In view of the
case analyzed earlier, we may assume that
\begin{equation*}
  2^{j_1}\leq |n_{k_1}|^\alpha.
\end{equation*}
For \eqref{ar4.1} it suffices to prove that for $\epsilon\in\{-1,0\}$
\begin{equation*}
  \begin{split}
    &\sum_{k''\in\mathbb{Z}}\sum_{j''\geq
      0}\big\|\chi_{k_2}(\xi_2)(\tau_2-\omega(\xi_2)+i)^\epsilon\cdot
    [f_{k_1}^{j_1}\ast\mathcal{F}(m'_{k'',j''})]\big\|_{Z_{k_2}}\\
    &\leq C\ln(2+|n_{k_1}| )2^{\epsilon
      j_1}2^{j_1/2}\beta_{k_1,j_1}\|f_{k_1}^{j_1}\|_{L^2}.
  \end{split}
\end{equation*}
Using the definition of the $Z_k$ spaces, this is equivalent to
proving that
\begin{equation}\label{mi21}
  \begin{split}
    &\sum_{k''\in\mathbb{Z}}\sum_{j''\geq 0}\sum_{j_2\geq 0}2^{\epsilon j_2}2^{j_2/2}\beta_{k_2,j_2}\|\mathbf{1}_{D_{k_2}^{j_2}}\cdot [f_{k_1}^{j_1}\ast|\mathcal{F}(m'_{k'',j''})|]\|_{L^2}\\
    &\leq C\ln(2+|n_{k_1}| )2^{\epsilon
      j_1}2^{j_1/2}\beta_{k_1,j_1}\|f_{k_1}^{j_1}\|_{L^2}.
  \end{split}
\end{equation}
This follows using the bound \eqref{on31} for
$2^{k''}|n_{k_1}|^\alpha\leq 1$ and $2^{k''}\ge |n_{k_1}|/100$, and
the bound \eqref{on32} for
$2^{k''}\in[|n_{k_1}|^{-\alpha},|n_{k_1}|/100]$.  \medskip

{\bf{Step 5: proof of \eqref{ar4} and \eqref{ar4.1} in the case
    $k_1,k_2\in\Z\setminus\{0\}$, $|k_1-k_2|\geq 10$.}} Clearly, it
suffices to prove the stronger bound \eqref{ar4.1}.  In view of
\eqref{repr1}, we may assume that $f_{k_1}=f_{k_1}^{j_1}$ is an $L^2$
function supported in $D_{k_1}^{j_1}$,
$\|(\tau_1-\omega(\xi_1)+i)^\epsilon f_{k_1}\|_{Z_{k_1}}\approx
2^{\epsilon j_1}2^{j_1/2}\beta_{k_1,j_1}\|f_{k_1}^{j_1}\|_{L^2}$. It
suffices to prove that
\begin{equation*}
  \begin{split}
    &\sum_{2^{k''}\geq(|n_{k_1}|+|n_{k_2}|)^{1/2}}\sum_{j''\geq
      0}\big\|\chi_{k_2}(\xi_2)(\tau_2-\omega(\xi_2)+i)^\epsilon\cdot
    [f_{k_1}^{j_1}\ast\mathcal{F}(m'_{k'',j''})]\big\|_{Z_{k_2}}\\
    &\leq C|k_1-k_2|^{-60}2^{\epsilon
      j_1}2^{j_1/2}\beta_{k_1,j_1}\|f_{k_1}^{j_1}\|_{L^2}.
  \end{split}
\end{equation*}
Using \eqref{on31} and \eqref{om20}, it suffices to prove that
\begin{equation}\label{yu8}
  \begin{split}
    \sum_{2^{k''}\geq(|n_{k_1}|+|n_{k_2}|)^{1/2}}&\sum_{j_2,j''}2^{\epsilon j_2}2^{j_2/2}\beta_{k_2,j_2}\|f_{k_1}^{j_1}\|_{L^2}2^{10k''+10j''}\|m'_{k'',j''}\|_{L^2}\\
    &\leq C|k_1-k_2|^{-60}2^{\epsilon
      j_1}2^{j_1/2}\beta_{k_1,j_1}\|f_{k_1}^{j_1}\|_{L^2},
  \end{split}
\end{equation}
where the sum over $j_2$ and $j''$ is taken over the set
\begin{equation*}
  \{(j_2,j'')\in\Z_+\times\Z_+:|j_2-j_1|\leq C\text{ or }j_1,j_2\leq 10k''+10j''+C\}.
\end{equation*}
The bound \eqref{yu8} follows easily using \eqref{mi4}.  \medskip

{\bf{Step 6: proof of \eqref{ar4} and \eqref{ar4.1} in the case
    $k_2=0$, $k_1\neq 0$.}} In view of \eqref{yu4} and the discussion
in Steps 3, 4, and 5, it suffices to prove that
\begin{equation*}
  \|\chi_0(\xi_2/2^{10})(\tau_2+i)^\epsilon(f_{k_1}\ast\mathcal{F}(m'))\|_{Z_0}\leq C|k_1|^{-60}\|(\tau_1-\omega(\xi_1)+i)^{\epsilon}f_{k_1}\|_{Z_{k_1}}
\end{equation*}
for any $f_{k_1}\in Z_{k_1}$, $\epsilon\in\{-1,0\}$. In view of
\eqref{repr1}, we may assume that $f_{k_1}=f_{k_1}^{j_1}$ is an $L^2$
function supported in $D_{k_1}^{j_1}$,
$\|(\tau_1-\omega(\xi_1)+i)^\epsilon f_{k_1}\|_{Z_{k_1}}\approx
2^{\epsilon j_1}2^{j_1/2}\beta_{k_1,j_1}\|f_{k_1}^{j_1}\|_{L^2}$. It
suffices to prove that
\begin{equation}\label{yu60}
  \begin{split}
    \sum_{2^{k''+10}\geq |n_{k_1}|}\sum_{j''\geq 0}\sum_{j_2\geq 0}&2^{j_2(1-\delta)}2^{\epsilon j_2}\|\mathcal{F}^{-1}[\chi_0(\xi_2/2^{10})\eta_{j_2}(\tau_2)(f_{k_1}^{j_1}\ast\mathcal{F}(m'_{k'',j''}))]\|_{L^1_xL^2_t}\\
    &\leq C|k_1|^{-60}2^{\epsilon
      j_1}2^{j_1/2}\beta_{k_1,j_1}\|f_{k_1}^{j_1}\|_{L^2},
  \end{split}
\end{equation}
where the restriction $2^{k''+10}\geq |n_{k_1}|$ may be assumed due to
the support property of $f_{k_1}^{j_1}$. Using \eqref{mi4} and the
support properties,
\begin{equation*}
  \begin{split}
    \|\mathcal{F}^{-1}[\chi_0(\xi_2/2^{10})\eta_{j_2}(\tau_2)(f_{k_1}^{j_1}\ast\mathcal{F}(m'_{k'',j''}))]\|_{L^1_xL^2_t}&\leq C\|f_{k_1}^{j_1}\|_{L^2}\|m'_{k'',j''}\|_{L^2_xL^\infty_t}\\
    &\leq C2^{-70(k''+j'')}\|f_{k_1}^{j_1}\|_{L^2},
  \end{split}
\end{equation*}
and
$\|\mathcal{F}^{-1}[\chi_0(\xi_2/2^{10})\eta_{j_2}(\tau_2)(f_{k_1}^{j_1}\ast\mathcal{F}(m'_{k'',j''}))]\|_{L^1_xL^2_t}=0$
unless
\begin{equation*}
  |j_1-j_2|\leq 4\quad\text{ or }\quad j_1,j_2\leq j''+\log_2(|n_{k_1}|^\alpha)+C.
\end{equation*}
The bound \eqref{yu60} follows easily, using also
$2^{j_1/2}\beta_{k_1,j_1}\geq 2^{j_1(1-\delta)}|n_{k_1}|^{-1}$.
\medskip

{\bf{Step 7: proof of \eqref{ar4} and \eqref{ar4.2} in the case
    $k_2\neq 0$, $k_1=0$.}} In view of \eqref{ro81} and the discussion
in Steps 3, 4, and 5, it suffices to prove that
\begin{equation*}
  \|\chi_{k_2}(\xi_2)(\tau_2-\omega(\xi_2)+i)^\epsilon(f_0\ast\mathcal{F}(m'))\|_{Z_{k_2}}\leq C|k_2|^{-60}\|(\tau_1-\omega(\xi_1)+i)^{\epsilon}f_0\|_{X_0^0}
\end{equation*}
for any $f_0\in X_0^0$ supported in $\{(\xi_1,\tau_1):|\xi_1|\leq
2^{-20}\}$, $\epsilon\in\{-1,0\}$. In view of \eqref{repr2}, we may
assume that $f_{0}=f_{0,k'}^{j_1}$ is an $L^2$ function supported in
$D_{0,k'}^{j_1}$, $k'\leq -10$, $\|(\tau_1-\omega(\xi_1)+i)^\epsilon
f_{0}\|_{X_0^0}\approx 2^{\epsilon
  j_1}2^{j_1(1-\delta)}\|f_{0,k'}^{j_1}\|_{L^2}$. It suffices to prove
that
\begin{equation}\label{yu80}
  \begin{split}
    \sum_{2^{k''+10}\geq |n_{k_2}|}\sum_{j''\geq 0}\sum_{j_2\geq 0}&2^{j_2/2}\beta_{k_2,j_2}2^{\epsilon j_2}\|\mathbf{1}_{D_{k_2}^{j_2}}\cdot (f_{0,k'}^{j_1}\ast\mathcal{F}(m'_{k'',j''}))\|_{L^2}\\
    &\leq C|k_2|^{-60}2^{\epsilon
      j_1}2^{j_1(1-\delta)}\|f_{0,k'}^{j_1}\|_{L^2},
  \end{split}
\end{equation}
where the restriction $2^{k''+10}\geq |n_{k_2}|$ may be assumed due to
support properties. Using \eqref{mi4} and Plancherel theorem we have
\begin{equation*}
  \|\mathbf{1}_{D_{k_2}^{j_2}}\cdot (f_{0,k'}^{j_1}\ast\mathcal{F}(m'_{k'',j''}))\|_{L^2}\leq C2^{-70(k''+j'')}\|f_{0,k'}^{j_1}\|_{L^2}.
\end{equation*}
Using support properties we have $\|\mathbf{1}_{D_{k_2}^{j_2}}\cdot
(f_{0,k'}^{j_1}\ast\mathcal{F}(m'_{k'',j''}))\|_{L^2}=0$ unless
\begin{equation*}
  |j_1-j_2|\leq 4\quad\text{ or }\quad j_1,j_2\leq j''+\log_2(|n_{k_2}|^\alpha)+C.
\end{equation*} 
The bound \eqref{yu80} follows easily, using also
$2^{j_2/2}\beta_{k_2,j_2}\leq C2^{j_2(1-\delta)}$.

\section{The main technical lemma}\label{mainlemma}

In this section we combine the estimates in sections \ref{Dyadic2} and
\ref{mult} to prove our main global estimate. We define
\begin{equation}\label{ne1}
  \begin{split}
    &\|\cdot\|_{\widetilde F_k}=\|\cdot \|_{F_k} \text{ for } k\not=0
    \text{ and }
    \|\cdot \|_{\widetilde F_{0}}=\|\mathcal{F}(\cdot)\|_{X^0_0},\\
    &\|\cdot\|_{\widetilde N_k}=\|\cdot \|_{N_k} \text{ for } k\not=0
    \text{ and } \|\cdot \|_{\widetilde
      N_{0}}=\|(\tau-\omega(\xi)+i)^{-1}\mathcal{F}(\cdot)\|_{X^0_0},
  \end{split}
\end{equation}
see \eqref{ro80}. These norms are clearly controlled by $\|\cdot
\|_{F_k}$ and $\|\cdot\|_{N_k}$ respectively. Moreover,
\begin{equation*}
  \|\partial_x(\cdot)\|_{F_k}\leq C(1+|n_k|)\|\cdot\|_{\widetilde{F}_k},\qquad\|\partial_x(\cdot)\|_{N_k}\leq C(1+|n_k|)\|\cdot\|_{\widetilde{N}_k}.
\end{equation*}

\begin{lemma}\label{Lemmat1}
  Assume $\sigma\in[0,2]$ and $\Psi:\mathbb{R}\to\mathbb{R}$ is
  defined as in \eqref{fg9}. For any
  $(k_1,k_2,k)\in\mathbb{Z}\times\mathbb{Z}\times\mathbb{Z}$ assume
  that $a_{k_1,k_2,k}\in[-4,4]$, $w_{k_1,k_2,k},w'_{k_2,k_1,k}\in
  C(\mathbb{R}:\widetilde{H}^{\sigma})$ are supported in
  $\mathbb{R}_x\times[-4,4]$, $\mathcal{F}(w_{k_1,k_2,k})\in Z_{k_1}$,
  $\mathcal{F}(w'_{k_2,k_1,k})\in Z_{k_2}$,
  \begin{equation*}
    \sup_{k_2,k\in\mathbb{Z}}\|\mathcal{F}(w_{k_1,k_2,k})\|_{Z_{k_1}}=\Gamma_{k_1}\text{ and }\sup_{k_1,k\in\mathbb{Z}}\|\mathcal{F}(w'_{k_2,k_1,k})\|_{Z_{k_2}}=\Gamma'_{k_2}.
  \end{equation*}
  Then
  \begin{equation}\label{ne2}
    \begin{split}
      \sum_{k\in\mathbb{Z}}(1+|n_k|
      )^{2\sigma+\delta/4}&\Big(\sum_{k_1,k_2\in\mathbb{Z}}
      \|\partial_x P_k(e^{ia_{k_1,k_2,k}\Psi}w_{k_1,k_2,k}w'_{k_2,k_1,k})\|_{N_k}\\
      &+\|P_k(\partial_x(e^{ia_{k_1,k_2,k}\Psi})w_{k_1,k_2,k}w'_{k_2,k_1,k})\|_{N_k}\Big)^2\\
      &\leq
      C\Big(\sum_{k_1\in\mathbb{Z}}(1+|n_{k_1}|)^{-\delta/4}{\Gamma_{k_1}}^2\Big)\Big(\sum_{k_2\in\mathbb{Z}}(1+|n_{k_2}|)^{2\sigma-\delta/4}{\Gamma'_{k_2}}^2\Big)\\
      &+C\Big(\sum_{k_1\in\mathbb{Z}}(1+|n_{k_1}|)^{2\sigma-\delta/4}{\Gamma_{k_1}}^2\Big)\Big(\sum_{k_2\in\mathbb{Z}}(1+|n_{k_2}|)^{-\delta/4}{\Gamma'_{k_2}}^2\Big).
    \end{split}
  \end{equation}
\end{lemma}
\begin{proof}[Proof of Lemma \ref{Lemmat1}]
  Assume first that
  \begin{equation}\label{end1}
    a_{k_1,k_2,k}=0\text{ for any }k_1,k_2,k\in\mathbb{Z}.
  \end{equation}
  In this case we use only the dyadic estimates in section
  \ref{Dyadic2}. For any $k\in\mathbb{Z}$ let
  \begin{equation*}
    Q_k=\{(k_1,k_2)\in\mathbb{Z}\times\mathbb{Z}:(I_{k_1}+I_{k_2})\cap I_{k}\neq\emptyset\text{ and }|n_{k_1}|\leq |n_{k_2}|\}.
  \end{equation*}
  With $J_l$ as in section \ref{linear}, it suffices to prove the
  (slightly stronger) estimate
  \begin{equation}\label{gi1}
    \begin{split}
      \sum_{l=0}^\infty&2^{(2\sigma+2+\delta/2)l}\sum_{n_k\in
        J_l}\big(\sum_{(k_1,k_2)\in Q_k}
      \|P_k(w_{k_1,k_2,k}\cdot w'_{k_2,k_1,k})\|_{\widetilde{N}_k}\big)^2\\
      &\leq
      C\big(\sum_{l_1=0}^\infty2^{-(\delta/2)l_1}\sum_{n_{k_1}\in
        J_{l_1}}{\Gamma_{k_1}}^2\big)\big(\sum_{l_2=0}^\infty2^{(2\sigma-\delta/2)
        l_2}\sum_{n_{k_2}\in J_{l_2}}{\Gamma'_{k_2}}^2\big).
    \end{split}
  \end{equation}
  We fix now $l\in\mathbb{Z}_+$ and estimate
  \begin{equation*}
    2^{2l}\sum_{n_k\in J_l}\big(\sum_{(k_1,k_2)\in Q_k}\|P_k(w_{k_1,k_2,k}\cdot w'_{k_2,k_1,k})\|_{\widetilde{N}_k}\big)^2.
  \end{equation*}
  We split the set $Q_k=Q'_k\cup Q''_k\cup Q'''_k$, where we define
  the three subsets according to the conditions of Lemma
  \ref{Lemmak1}, Lemma \ref{Lemmak2}, and Lemma \ref{Lemmak3}:
  \begin{align*}
    Q'_k=&\begin{cases} \{(k_1,k_2)\in Q_k:|n_{k_1}|\leq
      |n_k|/2^{10}\} & \text{ if }
      |n_k|\geq 2^{20}\\ \emptyset& \text{ if } |n_k|<2^{20}\end{cases}\\
    Q''_k=&\{(k_1,k_2)\in Q_k:|n_{k_1}|\geq 2^{10}(1+|n_k| )\}\\
    Q'''_k=&\{(k_1,k_2)\in Q_k:(1+|n_{k_i}|)/(1+|n_k|)\in
    [2^{-20},2^{20}]\text{ for }i=1,2\}
  \end{align*}
  
  Using Lemma \ref{Lemmak1} we estimate
  \begin{equation}\label{ro1}
    \begin{split}
      &2^{2l}\sum_{n_k\in J_l}\big(\sum_{(k_1,k_2)\in Q'_k}\|P_k(w_{k_1,k_2,k}\cdot w'_{k_2,k_1,k})\|_{\widetilde{N}_k}\big)^2\\
      &\leq C2^{-2\delta l}\sum_{n_k\in J_l}\big(\sum_{l_1\leq l-10}\sum_{l_2\in[l-5,l+5]}2^{-l_1/2}\Sigma'(l_1,l_2,n_k)\big)^2\\
      &\leq C2^{-3l\delta/2}\sum_{n_k\in J_l}\sum_{l_1\leq
        l-10}\sum_{l_2\in[l-5,l+5]}2^{-l_1}\Sigma'(l_1,l_2,n_k)^2,
    \end{split}
  \end{equation}
  where
  \begin{equation*}
    \Sigma'(l_1,l_2,n_k)=\sum_{n_{k_1}\in J_{l_1},n_{k_2}\in J_{l_2},|n_{k_1}+n_{k_2}-n_k|\leq 2^{l/2+10}}\Gamma_{k_1}\Gamma'_{k_2}.
  \end{equation*}
  We observe that for any $n_k\in J_l$
  \begin{equation*}
    \big|\{(n_{k_1},n_{k_2})\in J_{l_1}\times J_{l_2}:|n_{k_1}+n_{k_2}-n_k|\leq 2^{l/2+10}\}\big|\leq C2^{l_1/2}.
  \end{equation*}
  Indeed, for any $n_{k_1}\in J_{l_1}$ there are at most $C$ numbers
  $n_{k_2}\in J_{l_2}$ for which $|n_{k_1}+n_{k_2}-n_k|\leq
  2^{l/2+10}$.  Moreover, we observe that for fixed $k_1,k_2$
  \begin{equation*}
    \big|\{n_{k}\in J_{l}:|n_{k_1}+n_{k_2}-n_k|\leq 2^{l/2+10}\}\big|\leq C.
  \end{equation*}
  Thus
  \begin{equation*}
    \sum_{n_k\in J_l}\Sigma'(l_1,l_2,n_k)^2\leq C2^{l_1/2}\big(\sum_{n_{k_1}\in J_{l_1}}{\Gamma_{k_1}}^2\big)\big(\sum_{n_{k_2}\in J_{l_2}}{\Gamma'_{k_2}}^2\big),
  \end{equation*}
  which shows that the left-hand side of \eqref{ro1} is dominated by
  \begin{equation}\label{ro11}
    C2^{-3l\delta/2}\big(\sum_{l_1=0}^\infty2^{-l_1/2}\sum_{n_{k_1}\in J_{l_1}}{\Gamma_{k_1}}^2\big)\big(\sum_{l_2\in[l-5,l+5]}\sum_{n_{k_2}\in J_{l_2}}{\Gamma'_{k_2}}^2\big).
  \end{equation}

  Using now Lemma \ref{Lemmak2} we estimate
  \begin{equation}\label{ro2}
    \begin{split}
      &2^{2l}\sum_{n_k\in J_l}\big(\sum_{(k_1,k_2)\in Q''_k}
      \|P_k(w_{k_1,k_2,k}\cdot w'_{k_2,k_1,k})\|_{\widetilde{N}_k}\big)^2\\
      &\leq C2^{-l}\sum_{n_k\in J_l}\big(\sum_{l_1,l_2\geq l+5,|l_1-l_2|\leq 2}2^{-\delta l_1}\Sigma''(l_1,l_2,n_k)\big)^2\\
      &\leq C2^{-(1+\delta)l}\sum_{n_k\in J_l}\sum_{l_1,l_2\geq
        l+5,|l_1-l_2|\leq 2}2^{-\delta l_1}\Sigma''(l_1,l_2,n_k)^2,
    \end{split}
  \end{equation}
  where
  \begin{equation*}
    \Sigma''(l_1,l_2,n_k)=\sum_{n_{k_1}\in J_{l_1},n_{k_2}\in J_{l_2},|n_{k_1}+n_{k_2}-n_k|\leq 2^{l_1 /2+10}}\Gamma_{k_1}\Gamma'_{k_2}.
  \end{equation*}
  The Cauchy-Schwarz inequality implies
  \begin{equation*}
    \Sigma''(l_1,l_2,n_k)^2\leq \big(\sum_{n_{k_1}\in J_{l_1}}{\Gamma_{k_1}}^2\big)\big(\sum_{n_{k_2}\in J_{l_2}}{\Gamma'_{k_2}}^2\big).
  \end{equation*}
  Since $|\{n_k : n_k \in J_l\}|\leq C2^{l/2}$ the left-hand side of
  \eqref{ro2} is dominated by
  \begin{equation}\label{ro12}
    C2^{-(1/2+\delta)l}\big(\sum_{l_1=l+5}^\infty2^{-(\delta/2)l_1}\sum_{n_{k_1}\in J_{l_1}}{\Gamma_{k_1}}^2\big)\big(\sum_{l_2=l+5}^\infty2^{-(\delta/2)l_2}\sum_{n_{k_2}\in J_{l_2}}{\Gamma'_{k_2}}^2\big).
  \end{equation}
  
  Finally, using Lemma \ref{Lemmak3} we estimate
  \begin{equation}\label{ro3}
    \begin{split}
      &2^{2l}\sum_{n_k\in J_l}\big(\sum_{(k_1,k_2)\in Q'''_k}
      \|P_k(w_{k_1,k_2,k}\cdot w'_{k_2,k_1,k})\|_{\widetilde{N}_k}\big)^2\\
      &\leq C2^{-2\delta l}\sum_{n_k\in J_l}\big[\sum_{\genfrac{}{}{0pt}{}{1+|n_{k_1}|,1+|n_{k_2}|\in [2^{l-40},2^{l+40}]}{|n_{k_1}+n_{k_2}-n_k|\leq C2^{l/2+10}}}\Lambda(k_1,k_2,k)\Gamma_{k_1}\Gamma'_{k_2}\big]^2\\
      &\leq C2^{-2\delta l}\big[\sum_{1+|n_{k_1}|\in
        [2^{l-40},2^{l+40}]}{\Gamma_{k_1}}^2\big]\big[\sum_{1+|n_{k_2}|\in
        [2^{l-40},2^{l+40}]}{\Gamma'_{k_2}}^2\big].
    \end{split}
  \end{equation}
  The bound \eqref{gi1} follows from \eqref{ro11}, \eqref{ro12}, and
  \eqref{ro3}.

  We remove now the hypothesis \eqref{end1}. Let $\Gamma(\sigma)$
  denote the right-hand side of \eqref{ne2}. Since
  $\partial_x(e^{ia_{k_1,k_2,k}\Psi})\eta_0(t/4)\in S^2_{100}$, it
  follows from Proposition \ref{Lemmab4} (b) (with $\epsilon=-1$) that
  \begin{equation*}
    \begin{split}
      &\|P_k(\partial_x(e^{ia_{k_1,k_2,k}\Psi})w_{k_1,k_2,k}w'_{k_2,k_1,k})\|_{N_k}\\
      &\leq
      C\sum_{\nu\in\mathbb{Z}}(1+|\nu|)^{-60}\ln(2+|n_{k+\nu}|)\|P_{k+\nu}(w_{k_1,k_2,k}w'_{k_2,k_1,k})\|_{\widetilde{N}_{k+\nu}}
    \end{split}
  \end{equation*}
  for any $k,k_1,k_2\in\mathbb{Z}$. Thus, using \eqref{gi1},
  \begin{equation}\label{ne3}
    \sum_{k\in\mathbb{Z}}(1+|n_k|
    )^{2\sigma+2+\delta/4}\Big(\sum_{k_1,k_2\in\mathbb{Z}}\|P_k(\partial_x(e^{ia_{k_1,k_2,k}\Psi})w_{k_1,k_2,k}w'_{k_2,k_1,k})\|_{N_k}\Big)^2\leq \Gamma(\sigma).
  \end{equation}

  To control the first term in the right-hand of \eqref{ne2} we
  decompose the functions $w_{k_1,k_2,k}$ and $w'_{k_2,k_1,k}$ into
  high and low modulation components according to \eqref{ar1set}
  \begin{equation*}
    \begin{split}
      w_{k_1,k_2,k}&=u_{k_1,k_2,k}+v_{k_1,k_2,k}\\
      &=\mathcal{F}^{-1}( \mathbf{1}_{M^{\mathrm{high}}_{k_1}}(
      \tau-\omega(\xi))\mathcal{F}(w_{k_1,k_2,k}))+\mathcal{F}^{-1}(\mathbf{1}_{M^{\mathrm{low}}_{k_1}}(
      \tau-\omega(\xi))\mathcal{F}(w_{k_1,k_2,k})),
    \end{split}
  \end{equation*}
  and
  \begin{equation*}
    \begin{split}
      w'_{k_2,k_1,k}&=u'_{k_2,k_1,k}+v'_{k_2,k_1,k}\\
      &=\mathcal{F}^{-1}(\mathbf{1}_{M^{\mathrm{high}}_{k_2}}(
      \tau-\omega(\xi))\mathcal{F}(w'_{k_2,k_1,k}))+\mathcal{F}^{-1}(\mathbf{1}_{M^{\mathrm{low}}_{k_2}}(
      \tau-\omega(\xi))\mathcal{F}(w'_{k_2,k_1,k})).
    \end{split}
  \end{equation*}
  It follows from Proposition \ref{Lemmab4} (a) (with $\epsilon=0$)
  that, for any $\nu\in\mathbb{Z}$
  \begin{equation*}
    \sup_{k_2,k\in\mathbb{Z}}\|\mathcal{F}(P_{k_1+\nu}(e^{ia_{k_1,k_2,k}\Psi}\eta_0(t/4)u_{k_1,k_2,k}))\|_{Z_{k_1+\nu}}\leq C(1+|\nu|)^{-50}\ln(2+|n_{k_1}|)\Gamma_{k_1}.
  \end{equation*}
  Thus, using \eqref{gi1} with
  $\widetilde{u}_{k_1+\nu,k_1,k_2,k}=P_{k_1+\nu}(e^{ia_{k_1,k_2,k}\Psi}\eta_0(t/4)u_{k_1,k_2,k})$
  \begin{equation}\label{end2}
    \begin{split}
      &\sum_{k\in\mathbb{Z}}(1+|n_k|
      )^{2\sigma+\delta/4}\Big(\sum_{k_1,k_2\in\mathbb{Z}}
      \|\partial_xP_k(e^{ia_{k_1,k_2,k}\Psi}u_{k_1,k_2,k}w'_{k_2,k_1,k})\|_{N_k}\Big)^2\\
      &\leq\sum_{k\in\mathbb{Z}}(1+|n_k| )^{2\sigma+\delta/4}\Big(\sum_{\nu,k_1,k_2\in\mathbb{Z}}\|\partial_xP_k(\widetilde{u}_{k_1+\nu,k_1,k_2,k}w'_{k_2,k_1,k})\|_{N_k}\Big)^2\\
      &\leq \Gamma(\sigma),
    \end{split}
  \end{equation}
  as desired. Similarly,
  \begin{equation}\label{end3}
    \begin{split}
      \sum_{k\in\mathbb{Z}}(1+|n_k|
      )&^{2\sigma+\delta/4}\Big(\sum_{k_1,k_2\in\mathbb{Z}}
      \|\partial_xP_k(e^{ia_{k_1,k_2,k}\Psi}v_{k_1,k_2,k}u'_{k_2,k_1,k}\eta_0(t/4))\|_{N_k}\Big)^2\leq
      \Gamma(\sigma).
    \end{split}
  \end{equation}
  
  Finally, to control the contribution of
  $v_{k_1,k_2,k}v'_{k_2,k_1,k}$ we make the observation that the
  product of two functions of low modulation has high modulation:
  \begin{equation*}
    \mathcal{F}(P_{k'}(v_{k_1,k_2,k}v'_{k_2,k_1,k}))\in Z_{k'}^{\mathrm{high}}\text{ for any }k'\in\mathbb{Z}.
  \end{equation*}
  This follows from \eqref{om20} (recall that
  $Z_k^{\mathrm{high}}=Z_k$ if $|n_k|^\al\leq 2^{20}$). It follows
  from Proposition \ref{Lemmab4} (a) (with $\epsilon=-1$) that
  \begin{equation*}
    \begin{split}
      &\|\partial_xP_k(e^{ia_{k_1,k_2,k}\Psi}\eta_0(t/4)v_{k_1,k_2,k}v'_{k_2,k_1,k})\|_{N_k}\\
      &\leq \| P_k(\partial_x(e^{ia_{k_1,k_2,k}\Psi})\eta_0(t/4)v_{k_1,k_2,k}v'_{k_2,k_1,k})\|_{N_k}\\
      &+C\sum_{\nu\in\mathbb{Z}}(1+|\nu|)^{-50}\ln(2+|n_{k+\nu}|)
      \|\partial_xP_{k+\nu}(v_{k_1,k_2,k}v'_{k_2,k_1,k})\|_{N_{k+\nu}}.
    \end{split}
  \end{equation*}
  Thus, using \eqref{gi1} and \eqref{ne3}
  \begin{equation}\label{end4}
    \begin{split}
      \sum_{k\in\mathbb{Z}}(1+|n_k|
      )&^{2\sigma+\delta/4}\Big(\sum_{k_1,k_2\in\mathbb{Z}}
      \|\partial_xP_k(e^{ia_{k_1,k_2,k}\Psi}\eta_0(t/4)v_{k_1,k_2,k}v'_{k_2,k_1,k})\|_{N_k}\Big)^2\leq
      \Gamma(\sigma).
    \end{split}
  \end{equation}
  The lemma follows from \eqref{ne3}, \eqref{end2}, \eqref{end3}, and
  \eqref{end4}.
\end{proof}

\section{Commutator estimates}\label{commutatorest}
We prove now several commutator estimates. Recall the definitions
\eqref{ne1}.

\begin{lemma}\label{lm:comm}
  Assume that $R(D)=\partial_x^{\sigma_1} D^{\sigma_2}$ for
  $\sigma_1\in \{0,1\}$ and $1<\sigma_2<2$ or $\sigma_2=0$. Assume
  further that $m,m'\in S^\infty_{150}$,
  $\|m\|_{S^\infty_{150}}+\|m'\|_{S^\infty_{150}}\leq 1$. Then, for
  any $\sigma\in[0,2]$ and $k,\mu\in\mathbb{Z}$
  \begin{equation}\label{eq:comm-a}
    \begin{split}
      &(1+|\mu|)^{40}(1+|n_{k+\mu}|)^{2\sigma}
      \|P_{k+\mu}[mP_kR(D)(m'w)-P_kR(D)(mm'w)]\|^2_{F_{k+\mu}}\\
      \leq{} &
      C\sum_{\nu\in\mathbb{Z}}(1+|\nu|)^{-40}(1+|n_{k+\nu}|)^{2\sigma+2\sigma_1+2\sigma_2-1}\ln^2(2+|n_{k+\nu}|)
      \|P_{k+\nu}w\|^2_{\widetilde F_{k+\nu}},
    \end{split}
  \end{equation}
  and
  \begin{equation}\label{eq:comm-b}
    \begin{split}
      &(1+|\mu|)^{40}(1+|n_{k+\mu}|)^{2\sigma}
      \|P_{k+\mu}[mP_kR(D)(m'w)-P_kR(D)(mm'w)]\|^2_{N_{k+\mu}}\\
      \leq{}&
      C\sum_{\nu\in\mathbb{Z}}(1+|\nu|)^{-40}(1+|n_{k+\nu}|)^{2\sigma+2\sigma_1+2\sigma_2-1}\ln^2(2+|n_{k+\nu}|)
      \|P_{k+\nu}w\|^2_{\widetilde N_{k+\nu}}.
    \end{split}
  \end{equation}
\end{lemma}

\begin{proof}[Proof of Lemma \ref{lm:comm}] We decompose $w=\sum_{\nu
    \in \Z}P_{k+\nu}w$ and define the function
  \[q(\xi)=(i\xi)^{\sigma_1}|\xi|^{\sigma_2}\chi_k(\xi).\] We
  calculate
  \begin{equation*}
    \begin{split}
      &\mathcal{F}[P_{k+\mu}[mP_kR(D)(m'P_{k+\nu}w)-P_kR(D)(mm'P_{k+\nu}w)]](\xi,\tau)\\
      =&C\chi_{k+\mu}(\xi)\int_{\mathbb{R}\times\mathbb{R}}\mathcal{F}(m)(\xi_1,\tau_1)
      \mathcal{F}(m'P_{k+\nu}w)(\xi-\xi_1,\tau-\tau_1)[q(\xi)-q(\xi-\xi_1)]\,d\xi_1 d\tau_1\\
      =&C\int_{\mathbb{R}\times\mathbb{R}}\mathcal{F}(P_{k+\nu}w)(\xi-\xi_2,\tau-\tau_2)
      \cdot K(\xi_2,\tau_2,\xi)\,d\xi_2d\tau_2\\
      =& C\int_{I_{k+\mu}}\mathrm{H}(\xi-\gamma)
      \Big[\int_{\mathbb{R}^2}\mathcal{F}(P_{k+\nu}w)(\xi-\xi_2,\tau-\tau_2)\cdot
      K'(\xi_2,\tau_2,\gamma)\,d\xi_2d\tau_2\Big]d\gamma,
    \end{split}
  \end{equation*}
  where $\mathrm{H}$ denotes the Heaviside-function and
  \begin{multline} K(\xi_2,\tau_2,\xi)=\int_{\mathbb{R}\times\mathbb{R}}\mathcal{F}(m)(\xi_1,\tau_1)\mathcal{F}(m')(\xi_2-\xi_1,\tau_2-\tau_1)\\
    \cdot [q(\xi)-q(\xi-\xi_1)]\chi_{k+\mu}(\xi)\,d\xi_1d\tau_1;
  \end{multline}
  and
  \begin{multline} K'(\xi_2,\tau_2,\gamma)=\int_{\mathbb{R}\times\mathbb{R}}\mathcal{F}(m)(\xi_1,\tau_1)\mathcal{F}(m')(\xi_2-\xi_1,\tau_2-\tau_1)\\
    \cdot \partial_\gamma[(q(\gamma)-q(\gamma-\xi_1))\chi_{k+\mu}(\gamma)]\,d\xi_1d\tau_1.
  \end{multline} {\bf Case 1: $k+\mu\not=0$.} By definition of the
  norms it follows for $\epsilon\in \{0,-1\}$
  \begin{align*}
    &\|(\tau-\omega(\xi)+i)^{\epsilon}\mathcal{F}[P_{k+\mu}[mP_kR(D)(m'P_{k+\nu}w)-P_kR(D)(mm'P_{k+\nu}w)]]\|_{Z_{k+\mu}}\\
    \leq & C \int_{I_{k+\mu}}n(\gamma)d\gamma,
  \end{align*}
  where
  \[
  n(\gamma):=\Big\|(\tau-\omega(\xi)+i)^{\epsilon}
  \int_{\mathbb{R}^2}\mathcal{F}(P_{k+\nu}w)(\xi-\xi_2,\tau-\tau_2)\cdot
  K'(\xi_2,\tau_2,\gamma)\,d\xi_2d\tau_2\Big\|_{Z_{k+\mu}}.
  \]
  For $\gamma\in I_{k+\mu}$ fixed it is easy to see that
  \begin{equation*}
    \mathcal{F}^{-1}(K'(.,.,\gamma))=Cm'\cdot \mathcal{F}^{-1}\big[\mathcal{F}(m)(\xi_1,\tau_1)\cdot\partial_\gamma[(q(\gamma)-q(\gamma-\xi_1))\chi_{k+\mu}(\gamma)]\big] 
  \end{equation*}
  is a restricted admissible factor and
  \begin{equation*}
    \|\mathcal{F}^{-1}(K'(.,.,\gamma))\|_{S^2_{100}}\leq C(1+|\mu| )^{-40}(1+|n_{k+\mu}| )^{\sigma_1+\sigma_2-1}.
  \end{equation*}

  The bounds \eqref{eq:comm-a} and \eqref{eq:comm-b} follow from
  estimate \eqref{ar4.1} in the case $|k+\nu|\geq 1$ and from
  \eqref{ar4.2} in the case $k+\nu=0$, combined with the
  Cauchy-Schwarz inequality.  Recall that the integration in $\gamma$
  is over an interval of length $\approx (1+|n_{k+\mu}|)^{1/2}$.

  {\bf Case 2: $k+\mu=0$ and $|k|\geq 2$.}  We use the following
  decomposition: For any $k'\in\mathbb{Z}$ define
  \begin{equation*}
    m_{k'}=\mathcal{F}^{-1}\big[\widetilde{\eta}_{k'}\mathcal{F}(m)\big]
  \end{equation*}
  and set $m_{\leq k'}=\sum_{k'' \leq k'}m_{k''}$, $m_{> k'}=\sum_{k''
    > k'}m_{k''}$.  If $m$ satisfies \eqref{ar3} we obtain
  \begin{equation*}
    \|m_{>k'}\|_{S^2_{100}}\leq C 2^{-k'}(1+2^{k'})^{-80}.
  \end{equation*}
  We have
  \begin{equation}\label{eq:case2a}
    \begin{split}
      &P_{k+\mu}[mP_kR(D)(m'P_{k+\nu}w)-P_kR(D)(mm'P_{k+\nu}w)]\\
      =&P_{0}[m_{>k'} P_kR(D)(m'P_{k+\nu}w)]
    \end{split}
  \end{equation}
  for $k'=\log_2(|n_k|)-10$.  We apply \eqref{ar4.1}
  \begin{equation}\label{eq:case2a-1}
    \begin{split}
      &\|(\tau-\omega(\xi)+i)^\epsilon
      \mathcal{F}P_{0}[m_{>k'}P_kR(D)(m'P_{k+\nu}w)]](\xi,\tau)\|_{Z_0}
      \\
      \leq {} &C
      (1+|k|)^{-50}(1+|n_k|)^{-80}\|\chi_{k}(\xi)(\tau-\omega(\xi)+i)^\epsilon\mathcal{F}(m'P_{k+\nu}w)(\tau,\xi)\|_{Z_k}.
    \end{split}
  \end{equation}
  If $|\nu|\geq 2$ we repeat the same argument with $m'$:
  \[
  \chi_{k}(\xi)(\tau-\omega(\xi)+i)^\epsilon\mathcal{F}(m'P_{k+\nu}w)(\tau,\xi)=\chi_{k}(\xi)(\tau+\omega(\xi)+i)^\epsilon\mathcal{F}(m_{>k'}'
  P_{k+\nu}w)(\tau,\xi).
  \]
  with $k'=\log_2(1+|n_\nu|)-10$, and we apply \eqref{ar4.1} if
  $k+\nu\not=0$ and \eqref{ar4.2} otherwise.

  If $|\nu|\leq 1$ we can afford to use the crude bound
  \begin{equation}\label{eq:crude}
    \begin{split}
      &\|\chi_{k}(\xi)(\tau-\omega(\xi)+i)^\epsilon\mathcal{F}(m' P_{k+\nu}w)(\tau,\xi)\|_{Z_k}\\
      \leq {}&
      C(1+|n_k|^{\al+1})\|\chi_{k+\nu}(\xi)(\tau-\omega(\xi)+i)^\epsilon\mathcal{F}w(\tau,\xi)\|_{Z_{k+\nu}},
    \end{split}
  \end{equation}
  which is straightforward, compare \eqref{ar4} and its proof for the
  high modulation case. In conjunction with \eqref{eq:case2a-1} this
  finishes the discussion of Case 2.

  {\bf Case 3: $k+\mu=0$ and $|k|\leq 1$.}

  \emph{Subcase 3a:} $|\nu|\geq 3$. Define
  $\nu'=\log_2(1+|n_{\nu}|)-10$. It suffices to consider
  \[
  \|\chi_0(\xi)(\tau-\omega(\xi)+i)^{\epsilon}\mathcal{F}[m
  R(D)P_k[m'_{>\nu'}P_{k+\nu}w]-R(D)P_k[[mm']_{>\nu'}P_{k+\nu}w]\|_{Z_0}.
  \]
  We apply the triangle inequality and obtain the estimate
  \begin{align*}
    &\|\chi_0(\xi)(\tau-\omega(\xi)+i)^{\epsilon}\mathcal{F}[m
    R(D)P_k[m'_{>\nu'}P_{k+\nu}w]]\|_{Z_0} \\\leq& C
    \|\chi_k(\xi)(\tau-\omega(\xi)+i)^{\epsilon}\mathcal{F}[m'_{>\nu'}P_{k+\nu}w]\|_{Z_k}
  \end{align*}
  for the first contribution by applying \eqref{ar4}. Due to
  \[\|m_{>\nu'}\|_{S^2_{100}}\leq C(1+|n_\nu|)^{-80}\] we can now
  apply \eqref{ar4.1} to conclude further
  \begin{align*}
    &\|\chi_k(\xi)(\tau-\omega(\xi)+i)^{\epsilon}\mathcal{F}[m'_{>\nu'}P_{k+\nu}w]\|_{Z_k}\\
    \leq{}& C(1+|\nu|)^{-60}(1+|n_\nu|)^{-60}
    \|\chi_{k+\nu}(\xi)(\tau-\omega(\xi)+i)^{\epsilon}\mathcal{F}w\|_{Z_{k+\nu}},
  \end{align*}
  which is sufficient.  For the second contribution we directly use
  the estimate \eqref{ar4.1} and obtain
  \begin{align*}
    &\|\chi_0(\xi)(\tau-\omega(\xi)+i)^{\epsilon}\mathcal{F}[R(D)P_{k}[[mm']_{\nu'}P_{k+\nu}w]\|_{Z_0}\\
    \leq {} & C(1+|\nu|)^{-60}(1+|n_\nu|)^{-60}
    \|\chi_{k+\nu}(\xi)(\tau-\omega(\xi)+i)^{\epsilon}\mathcal{F}w\|_{Z_{k+\nu}},
  \end{align*}
  because
  \[\|[mm']_{>\nu'}\|_{S^2_{100}}\leq C(1+|n_\nu|)^{-80}.\]

  \emph{Subcase 3b:} $|\nu|\leq 2$. The only issue here is the
  structure low frequency component $Z_0$ of the norms.  We decompose
  $m=m_{\leq -10}+m_{> -10}$ and $m'P_{k+\nu}w=[m'P_{k+\nu}w]_{\leq
    -20}+[m'P_{\nu}w]_{> -20}$.

  \underline{Contribution i):} $m_{\leq -10}$ and
  $[m'P_{k+\nu}w]_{\leq -20}$. In the case where $\sigma_1=\sigma_2=0$
  we have
  \[m_{\leq -10}R(D)P_k[m'P_{k+\nu}w]_{\leq -20}-R(D)P_k m_{\leq
    -10}[m'P_{k+\nu}w]_{\leq -20}=0,\] and if $\sigma_1=1$,
  $\sigma_2=0$ we obtain
  \begin{align*}&m_{\leq -10}R(D)P_k[m'P_{k+\nu}w]_{\leq
      -20}-R(D)P_km_{\leq -10}[m'P_{k+\nu}w]_{\leq
      -20}\\=&-(\partial_xm_{\leq -10})P_{k}[m'P_{k+\nu}w]_{\leq -20},
  \end{align*} hence we can assume $k=0$ and the presence of $P_k$ is
  redundant. In this case, we decompose \[m'=m'_{\leq -30}+m'_{>-30},
  \quad P_{k+\nu}w=P_{k+\nu}w_{\leq -30}+P_{k+\nu}w_{>-30}.\] For the
  first contribution ($m'_{\leq -30}$ and $P_{k+\nu}w_{\leq -30}$) we
  obtain the bound
  \[
  \|(\partial_xm_{\leq -10})m'_{\leq -30}[P_{k+\nu}w]_{\leq
    -30}\|_{Z_0}\leq C \|P_{k+\nu}w\|_{\widetilde{Z}_{k+\nu}}
  \]
  by using $\|\partial_xm_{\leq -10} \cdot m'_{\leq
    -30}\|_{S^2_{100}}\leq 1$ and \eqref{ar4.2}.  For the second
  contribution ($m'_{>-30}$ and $P_{k+\nu}w$) we obtain the bound
  \[
  \|(\partial_xm_{\leq -10})P_k[m'_{>-30}P_{k+\nu}w]_{\leq
    -20}\|_{Z_0}\leq C \|P_{k+\nu}w\|_{\widetilde{Z}_{k+\nu}}
  \]
  where we successively use \eqref{ar4.1} or \eqref{ar4.2} as well
  as \[\|\partial_xm_{\leq -10}\|_{S^2_{100}}\leq 1 , \quad \|m'_{>
    -30}\|_{S^2_{100}}\leq 1.\]

  Concerning the third contribution ($m'_{\leq -30}$ and
  $[P_{k+\nu}w]_{>-30}$) we successively apply \eqref{ar4.2} and
  \eqref{ar4} and we observe that $\|[P_{k+\nu}w]_{>-30}\|_{Z_0}\leq C
  \|P_{k+\nu}w\|_{\widetilde{Z}_0}$.

  If $\sigma_1+\sigma_2>1$ we apply the triangle inequality and use
  \eqref{ar4} for the first term
  \begin{align*}
    & \|\chi_0(\xi)(\tau-\omega(\xi)+i)^{\epsilon}\mathcal{F}[m_{\leq -10} R(D)P_k[m'P_{k+\nu}w]_{\leq -20}]\|_{Z_0}\\
    \leq{} & C
    \|\chi_0(\xi)(\tau-\omega(\xi)+i)^{\epsilon}\mathcal{F}[R(D)[m'P_{k+\nu}w]_{\leq
      -20}\|_{Z_0}\\
    \leq{} & C \sum_{j=0}^\infty 2^{j \epsilon}2^{(1-\delta)j}
    \|\eta_j(\tau)\mathcal{F}[m'P_{k+\nu}w]\|_{L^2_{\tau,\xi}},
  \end{align*}
  We decompose in modulation and use Plancherel (similarly to Step 1
  in the proof of Proposition \ref{Lemmab4}) to obtain the estimate
  \begin{equation}\label{eq:l2-bound}
    \begin{split}
      &\sum_{j=0}^\infty 2^{j \epsilon}2^{(1-\delta)j} \|\eta_j(\tau)\mathcal{F}[m'P_{k+\nu}w]\|_{L^2_{\tau,\xi}}\\
      \leq &C \sum_{j_1=0}^\infty 2^{j_1 \epsilon}2^{(1-\delta)j_1}
      \|\chi_{k+\nu}(\xi)\eta_{j_1}(\tau)\mathcal{F}w(\tau,\xi)\|_{L^2_{\tau,\xi}},
    \end{split}
  \end{equation}
  where we exploit that $m'\in S^\infty_{100}$, which is
  sufficient. Concerning the second term
  \begin{align*}
    & \|\chi_0(\xi)(\tau-\omega(\xi)+i)^{\epsilon}\mathcal{F}[R(D)P_k m_{\leq -10}[m'P_{k+\nu}w]_{\leq -20}]\|_{Z_0}\\
    \leq{} &C\sum_{j=0}^\infty2^{\epsilon j}2^{(1-\delta)j}
    \|\eta_j(\tau)\chi_0(\xi)\mathcal{F}m_{\leq -10}
    [m'P_{k+\nu}w]_{\leq -20}\|_{L^2} \\\leq{} &C \sum_{j_1=0}^\infty
    2^{j_1 \epsilon}2^{(1-\delta)j_1}
    \|\chi_{k+\nu}(\xi)\eta_{j_1}(\tau)\mathcal{F}w(\tau,\xi)\|_{L^2_{\tau,\xi}},
  \end{align*}
  as in \eqref{eq:l2-bound}, using $m_{\leq -10} , m'\in
  S^\infty_{100}$.

  \underline{Contribution ii):} $m_{> -10}$ and $[m'P_{k+\nu}w]_{\leq
    -20}$. We apply the triangle inequality. Note that
  $\|m_{>-10}\|_{S^2_{100}}\leq 1$ and there is only a contribution
  from the first term if $k=0$. Note that for $|\xi|\leq 2^{-20}$ the
  term vanishes.  We obtain the bound
  \begin{align*}
    & \|\chi_0(\xi)(\tau-\omega(\xi)+i)^{\epsilon}\mathcal{F}[m_{> -10} R(D)P_k[m'P_{k+\nu}w]_{\leq -20}]\|_{Z_0}\\
    \leq{} &C \sum_{j=0}^\infty 2^{j \epsilon}2^{(1-\delta)j}
    \|\chi_0(\xi)\eta_j (\tau)\mathcal{F}[m_{> -10}
    R(D)[m'P_{k+\nu}w]_{\leq -20}\|_{L^2_{\tau,\xi}} \\\leq{} &C
    \sum_{j_1=0}^\infty 2^{\epsilon j_1 } 2^{(1-\delta)j_1}
    \|\chi_{k+\nu}(\xi)\eta_{j_1}
    (\tau)\mathcal{F}w(\tau,\xi)\|_{L^2_{\tau,\xi}},
  \end{align*}
  by applying \eqref{eq:l2-bound} twice.  The second term can be
  treated similarly.

  \underline{Contribution iii):} $m$ and $[m'P_{k+\nu}w]_{>
    -20}$. Again, we apply the triangle inequality. For the first term
  we apply \eqref{ar4} and use the definition of the spaces
  \begin{align*}
    & \|\chi_0(\xi)(\tau-\omega(\xi)+i)^{\epsilon}\mathcal{F}[m R(D)P_k[m'P_{k+\nu}w]_{> -20}]\|_{Z_0}\\
    \leq{} &C
    \|\chi_k(\xi)(\tau-\omega(\xi)+i)^{\epsilon}\mathcal{F}[R(D)m'P_{k+\nu}w]_{>
      -20}\|_{Z_k} \\\leq{} &C \sum_{j=0}^\infty 2^{j
      \epsilon}2^{(1-\delta)j}
    \|\chi_k(\xi)\eta_j(\tau)\mathcal{F}[m'P_{k+\nu}w]\|_{L^2_{\tau,
        \xi}},
  \end{align*}
  and apply \eqref{eq:l2-bound}. Concerning the second term we apply
  \eqref{ar4} to obtain
  \begin{align*}
    & \|\chi_0(\xi)(\tau-\omega(\xi)+i)^{\epsilon}\mathcal{F}[P_kR(D)[m[m'P_{k+\nu}w]_{> -20}]]\|_{Z_0}\\
    \leq{} &C \sum_{k'\in \Z}(1+|k'|)^{-20}\|\chi_{k'}(\xi)(\tau-\omega(\xi)+i)^{\epsilon}\mathcal{F}[m' P_{k+\nu}w]_{> -20}]\|_{Z_{k'}}\\
    \leq{} &C \sum_{j=0}^{\infty}2^{j
      \epsilon}2^{(1-\delta)j}\|\eta_j(\tau)\mathcal{F}[m'
    P_{k+\nu}w]\|_{L^{2}_{\tau,\xi}}
  \end{align*}
  The claim follows from \eqref{eq:l2-bound}.
\end{proof}

Additionally, we will need a higher order commutator estimates.  Let
us define
\begin{align*}
  &\big [D^\al\partial_x;m'\big ]_{(3)}w\\:=&D^\al\partial_x (m' w)-
  m' D^\al\partial_x w -(\al+1)\partial_x(m') D^\al w +\tfrac{\al
    (\al+1)}{2}\partial_x^2(m')D^{\al-2}\partial_x w.
\end{align*}
\begin{lemma}\label{lm:comm-ext}
  Let $\sigma\in[0,2]$. Assume that $R(D)=\partial_x D^{\al}$ for
  $1<\al<2$. Assume further that $m,m'\in S^\infty_{201}$,
  $\|m\|_{S^\infty_{201}}+\|m'\|_{S^\infty_{201}}\leq 1$. Then, for
  any $k,\mu\in\mathbb{Z}$, $k \not=0$,
  \begin{equation}\label{eq:comm-ext}
    \begin{split}
      & (1+|\mu|)^{40}(1+|n_{k+\mu}|)^{2\sigma}
      \|P_{k+\mu}[mP_k \big [D^\al\partial_x;m'\big ]_{(3)}w]\|^2_{N_{k+\mu}}\\
      \leq &C
      (\|\partial_xm\|_{S^2_{200}}^2+\|\partial_xm'\|_{S^2_{200}}^2)
      \\&\cdot \sum_{\nu \in \Z}
      (1+|\nu|)^{-40}(1+|n_{k+\nu}|)^{2\sigma+2\al-3}\ln^2(2+|n_{k+\nu}|)
      \|P_{k+\nu}w\|^2_{\widetilde N_{k+\nu}}.
    \end{split}
  \end{equation}
\end{lemma}
\begin{proof}[Proof of Lemma \ref{lm:comm-ext}]
  We decompose $w=\sum_{\nu \in \Z} w_{k,\nu}$ where
  $w_{k,\nu}=P_{k+\nu}w$.

  {\bf Case 1: $1\leq|k|\leq 10$.} We apply \eqref{ar4} in order to
  obtain
  \[
  (1+|\mu|)^{40}\|P_{k+\mu}[mP_k \big [D^\al\partial_x;m'\big
  ]_{(3)}w_{k,\nu}]\|_{N_{k+\mu}}\leq C \|P_k \big
  [D^\al\partial_x;m'\big ]_{(3)}w_{k,\nu}\|_{N_k},
  \]
  Further, since $k \not=0$ and
  \[||\xi|^\al \xi-|\xi-\xi_1|^\al(\xi-\xi_1)|\leq C
  |\xi_1|(|\xi|^\al+|\xi-\xi_1|^\al)\] we have
  \begin{align*}
    &\|P_k [D^\al\partial_x (m' w_{k,\nu})-m' D^\al\partial_x w_{k,\nu}]\|_{N_k}\\
    \leq &C (1+|n_{k+\nu}|)^\al \|\chi_k(\xi)(\tau-\omega(\xi)+i)^\epsilon |\mathcal{F}\partial_x m'|\ast |\mathcal{F}w_{k,\nu}|\|_{Z_{k}}\\
    \leq &C
    (1+|\nu|)^{-40}\|\partial_xm'\|_{S^2_{200}}\|w_{k,\nu}\|_{\widetilde{N}_{k+\nu}}
  \end{align*}
  where in the last step we have used \eqref{ar4.2} in case $k+\nu=0$
  and \eqref{ar4.1} otherwise. For the other two terms we have
  \[
  \|\partial_x(m') D^\al w_{k,\nu}\|_{N_k}
  +\|\partial_x^2(m')D^{\al-2}\partial_x w_{k,\nu}\|_{N_k}\leq C
  (1+\nu)^{-40}\|\partial_xm'\|_{S^2_{200}}\|w_{k,\nu}\|_{\widetilde{N}_{k+\nu}},
  \]
  by \eqref{ar4.2} in case $k+\nu=0$ and \eqref{ar4.1} otherwise.

  {\bf Case 2: $k+\mu=0$ and $|k|>10$.} In this case we may replace
  $m$ by $m_{\geq 0}$ and use \eqref{ar4.1} to obtain the upper bound
  \begin{align*}
    &\|P_{k+\mu}[mP_k \big [D^\al\partial_x;m'\big
    ]_{(3)}w_{k,\nu}\|^2_{N_{k+\mu}} \\
    \leq &C (1+|k|)^{-90}\|m_{\geq 0}\|_{S^2_{150}}^2\|P_k \big
    [D^\al\partial_x;m'\big ]_{(3)}w_{k,\nu}\|^2_{N_{k}},
  \end{align*}
  and observe that $\|m_{\geq 0}\|_{S^2_{150}}\leq C \|\partial_x
  m\|_{S^2_{200}}$.  We apply the triangle inequality and bound each
  term individually, using Proposition \ref{Lemmab4}.

  {\bf Case 3: $k+\mu \not=0$, $k+\nu=0$ and $|k|>10$.}  In this case
  we may replace $m'$ by $m'_{\geq 0}$. We use the crude bound
  (similar to \eqref{eq:crude})
  \begin{align*}
    &(1+|\mu|)^{40}\|P_{k+\mu} [mP_k \big[D^\al\partial_x;m'\big]_{(3)} w_{k,\nu}\|^2_{N_{k+\mu}}\\
    \leq &C (1+|n_k|)^{\al+1}\|P_k \big[D^\al\partial_x ; m'_{\geq
      0}\big]_{(3)}w_{k,\nu}\|^2_{N_{k}}.
  \end{align*}
  We apply the triangle inequality and use $\|m'_{\geq
    0}\|_{S^2_{150}}\leq C \|\partial_x m'\|_{S^2_{200}}$ and
  \eqref{ar4.2} to bound each of the four terms individually. We
  obtain
  \[
  \|P_k \big[D^\al\partial_x ; m'_{\geq
    0}\big]_{(3)}w_{k,\nu}\|^2_{N_{k}}\leq C \|\partial_x
  m'\|^2_{S^2_{200}}(1+|\nu|)^{-80}\|w_{k,\nu}\|^2_{\widetilde{N}_{k+\nu}}.
  \]

  {\bf Case 4: $k+\mu \not=0$ and $k+\nu\not=0$ and $|k|>10$.}  For
  the smoothed out (at $\xi=0$) symbol
  $q(\xi)=i\xi|\xi|^{\al}(1-\eta_0)(2^{10}\xi)$ we calculate
  \begin{align*}
    &q(\xi-\xi_1)-q(\xi-\xi_2)-q'(\xi-\xi_2)(\xi_2-\xi_1)
    -\tfrac{1}{2} q''(\xi-\xi_2)(\xi_2-\xi_1)^2\\
    =&(\xi_2-\xi_1)^3 I(\xi-\xi_2,\xi_2-\xi_1),
  \end{align*}
  where \[ I(\xi-\xi_2,\xi_2-\xi_1):= \int_0^1
  q'''(\xi-\xi_2+s(\xi_2-\xi_1)) \frac{(1-s)^2}{2}ds.
  \]
  We obtain
  \begin{align*}
    \mathcal{F}[P_{k+\mu}[mP_k\big [D^\al\partial_x;m'\big
    ]_{(3)}w_{k,\nu}](\xi,\tau) =&C
    \chi_{k+\mu}(\xi)\\\int_{I_{k+\nu}}
    \Big[\int_{\mathbb{R}^2}\mathbf{1}_{I_{k+\nu}}\mathcal{F}(
    w)(\xi-\xi_2,\tau-\tau_2) \mathrm{H}(\xi-\xi_2-\gamma)\cdot &
    K(\xi_2,\tau_2,\gamma)\,d\xi_2d\tau_2\Big]d\gamma,
  \end{align*}
  where $\mathrm{H}$ denotes the Heaviside-function and
  \begin{align*} K(\xi_2,\tau_2,\gamma)
    =\int_{|\xi_1|\geq |\mu|-2^{20}}&\mathcal{F}(m)(\xi_1,\tau_1)\mathcal{F}(\partial_x^3m')(\xi_2-\xi_1,\tau_2-\tau_1)\\
    &\cdot \partial_\gamma
    [\chi_k(\gamma+\xi_2-\xi_1)\chi_{k+\nu}(\gamma)
    I(\gamma,\xi_2-\xi_1)]\,d\xi_1d\tau_1.
  \end{align*}
  For fixed $\gamma \in I_{k+\nu}$ the function
  $\mathcal{F}^{-1}K(\cdot , \cdot, \gamma)$ is a restricted
  admissible factor satisfying
  \begin{equation*}
    \|\mathcal{F}^{-1}(K'(\cdot ,\cdot,\gamma))\|_{S^2_{100}}\leq C\|\partial_xm'\|_{S^2_{200}}(1+|n_{k}| )^{\al-2}(1+|\nu|)^{-40}(1+|\mu|)^{-40}.
  \end{equation*}
  We have
  \begin{align*}
    &\|P_{k+\mu}[\big  [D^\al\partial_x;m' \big  ]_{(3)} ]\|_{N_{k+\mu}}\\
    \leq&C\int_{I_{k+\nu}}\left\|\frac{\chi_{k+\mu}(\xi)}{\tau-\omega(\xi)+i}
      \int_{\mathbb{R}^2}\mathbf{1}_{I_{k+\nu}}\mathcal{F}(w)(\xi-\xi_2,\tau-\tau_2)\cdot
      K(\xi_2,\tau_2,\gamma)\,d\xi_2d\tau_2\right\|_{Z_{k+\mu}}d\gamma.
  \end{align*}
  Finally, we apply \eqref{ar4.1} to the integrand for fixed $\gamma$
  and use the fact $|I_{k+\nu}|\leq C |n_{k+\nu}|^{\frac12}$.
\end{proof}

Moreover, we will need a more specific commutator type estimate which
makes use of the bilinear estimates from Section \ref{Dyadic2}. Let us
define an extension of the low frequency part of the initial data
$\widetilde{\phi_{\mathrm{low}}}(x,t):=\eta_0(t/4)\phi_{\mathrm{low}}(x)$. Recall
that $\|\widetilde{\phi_{\mathrm{low}}}\|_{L^2}\leq C \ep_0$.

\begin{lemma}\label{lm:comm-h}
  Assume that $m,m'\in S^\infty_{201}$,
  $\|m\|_{S^\infty_{201}}+\|m'\|_{S^\infty_{201}}\leq 1$. Then, for
  any $\sigma\in[0,2]$ and $k\in\mathbb{Z}\setminus\{0\}$, $\mu \in
  \Z$,
  \begin{equation}\label{eq:comm-h}
    \begin{split}
      & (1+|\mu|)^{40}(1+|n_{k+\mu}|)^{2\sigma}
      \|P_{k+\mu}[m [P_k  (\widetilde{\phi_{\mathrm{low}}}\partial_x (m'w)- \widetilde{\phi_{\mathrm{low}}} P_k \partial_x (m'w)]]\|^2_{N_{k+\mu}}\\
      &\leq C \ep_0^2 \cdot \sum_{\nu \in \Z}
      (1+|n_{k+\nu}|)^{2\sigma}\|P_{k+\nu}w\|^2_{F_{k+\nu}}.
    \end{split}
  \end{equation}
\end{lemma}
\begin{proof}[Proof of Lemma \ref{lm:comm-h}]
  We decompose $m'w=\sum_{\nu' \in \Z}P_{k+\nu'}(m'w)$ and there are
  nontrivial contributions only if $|\nu'|\leq 5$.  It suffices to
  consider the case where $k+\mu\not =0$ and $|k|>10$ because
  otherwise the estimate follows from \eqref{ar4.1} and
  \eqref{ar4.2}. We replace $m$ with $m_{\geq k'}$ for
  $k'=\log_2(1+|n_\mu|)-10$ in case $|\mu|\geq 10$.  We compute the
  Fourier transform
  \begin{align*}
    & \mathcal{F}[P_{k+\mu}[m[P_k (\widetilde{\phi_{\mathrm{low}}}\partial_x P_{k+\nu'}(m'w))- \widetilde{\phi_{\mathrm{low}}} P_k \partial_x P_{k+\nu'}(m'w)]](\xi,\tau)\\
    =&C \chi_{k+\mu}(\xi)\int_{I_{k+\nu'}}I(\xi,\tau,\gamma)d\gamma,
  \end{align*}
  where $\mathrm{H}$ denotes the Heaviside-function,
  $I(\xi,\tau,\gamma)$ is defined as
  \[
  I(\xi,\tau,\gamma):=\int_{\mathbb{R}^2}\mathcal{F}\partial_x
  (m'w)(\xi-\xi_2,\tau-\tau_2) \mathrm{H}(\xi-\xi_2-\gamma)
  \mathcal{F}M(\xi_2,\tau_2,\gamma)\,d\xi_2d\tau_2
  \]
  and
  \begin{multline*} \mathcal{F}M(\xi_2,\tau_2,\gamma) =\int
    \mathcal{F}(m)(\xi_1,\tau_1)\mathcal{F}(\widetilde{\phi_{\mathrm{low}}})(\xi_2-\xi_1,\tau_2-\tau_1)
    \\ \cdot\partial_\gamma
    [(\chi_k(\gamma+\xi_2-\xi_1)-\chi_k(\gamma))\chi_{k+\nu'}(\gamma)]\,d\xi_1d\tau_1.
  \end{multline*}
  It follows
  \begin{align*}
    & \sum_{|\nu'|\leq 5}\|\mathcal{F}[P_{k+\mu}[m[P_k (\widetilde{\phi_{\mathrm{low}}}\partial_x P_{k+\nu'}(m'w))- \widetilde{\phi_{\mathrm{low}}} P_k \partial_x P_{k+\nu'}(m'w)]](\xi,\tau)\|_{Z_{k+\mu}}\\
    \leq &C \sum_{|\nu'|\leq 5}\int_{I_{k+\nu'}}\Big\|
    \chi_{k+\mu}(\xi) \mathcal{F}[\partial_x (m'w)
    M(\cdot,\cdot,\gamma)](\tau,\xi)\Big\|_{Z_{k+\mu} }d\gamma\\
    \leq &C \sup_{|\nu'|\leq 5}\Big[ \int_{I_{k+\nu'}}\Big\|
    \chi_{k+\mu}(\xi) \mathcal{F}[\partial_x (w)
    m'M(\cdot,\cdot,\gamma)](\tau,\xi)\Big\|_{Z_{k+\mu} }d\gamma \\
    &+ \int_{I_{k+\nu'}}\Big\| \chi_{k+\mu}(\xi) \mathcal{F}[w
    \partial_x (m')M(\cdot,\cdot,\gamma)](\tau,\xi)\Big\|_{Z_{k+\mu}
    }d\gamma\Big]
  \end{align*}
  Let $M'_\gamma$ denote either $m'M(\cdot,\cdot,\gamma)$ or
  $\partial_x (m')M(\cdot,\cdot,\gamma)$.  For $\gamma \in I_{k+\nu'}$
  and $|\nu'|\leq 5$ one can show that
  \[
  \|\chi_{k_1} \mathcal{F}M'_\gamma\|_{Z_{k_1}}\leq C \ep_0
  (1+|n_k|)^{-1} (1+|k_1|)^{-60}(1+|\mu|)^{-60}
  \]
  if $k_1\not=0$, and
  \[
  \|\chi_{0} \mathcal{F}M'_\gamma\|_{X_0^\delta}\leq C \ep_0
  (1+|n_k|)^{-1}(1+|\mu|)^{-60}.
  \]
  We decompose $w=\sum_{\nu \in \Z} P_{k+\nu}w$ and apply Lemmas
  \ref{Lemmak1}-\ref{Lemmak3} and obtain for $\mu \in \Z$, $\mu
  +k\not=0 $ and $\sigma'\in \{0,1\}$:
  \begin{align*}
    &\Big\| \chi_{k+\mu}(\xi) \mathcal{F}[P_{k+\nu}
    (\partial_x^{\sigma'} w)
    M'_\gamma](\tau,\xi)\Big\|_{Z_{k+\mu}} \\
    \leq & C \ep_0 (1+|\mu-\nu|)^{-60}(1+|\mu|)^{-60}(1+|n_{k
      +\nu}|)^{1/2-\delta}(1+|n_k|)^{-1}\|P_{k+\nu}w\|_{F_{k+\nu}}
  \end{align*}
  Note that $|I_{k+\nu'}|\leq C |n_{k+\nu'}|^\frac12 \approx
  |n_k|^\frac12$ for $|\nu'|\leq 5$.  The claim follows by summing up
  with respect to $\nu$ and Cauchy-Schwarz.
\end{proof}

 \section{Proof of Proposition \ref{Lemmat2}}\label{lastsection}
 The properties in Part (a) are standard, cf. \cite[Lemma 4.2]{IoKeTa}
 and its proof.  We will only show the a priori estimate
 \eqref{toshow1} because the estimate \eqref{toshow2} for differences
 is very similar (recall that $\phi_{\rm low}=\phi_{\rm low}'$).

 We need to estimate the following expressions, see \eqref{rh5} and
 \eqref{rh21},
 \begin{equation}\label{eq:r0}
   R_0=-P_0\partial_x (\phi_{\mathrm{low}}\cdot v)-P_0\partial_x(v^2/2)-D^\al\partial_xP_0(\phi_{\mathrm{low}})-P_0\partial_x(\phi_{\mathrm{low}}^2/2),
 \end{equation}
 and for $k \in \Z\setminus \{0\}$
 \begin{equation*}
   R_k=R_k^{(1)}+R_k^{(2)}+R_k^{(3)}+R_k^{(4)}+R_k^{(5)},
 \end{equation*}
 where
 \begin{align}
   R_k^{(1)}:=&-e^{-ia_k\Psi}P_k\partial_x(v^2/2)\label{eq:r1}\\
   R_k^{(2)}:=  &-\phi_{\mathrm{low}}[\partial_xv_k-D^\alpha v_k\cdot (i n_k|n_k|^{-\alpha})]]\label{eq:r2}\\
   R_k^{(3)}:=    &-[e^{-ia_k\Psi}D^\alpha\partial_x(e^{ia_k\Psi}v_k)-D^\al\partial_x(v_k)-(\alpha+1)D^\alpha v_k\cdot (ia_k\Psi')]\label{eq:r3}\\
   R_k^{(4)}:=    &-e^{-ia_k\Psi}[P_k(\phi_{\mathrm{low}}\cdot \partial_xv)-\phi_{\mathrm{low}}\cdot \partial_x(P_kv)]\label{eq:r4}\\
   R_k^{(5)}:= &-[ia_k\phi^2_{\mathrm{low}}\cdot
   v_k+e^{-ia_k\Psi}P_k(v\cdot\partial_x\phi_{\mathrm{low}})].\label{eq:r5}
 \end{align}

 We fix extensions $\widetilde{\widetilde{v}}_k$ of the functions
 $v_k$ such that $\|\widetilde{\widetilde{v}}_k\|_{\F^{\sigma'}}\leq
 C\|v_k\|_{\F^{\sigma'}(T')}$, $\sigma'\in\{0,\sigma\}$, and $\supp
 \widetilde{\widetilde{v}}_k \subset \mathbb{R}_x \times [-4,4]$. For
 any interval $[a,b]\subseteq\mathbb{R}$ let
 \begin{equation*}
   P_{[a,b]}=\sum_{k\in \mathbb{Z}\cap [a,b]} P_k.
 \end{equation*}
 By \eqref{rh50} and the commutator estimate \eqref{eq:comm-a} the
 function
 \[\widetilde{v}_k=e^{-ia_k\Psi}P_{[k-1,k+1]}(e^{ia_k\Psi}\widetilde{\widetilde{v}}_k)\]
 is another extension of $v_k$ with the properties, supported in
 $\mathbb{R}_x \times [-4,4]$ and verifying
 \begin{equation}\label{bh10}
   \begin{split}
     &\|\widetilde{v}_k\|_{\F^{\sigma'}}\leq
     C\|v_k\|_{\F^{\sigma'}(T')},\qquad\sigma=\{0,\sigma'\},\\
     &\widetilde{v}_k=e^{-ia_k\Psi}P_{[k-2,k+2]}(e^{ia_{k}\Psi}\widetilde{v}_{k}).
   \end{split}
 \end{equation}
 We define
 \begin{equation}\label{eq:v-ext}
   \widetilde{v}=\sum_{k\in\mathbb{Z}}e^{ia_k\Psi}\widetilde{v}_k,
 \end{equation}

 We look at each of the contributions \eqref{eq:r0}-\eqref{eq:r5}
 separately.

 \underline{Contribution of \eqref{eq:r0}:} Recall the definition
 $\widetilde{\phi_{\mathrm{low}}}(x,t)=\phi_{\mathrm{low}}(x)\eta_0(t/4)$. We
 define the extension
 \[
 \widetilde{R}_0= -P_0\partial_x (\widetilde{\phi_{\mathrm{low}}}\cdot
 \widetilde{v})-P_0\partial_x(\widetilde{v}^2/2)
 -D^\al\partial_xP_0(\widetilde{\phi_{\mathrm{low}}})-P_0\partial_x(\widetilde{\phi_{\mathrm{low}}}^2/2),
 \]
 Obviously, it holds
 \begin{align*}
   &\|P_0[D^\al\partial_x\widetilde{\phi_{\mathrm{low}}}+\partial_x\widetilde{\phi_{\mathrm{low}}}^2/2]\|_{\N^\sigma}\\
   &\leq C\sum_{k'=-\infty}^2 2^{-k'}\big[\|\widetilde{\eta}_{k'}\mathcal{F}[D^\al\partial_x\widetilde{\phi_{\mathrm{low}}}]\|_{L^2_{\xi,\tau}}+\|\widetilde{\eta}_{k'}\mathcal{F}[P_0\partial_x\widetilde{\phi_{\mathrm{low}}}^2]\|_{L^2_{\xi,\tau}}\big]\\
   &\leq
   C(\|\phi_{\mathrm{low}}\|_{L^2}+\|\phi_{\mathrm{low}}\|^2_{L^2}).
 \end{align*}
 To estimate the contribution from the first two terms, we estimate
 \begin{align*}
   &\|P_0\partial_x (\widetilde{\phi_{\mathrm{low}}}\cdot \widetilde{v})+P_0\partial_x(\widetilde{v}^2/2) \|_{\N^\sigma}\\
   &\leq C\sum_{k'=-\infty}^2 \big[\|\widetilde{\eta}_{k'}\mathcal{F}[\widetilde{\phi_{\mathrm{low}}}\cdot \widetilde{v}]\|_{L^2_{\xi,\tau}}+\|\widetilde{\eta}_{k'}\mathcal{F}[\widetilde{v}^2]\|_{L^2_{\xi,\tau}}\big]\\
   &\leq C\ep_0 \|\widetilde{v}\|_{L^\infty_t L^2_x}+C\|\widetilde{v}\|_{L^\infty_t L^2_x}^2\\
   &\leq C \ep_0 \big[\sum_{k \in \Z}
   \|\widetilde{v}_k\|^2_{L^\infty_t L^2_x}\big]^{1/2}+C \sum_{k \in
     \Z}\|\widetilde{v}_k\|_{L^\infty_t L^2_x}^2.
 \end{align*}
 Using \eqref{hh80}, the two estimates above imply
 \begin{equation}\label{endo0}
   \|\widetilde{R}_0\|^2_{\N^\sigma}\leq C\|\phi\|_{H^0}^2+C \sum_{k \in \Z}\|\widetilde{v}_k\|^2_{\F^0}\big(\ep_0^2+\sum_{k \in \Z}\|\widetilde{v}_k\|^2_{\F^0}\big).
 \end{equation}

 \underline{Contribution of \eqref{eq:r1}:} We define
 \begin{equation*}
   \widetilde{R}_k^{(1)}=-e^{-ia_k\Psi}P_k\partial_x(\widetilde{v}^2/2).
 \end{equation*}
 This is an extension of $R_k^{(1)}$. An application of the commutator
 estimate \eqref{eq:comm-b} yields
 \begin{equation}\label{bh20}
   \begin{split}
     \sum_{k\in\mathbb{Z}\setminus\{0\}}\|\widetilde{R}_k^{(1)}\|^2_{\N^\sigma}
     \leq &C \sum_{k\in\mathbb{Z}}(1+|n_k|)^{2\sigma}
     \|\partial_x P_k(e^{-ia_{k}\Psi}(\widetilde{v}^2)\|_{N_k}^2\\
     +&C \sum_{k\in\mathbb{Z}}(1+|n_k|)^{2\sigma}\|\partial_x
     P_k[\widetilde{v}^2]\|_{N_k}^2
   \end{split}
 \end{equation}
 For $k',\nu'\in\mathbb{Z}$ we define
 $w_{k',\nu'}=P_{k'}\widetilde{v}_{k'+\nu'}$, so
 \begin{equation*}
   \widetilde{v}=\sum_{k',\nu'\in\mathbb{Z}}e^{ia_{k'+\nu'}\Psi}w_{k',\nu'}.
 \end{equation*}
 Using Lemma \ref{Lemmat1} (we ignore the $\delta/4$ gains) and this
 identity, the right-hand side of \eqref{bh20} is dominated by
 \begin{align*}
   C\sum_{\nu_1,\nu_2\in\mathbb{Z}}(1+|\nu_1|)^2(1+|\nu_2|
   )^2\Big(\sum_{k_1\in\mathbb{Z}}
   \|w_{k_1,\nu_1}\|_{F_{k_1}}^2\Big)\Big(\sum_{k_2\in\mathbb{Z}}(1+|n_{k_2}|)^{2\sigma}\|w_{k_2,\nu_2}\|_{F_{k_2}}^2\Big)
 \end{align*}
 For $|\nu|\leq 10$ fixed and $\sigma'\in\{0,\sigma\}$ we estimate
 simply
 \begin{equation*}
   \sum_{k\in\mathbb{Z}}(1+|n_{k}|)^{2\sigma'}\|w_{k,\nu}\|_{F_{k}}^2
   =\sum_{k\in\mathbb{Z}}(1+|n_{k-\nu}|)^{2\sigma'}\|P_{k-\nu}\widetilde{v}_k\|_{F_{k-\nu}}^2\leq \sum_{k\in\mathbb{Z}}\|\widetilde{v}_k\|_{\F^{\sigma'}}^2.
 \end{equation*}
 For $|\nu|\geq 11$ and $\sigma'\in\{0,\sigma\}$ we estimate, using
 \eqref{bh10} and Lemma \ref{lm:comm}
 \begin{align*}
   &\sum_{k\in\mathbb{Z}}(1+|n_{k}|)^{2\sigma'}\|w_{k,\nu}\|_{F_{k}}^2
   \\=&\sum_{k\in\mathbb{Z}}(1+|n_{k-\nu}|)^{2\sigma'}
   \|P_{k-\nu}[e^{-ia_k\Psi}P_{[k-2,k+2]}(e^{ia_{k}\Psi}\widetilde{v}_{k})]\|_{F_{k-\nu}}^2\\\leq&
   C(1+|\nu|)^{-40}\sum_{k\in\mathbb{Z}}\|\widetilde{v}_k\|_{\F^{\sigma'}}^2
 \end{align*}
 Therefore
 \begin{equation}\label{endo1}
   \sum_{k\in\mathbb{Z}\setminus\{0\}}\|\widetilde{R}_k^{(1)}\|^2_{\N^\sigma}\leq C\big(\sum_{k\in\mathbb{Z}}\|\widetilde{v}_k\|_{\F^0}^2\big)\big(\sum_{k\in\mathbb{Z}}\|\widetilde{v}_k\|_{\F^\sigma}^2\big).
 \end{equation}

 \underline{Contribution of \eqref{eq:r2}:} We define the extension
 \begin{equation*}
   \widetilde{R}_k^{(2)}:=-\widetilde{\phi_{\mathrm{low}}}[\partial_x\widetilde{v}_k
   -D^{\alpha}\widetilde{v}_k\cdot (i n_k|n_k|^{-\alpha})]
 \end{equation*}
 where we set
 $\widetilde{\phi_{\mathrm{low}}}(x,t):=\eta_0(t/4)\phi_{\mathrm{low}}(x)$.
 We note that for small $\delta>0$
 \[\|\widetilde{\phi_{\mathrm{low}}}\|_{S^{2}_{150}}+\|\mathcal{F}\widetilde{\phi_{\mathrm{low}}}\|_{X^\delta_0}\leq
 C \ep_0.\] We define
 \begin{equation*}
   \widetilde{u}_k:=D_k\widetilde{v}_k \text{ where } D_k:=\partial_x- (i n_k|n_k|^{-\alpha}) \cdot D^{\alpha}
 \end{equation*}
 and we decompose
 \[
 \widetilde{u}_k=\sum_{\nu \in \Z} \widetilde{u}_{k,\nu} \text{ where
 } \widetilde{u}_{k,\nu}=P_\nu \widetilde{u}_{k}.
 \]
 Now, by definition
 \begin{equation*}
   \sum_{k\in \Z \setminus\{0\}}
   \|\widetilde{R}_k^{(2)}\|^2_{\N^{\sigma}} \leq \sum_{k\in \Z \setminus\{0\}}
   \sum_{k_1 \in \Z}(1+|n_{k_1}|)^{2\sigma}
   \left(\sum_{|\nu -k_1|\leq 5} \|P_{k_1}[\widetilde{\phi_{\mathrm{low}}}\widetilde{u}_{k,\nu}]\|_{N_{k_1}}\right)^2.
 \end{equation*}
 If $|k_1|\leq 10$ the estimates \eqref{ar4.1} and \eqref{ar4.2} imply
 that
 \begin{equation*}
   \sum_{|\nu-k_1|\leq 5}\|P_{k_1}[\widetilde{\phi_{\mathrm{low}}}\widetilde{u}_{k,\nu}]\|_{N_{k_1}}\leq C \ep_0 \sum_{|\nu|\leq 15}\|\widetilde{u}_{k,\nu}\|_{F_\nu}\leq C \ep_0\|\widetilde{v}_k\|_{\F^\sigma}
 \end{equation*}
 If $|k_1|>10$ we use estimates \eqref{hj1} and \eqref{hj1.2} to
 obtain
 \begin{equation*}(1+|n_{k_1}|)^{\sigma}\sum_{|\nu-k_1|\leq
     5}\|P_{k_1}[\widetilde{\phi_{\mathrm{low}}}\widetilde{u}_{k,\nu}]\|_{N_{k_1}}\leq
   C (1+|n_{k_1}|)^{\sigma-1/2-\delta} \ep_0 \sum_{|\nu-k_1|\leq
     5}\|\widetilde{u}_{k,\nu}\|_{F_\nu}.
 \end{equation*}
 The symbol of $P_\nu D_k$
 \begin{equation*}
   m(\xi):=\chi_{\nu}(\xi)\left(i\xi-\frac{|\xi|^{\alpha}}{|n_k|^{\alpha}}in_k\right)
 \end{equation*}
 satisfies $|m(\xi)|\leq C
 \chi_{\nu}(\xi)(1+|k-\nu|)^{2\alpha}(1+|n_{\nu}|)^{\frac12}$ by
 definition of the sequence $n_k$, see \eqref{rh30}.  Therefore we
 conclude that
 \begin{equation*}
   \sum_{k \in \Z \setminus\{0\}}\|\widetilde{R}_k^{(2)}\|^2_{\N^{\sigma}} \leq C\ep_0^2 \sum_{k \in \Z} \sum_{k_1 \in \Z}(1+|n_{k_1}|)^{2\sigma}(1+|k-k_1|)^{10}\|P_{k_1}\widetilde{v}_{k}\|^2_{F_{k_1}}.
 \end{equation*}
 and we use \eqref{bh10} and the commutator estimate \eqref{eq:comm-a}
 to obtain
 \begin{equation}\label{endo2}
   \sum_{k\in \Z \setminus\{0\}} \|\widetilde{R}_k^{(2)}\|^2_{\N^{\sigma}}\leq C\ep_0^2
   \sum_{k\in\Z}\|\widetilde{v}_k\|^2_{\F^{\sigma}}.
 \end{equation}

 \underline{Contribution of \eqref{eq:r3}:} As above, we define an
 extension
 \[
 \widetilde{R}_k^{(3)}:=
 -[e^{-ia_k\Psi}D^\alpha\partial_x(e^{ia_k\Psi}\widetilde{v}_k)-D^\al\partial_x(\widetilde{v}_k)-(\alpha+1)D^\alpha
 \widetilde{v}_k\cdot (ia_k\Psi')]
 \]
 for $k \not =0$. We use the property \eqref{bh10} and apply the
 commutator estimate \eqref{eq:comm-ext} to obtain
 \begin{align*}
   \|\widetilde{R}_k^{(3)}\|_{\N^\sigma}^2
   \leq & C \ep_0^2 |a_k|^2 \sum_{\nu \in \Z}(1+|n_{k+\nu}|)^{2\sigma+2\al-5/2}(1+|\nu|)^{-40}\|P_{k+\nu}\widetilde{v}_k\|_{\widetilde{N}_{k+\nu}}^2\\
   &+C \|e^{-ia_k \Psi}\partial_x^2 (e^{ia_k \Psi})
   D^{\al-2}\partial_x \widetilde{v}_k\|_{\N^\sigma}^2.
 \end{align*}
 Since $|a_k|= |n_k|^{1-\al}$, the first term is bounded by
 $\ep_0^2\|\widetilde{v}_k\|_{\N^\sigma}$.  Concerning the second
 term, we note that the restriction of $e^{-ia_k \Psi}\partial_x^2
 (e^{ia_k \Psi})$ to the time interval $[-4,4]$ is a restricted
 admissible factor with norm less than $C\ep_0 |a_k|$ and estimate
 \eqref{ar4.1} yields
 \[
 \|P_{[k-2,k+2]}[e^{-ia_k \Psi}\partial_x^2 (e^{ia_k \Psi})
 D^{\al-2}\partial_x \widetilde{v}_k]\|_{\N^\sigma}\leq C \ep_0
 \|\widetilde{v}_k\|_{\N^\sigma}.
 \]
 if $|k|\leq 5$, and for $|k|>5$ Lemma \ref{Lemmat1} implies that
 \[
 \sum_{|\mu|\leq 2}(1+|n_{k+\mu}|)^{2\sigma}\|P_{k+\mu}[e^{-ia_k
   \Psi}\partial_x^2 (e^{ia_k \Psi}) D^{\al-2}\partial_x
 \widetilde{v}_k]\|^2_{N_{k+\mu}}\leq C \ep_0
 \|\widetilde{v}_k\|_{\F^\sigma}.
 \]
 In conclusion, we obtain
 \begin{equation}\label{endo3}
   \sum_{k \in \Z \setminus \{0\}}\|\widetilde{R}_k^{(3)}\|_{\N^\sigma}^2\leq C \ep_0^2 \sum_{k \in \Z}\|\widetilde{v}_k\|_{\F^\sigma}^2.
 \end{equation}

 \underline{Contribution of \eqref{eq:r4}:} As above, we define the
 extension
 \[
 \widetilde{R}_k^{(4)}:=
 -e^{-ia_k\Psi}[P_k(\widetilde{\phi_{\mathrm{low}}}\cdot \partial_x\widetilde{v})-\phi_{\mathrm{low}}\cdot \partial_x(P_k\widetilde{v})],
 \]
 see \eqref{eq:v-ext}. Using \eqref{bh10} and Lemma \ref{lm:comm-h} we
 obtain
 \begin{align*}
   \|\widetilde{R}_k^{(4)}\|_{\N^\sigma}\leq & \sum_{\genfrac{}{}{0pt}{}{k_1\in \Z}{|k_1-k|\leq 5}}\|e^{-ia_k\Psi}[P_{k}(\widetilde{\phi_{\mathrm{low}}} \partial_x(e^{ia_{k_1}\Psi} \widetilde{v}_{k_1}))-\phi_{\mathrm{low}}\partial_x(P_{k}(e^{ia_{k_1}\Psi} \widetilde{v}_{k_1}))]\|_{\N^\sigma}\\
   \leq &C \ep_0 \sum_{\genfrac{}{}{0pt}{}{k_1\in \Z}{|k_1-k|\leq
       5}}\|\widetilde{v}_{k_1}\|_{\F^\sigma}.
 \end{align*}
 Summing up with respect to $k$ yields
 \begin{equation}\label{endo4}
   \sum_{k \in \Z\setminus \{0\}}\|\widetilde{R}_k^{(4)}\|^2_{\N^\sigma}\leq C \ep_0^2 \sum_{k\in \Z}\|\widetilde{v}_{k}\|^2_{\F^\sigma}.
 \end{equation}

 \underline{Contribution of \eqref{eq:r5}:} We define the extension
 \[
 \widetilde{R}_k^{(5)}:= -[ia_k\widetilde{\phi_{\mathrm{low}}}^2\cdot
 \widetilde{v}_k+e^{-ia_k\Psi}P_k(\widetilde{v}\cdot\partial_x\widetilde{\phi_{\mathrm{low}}})].
 \]
 Estimates \eqref{ar4.1} and \eqref{ar4.2} imply that
 \[
 \|ia_k \widetilde{\phi_{\mathrm{low}}}^2\cdot
 \widetilde{v}_k\|_{\N^\sigma}\leq C \ep^2_0
 \|\widetilde{v}_k\|_{\N^\sigma}.
 \]
 Concerning the second term we use \eqref{bh10} and \eqref{eq:comm-b}
 and obtain
 \begin{align*}
   &\|e^{-ia_k\Psi}P_k(\widetilde{v}\cdot\partial_x\widetilde{\phi_{\mathrm{low}}})\|_{\N^\sigma}
   \\\leq&
   \sum_{\genfrac{}{}{0pt}{}{k_1\in \Z}{|k_1-k|\leq 5}} \|e^{-ia_k\Psi}P_k(\widetilde{v}_{k_1} \cdot e^{ia_{k_1}\Psi}\partial_x\widetilde{\phi_{\mathrm{low}}})\|_{\N^\sigma}\\
   \leq{} &C \sum_{\genfrac{}{}{0pt}{}{k_1\in \Z}{|k_1-k|\leq 5}}
   \sup_{a_{k,k_1}\in [-4,4]}\|\widetilde{v}_{k_1} \cdot
   e^{ia_{k,k_1}\Psi}\partial_x\widetilde{\phi_{\mathrm{low}}}\|_{\N^\sigma}
 \end{align*}
 It follows from \eqref{ar4.1} and \eqref{ar4.2} that
 \[
 \sup_{a_{k,k_1}\in [-4,4]}\|P_{0}[\widetilde{v}_{k_1} \cdot
 e^{ia_{k,k_1}\Psi}\partial_x\widetilde{\phi_{\mathrm{low}}}]\|_{\N^\sigma}\leq
 C \ep_0 \|\widetilde{v}_{k_1}\|_{\N^\sigma}.
 \]
 Moreover, we have the trivial bound
 \begin{align*}
   &\sup_{a_{k,k_1}\in [-4,4]}\|(I-P_0) [\widetilde{v}_{k_1} \cdot e^{ia_{k,k_1}\Psi}\partial_x\widetilde{\phi_{\mathrm{low}}}]\|_{\N^\sigma}\\
   \leq{} & C \sup_{a_{k,k_1}\in [-4,4]}\|\widetilde{v}_{k_1} \cdot
   e^{ia_{k,k_1}\Psi}\partial_x\widetilde{\phi_{\mathrm{low}}}\|_{L^\infty_t
     H^\sigma_x} \leq C \ep_0 \|\widetilde{v}_{k_1}\|_{\F^\sigma}.
 \end{align*}
 By summing up with respect to $k$ we conclude
 \begin{equation}\label{endo5}
   \sum_{k\in \Z \setminus \{0\}} \|\widetilde{R}_k^{(5)}\|^2_{\N^\sigma}\leq C \ep_0^2\sum_{k \in \Z}\|\widetilde{v}_k\|^2_{\F^\sigma}.
 \end{equation}

 In summary, the estimate \eqref{toshow1} now follows from
 \eqref{endo1}, \eqref{endo2}, \eqref{endo3}, \eqref{endo4} and
 \eqref{endo5} and \eqref{endo0}.

\end{document}